\documentclass{elsarticle}
\usepackage[english]{babel}
\usepackage[utf8]{inputenc}
\usepackage[T1]{fontenc}
\usepackage[usenames, dvipsnames]{xcolor}
\usepackage{tikz}
\usepackage{amsthm}
\usepackage{amsmath}
\usepackage{latexsym}
\usepackage{float}
\usepackage{stmaryrd}
\usepackage{fullpage}
\usepackage{hyperref}
\usepackage{enumerate}

\everymath{\displaystyle}

\newtheorem{theorem}{Theorem}
\newtheorem{obs}[theorem]{Observation}
\newtheorem{lemma}[theorem]{Lemma}

\newtheorem{defi}[theorem]{Definition}
\newtheorem{prop}[theorem]{Proposition}
\newtheorem{conj}[theorem]{Conjecture}

\newtheorem{rem}[theorem]{Remark}

\newcommand{\outcomeP}{\mathcal{P}}
\newcommand{\outcomeN}{\mathcal{N}}
\newcommand{\grundy}{\mathcal{G}}
\newcommand{\nimsum}{\oplus}
\newcommand{\opt}{\mathrm{opt}}
\DeclareMathOperator{\mex}{mex}

\newcommand{\conestar}{\mathcal C_1^*}

\newcommand{\ctwostar}{\mathcal C_2^*}
\newcommand{\ctwobox}{\mathcal C_2^{\Box}}
\newcommand{\cthreebox}{\mathcal C_3^{\Box}}
\newcommand{\dzerostar}{\mathcal D_0^*}
\newcommand{\donestar}{\mathcal D_1^*}
\newcommand{\donebox}{\mathcal D_1^{\Box}}
\newcommand{\dtwostar}{\mathcal D_2^*}
\newcommand{\dtwobox}{\mathcal D_2^{\Box}}
\newcommand{\dthreebox}{\mathcal D_3^{\Box}}

\newcommand{\bipa}{\tikz[baseline=-4]{\draw (0,0) node[circle,fill=black,minimum size=0,inner sep=1]{} -- (0.3,0) node[circle,fill=black, minimum size=0, inner sep=1] {};}}
\newcommand{\joint}[1]{\tikz[baseline=-4]{\draw (0,0) node[circle,fill=black,minimum size=0,inner sep=1]{} -- (0.3,0) node[circle,fill=black, minimum size=0, inner sep=1] {} node[midway,above,scale=0.5] {$#1$};}}
\newcommand{\tripa}{\joint{2}}

\newcommand{\sstar}[1]{$S_{#1}$}
\newcommand{\sbstar}[3]{#1 \joint{#2} #3}

\tikzstyle{noeud}=[circle, fill=black, inner sep= 0, minimum size = 4]
\tikzstyle{gros_noeud}=[circle, fill=black, inner sep= 0, minimum size = 8]

\title{Octal Games on Graphs: \\ The game 0.33 on subdivided stars and bistars \tnoteref{gag}}
\tnotetext[gag]{This work has been supported by the ANR-14-CE25-0006 project of the French National Research Agency.}
\date{}
\author[limos]{Laurent Beaudou}
\author[laasp]{Pierre Coupechoux}
\author[liris]{Antoine Dailly\corref{cor}}
\author[ijf]{Sylvain Gravier}
\author[laasj,mam]{Julien Moncel}
\author[liris]{Aline Parreau}
\author[labri]{Éric Sopena}

\cortext[cor]{Corresponding author}
\address[limos]{LIMOS, 1 rue de la Chebarde, 63178 Aubière CEDEX, France.}
\address[laasj]{LAAS-CNRS, Université de Toulouse, CNRS, Université Toulouse 1 Capitole - IUT Rodez, Toulouse, France}
\address[mam]{Fédération de Recherche Maths à Modeler, Institut Fourier, 100 rue des Maths, BP 74, 38402 Saint-Martin d'Hères Cedex, France}
\address[laasp]{LAAS-CNRS, Université de Toulouse, CNRS, INSA, Toulouse, France.}
\address[liris]{Univ Lyon, Université Lyon 1, LIRIS UMR CNRS 5205, F-69621, Lyon, France.}
\address[ijf]{CNRS/Université Grenoble-Alpes, Institut Fourier/SFR Maths à Modeler, 100 rue des Maths - BP 74, 38402 Saint Martin d'Hères, France.}
\address[labri]{Univ. Bordeaux, Bordeaux INP, CNRS, LaBRI, UMR5800, F-33400 Talence, France}

\begin{document}

\begin{frontmatter}

\begin{abstract}
\emph{Octal games} are a well-defined family of two-player games played on heaps of counters, in which the players remove alternately a certain number of counters from a heap, sometimes being allowed to split a heap into two nonempty heaps, until no counter can be removed anymore.

We extend the definition of octal games to play them on graphs: heaps are replaced by connected components and counters by vertices. Thus, playing an octal game on a path $P_n$ is equivalent to playing the same octal game on a heap of $n$ counters.

We study one of the simplest octal games, called 0.33, in which the players can remove one vertex or two adjacent vertices without disconnecting the graph. We study this game on trees and give a complete resolution of this game on subdivided stars and bistars.
\end{abstract}

\begin{keyword}
	Combinatorial Games;
	Octal Games;
	Subtraction Games;
	Graphs
\end{keyword}

\end{frontmatter}

\section{Introduction}

\emph{Combinatorial games} are finite two-player games without chance, with perfect information and such that the last move alone determines which player wins the game.
Since the information is perfect and the game finite, there is always a winning strategy for one of the players. A formal definition of combinatorial games and basic results will be given in Section \ref{sec:def}. For more details, the interested reader can refer to \cite{winningways}, \cite{lip} or \cite{cgt}.

A well-known family of combinatorial games is the family of \emph{subtraction games}, which are played on a heap of counters. A subtraction game is defined by a list of positive integers $L$ and is denoted by $Sub(L)$. A player is allowed to remove $k$ counters from the heap if and only if $k \in L$. The first player unable to play loses the game. For example, consider the game $Sub(\{1,2\})$. In this game, both players take turns removing one or two counters from the heap, until the heap is empty. If the initial number of counters is a multiple of 3, then the second player has a winning strategy: by playing in such a way that the first player always gets a multiple of 3, he will take the last counter and win the game.

A natural generalization of subtraction games is to allow the players to split a heap into two nonempty heaps after having removed counters. This defines a much larger class of games, called \emph{octal games} \cite{winningways}. An octal game is represented by an octal code which entirely defines its rules. As an example, $Sub(\{1,2\})$ is defined as {\bf 0.33}. A precise definition will be given in Section \ref{sec:def}.
Octal games have been extensively studied. One of the most important questions \cite{Guy96} is the periodicity of these games. Indeed, it seems that all finite octal games have a periodic behaviour in the following sense: the set of initial numbers of counters for which the first player has a winning strategy is ultimately periodic. This is true for all subtraction games and for all finite octal games for which the study has been completed \cite{althofer,winningways}.


Octal games can also be played by placing counters in a row. Heaps are constituted by consecutive counters and only consecutive counters can be removed. According to this representation, it seems natural to play octal games on more complex structures like graphs. A position of the game is a graph and players remove vertices that induce a connected component which corresponds to consecutive counters.
The idea to extend the notion of octal games to graphs was already suggested in \cite{fleischer}. However, to our knowledge, this idea has not been further developed.
With our definition, playing the generalization of an octal game on a path is the same as playing  the original octal game.
In the special case of subtraction games, players have to keep the graph connected. As an example, playing {\bf 0.33} on a graph consists in removing one vertex or two adjacent vertices from the graph without disconnecting it.

This extension of octal games is in line with several take-away games on graphs such as {\sc Arc Kayles} \cite{S78} and {\sc Grim} \cite{adams}. However, it does not describe some other deletion games, such as the vertex and edge versions of the game \textsc{geography} \cite{S78,edgegeo}, vertex and edge deletion games with parity rules, considered in \cite{ottaway1} and \cite{ottaway2}, or scoring deletion games such as Le Pic ar\^ete \cite{picarete}. 


We will first give in Section \ref{sec:def} basic definitions from combinatorial game theory as well as a formal definition of octal games on graphs. We then focus on the game {\bf 0.33} which is one of the simplest octal games, and to its study on trees. 
We first study subdivided stars in Section \ref{sec:star}.
We prove that paths can be reduced modulo 3 which leads to a complete resolution, in contrast with the related studies on subdivided stars of {\sc Node Kayles} \cite{fleischer} and {\sc Arc Kayles} \cite{H15}. In Section \ref{sec:bistar}, we extend our results to subdivided bistars (i.e. trees with at most two vertices of degree at least 3) using a game operator similar to the sum of games. Unfortunately, these results cannot be extended to all trees and not even to caterpillars. In a forthcoming paper \cite{futurpapier}, some of our results are generalized to other subtraction games on subdivided stars.

\section{Definitions}\label{sec:def}

\subsection{Basics of Combinatorial Game Theory}
\emph{Combinatorial games} \cite{winningways} are two-player games such that:
\begin{enumerate}
	\item The two players play alternately.
	\item There is no chance.
	\item The game is finite (there are finitely many positions and no position can be encountered twice during the game).
	\item The information is perfect.
	\item The last move alone determines the winner.
\end{enumerate}

In \emph{normal} play, the player who plays the last move wins the game. In \emph{misère} play, the player who plays the last move loses the game. \emph{Impartial games} are combinatorial games where at each turn the moves are the same for both players. Hence the only distinction between the players is who plays the first move. In this paper, we will only consider impartial games in normal play.

Positions in impartial games have exactly two possible {\em outcomes}: either the first player has a winning strategy, or the second player has a winning strategy. If a game position falls into the first category, it is an \emph{$\outcomeN$-position} (for $\outcomeN$ext player wins); otherwise, it is a \emph{$\outcomeP$-position} (for $\outcomeP$revious player wins).

From a given position $J$ of the game, the different positions that can be reached by playing a move from $J$ are the \emph{options} of $J$, and the set of options of $J$ is denoted $\opt(J)$. If we know the outcomes of the positions in $\opt(J)$ we can deduce the outcome of $J$, using the following proposition:

\begin{prop}\label{prop:outcome} Let $J$ be a position of an impartial combinatorial game in normal play:
  \begin{itemize}
  \item If $\opt(J)=\emptyset$, then $J$ is a $\outcomeP$-position.
  \item If there exists a $\outcomeP$-position $J'$ in $\opt(J)$, then $J$  is an $\outcomeN$-position: a winning move consists in playing from $J$ to $J'$.
  \item If all the options of $J$ are $\outcomeN$-positions, then $J$ is a $\outcomeP$-position.
\end{itemize}
\end{prop}

Every position $J$ of a combinatorial game can be viewed as a combinatorial game with $J$ as the initial position. We therefore often consider positions as games.
Some games can be described as the union of smaller game positions. In order to study them, we define the concept of the sum of games.
Given two games $J_1$ and $J_2$, their \emph{disjoint sum}, denoted by $J_1+J_2$, is defined as the game where, at their turn, each player plays a legal move on either $J_1$ or $J_2$.  Once $J_1$ (resp. $J_2$) is finished, the two players play exclusively on $J_2$ (resp. $J_1$), until it is over. The player who plays the last move wins the game.

The question is now whether we can determine the outcome of a disjoint sum $J_1+J_2$ as a function of the outcomes of $J_1$ and $J_2$. If $J_1$ is a $\outcomeP$-position, then $J_1+J_2$ has the same outcome as $J_2$: the winning player of $J_2$ applies his strategy on $J_2$, and if the other player plays on $J_1$ then he applies the winning strategy on $J_1$. However, the disjoint sum of two $\outcomeN$-positions cannot be determined so easily.
In order to study the disjoint sum of two $\outcomeN$-positions, the \emph{equivalence} of two games $J_1$ and $J_2$ is defined as follows: $J_1 \equiv J_2$ if and only if $J_1+J_2$ is a $\outcomeP$-position. According to this relation, one can attribute to a game a value corresponding to its equivalence class, called the {\em Grundy value}. The Grundy value of a game position $P$ for a game $J$, denoted by $\grundy_J(P)$, can be computed from the Grundy value of its options using the following formula:

$$\grundy_J(P) = \mex(\grundy_J(P') | P' \in \opt(P))$$
where, for any set of integers $S$, $\mex(S)$ is the smallest nonnegative integer not in $S$. In particular, $P$ is a $\outcomeP$-position if and only if $\grundy_J(P)=0$. Note that this is consistent with Proposition~\ref{prop:outcome}.
When the context is clear, we will denote $\grundy_J(P)$ as $\grundy(P)$.

A fundamental result of Combinatorial Game Theory is the Sprague-Grundy Theorem that gives the Grundy values of the sum of games:

\begin{theorem}[Sprague-Grundy Theorem \cite{Spra36}]\label{thm:grundysum}
 Let $J_1$ and $J_2$ be two game positions. Then $\grundy(J_1+J_2)=\grundy(J_1)\oplus\grundy(J_2)$, where $\oplus$, called the nim-sum, is the bitwise XOR applied to the two values written in base~2.
\end{theorem}

A direct application of this theorem is that for any game position $J$, we have $\grundy(J+J)=0$. Moreover, we can see that two games $J_1$ and $J_2$ have the same Grundy value if and only if their disjoint sum $J_1+J_2$ is a $\outcomeP$-position.
 
\subsection{Octal games}
A well-known family of impartial games is the family of \emph{octal games}, which are played on heaps of counters. On their turn, each player removes some counters from one heap and may also divide the remaining counters of the heap into two nonempty heaps. The rules of an octal game are encoded according to an octal number as follows:

\begin{defi}[Octal games \cite{winningways}]
	\label{def:octalGames}
{\rm	Let $u_1,u_2,\ldots,u_n,\ldots$ be nonnegative integers such that for all $i$, $u_i \leq 7$.
	In the octal game ${\bf 0.u_1u_2...u_n...}$, a player can remove $i$ counters from a heap if and only if $u_i \neq 0$.
	
	Moreover, if we write $u_i$ as $u_i = b^i_1 + 2 \cdot b^i_2 + 4 \cdot b^i_3$ with $b^i_j \in \{0,1\}$, then, the player can, when removing $i$ counters from a heap:
	\begin{enumerate}
		\item empty the heap if and only if $b^i_1=1$;
		\item leave the heap nonempty if and only if $b^i_2=1$;
		\item split the remaining heap in two nonempty heaps if and only if $b^i_3=1$.
	\end{enumerate}}
\end{defi}

An octal game is {\em finite} if it has a finite number of non-zero values. In this case, we stop the code at the last non-zero $u_i$.

 For example, ${\bf u_i}=3$ means that a player can remove $i$ counters from a heap without splitting it. Octal games with only ${\bf 0}$ and ${\bf 3}$ in their code correspond to {\em subtraction games} since the heap is never divided. In particular, the game {\bf 0.33} is the game where one can remove one or two counters from a heap. A value of ${\bf u_i=7}$ means that one can remove $i$ counters from a heap, possibly dividing the heap in two heaps whereas ${\bf u_i=6}$ means that one can remove $i$ counters from a heap except if the heap has exactly $i$ counters, and possibly divide it into two heaps.

To study an octal game, it suffices to consider it on a single heap. Indeed, using Theorem~\ref{thm:grundysum}, one can obtain the Grundy value of any octal game by computing the nim-sum of its components. The \emph{Grundy sequence} of an octal game is the sequence of the Grundy values of the game on a heap of $n$ counters with $n=0,1,2,...$. For example, the Grundy sequence of {\bf 0.33} is $0,1,2,0,1,2,...$ since the Grundy value of the game {\bf 0.33} on a heap of size $n$ is $n \bmod 3$.

The Grundy sequence of {\bf 0.33} is periodic and one can prove that this is the case for all finite subtraction games \cite{winningways}. Actually, all the octal games which have been completely studied have an ultimately periodic Grundy sequence\footnote{For an up-to-date table of octal games, see \url{http://wwwhomes.uni-bielefeld.de/achim/octal.html}}. This led to the following conjecture, proposed by Guy:
\begin{conj}[Guy's conjecture \cite{Guy96}]
	All finite octal games have ultimately periodic Grundy sequences.
\end{conj}


\subsection{Octal games on graphs}

A natural question is whether this periodicity can be extended to more complex structures. A relevant structure is graphs. Indeed, as explained in the introduction, octal games are generally played with counters in a row. Considering a row of counters as a path and replacing the notion of consecutive counters by connected components, we get the following definition of octal games on graphs:

\begin{defi}[Octal game on graphs]
	\label{def:octalGamesOnGraphs}
{\rm Let $u_1,u_2,\ldots,u_n,\ldots$ be nonnegative integers such that for all $i$, $u_i \leq 7$.
 Let $G$ be a graph.
  
 In the octal game ${\bf 0.u_1u_2...u_n...}$ played on $G$, a player can remove a set $X_i$ of $i$ vertices of $G$ if and only if $u_i \neq 0$ and $X_i$ induces a connected graph.

 Moreover, if we write ${ u_i}$ as ${ u_i} = b^i_1 + 2 \cdot b^i_2 + 4 \cdot b^i_3$, with $b^i_j\in\{0,1\}$, and $H$ is the connected component of $G$ containing $X_i$, then:
	\begin{enumerate}
		\item the player can remove $H$ (i.e. $X_i=V(H)$) if and only if $b^i_1=1$;
		\item the player can leave $H$ connected with at least one vertex remaining (i.e $H\setminus \{X_i\}$ is nonempty and connected) if and only if $b^i_2=1$;
	     \item the player can disconnect $H$ if and only if $b^i_3=1$.
	\end{enumerate}}
\end{defi}

If $G$ is a path, then the game is equivalent to the corresponding standard octal game of Definition \ref{def:octalGames}. We now consider several examples. The game ${\bf 0.33}$ on a connected graph corresponds to the game where one can take one vertex or two adjacent vertices without disconnecting it. The game ${\bf 0.07}$ corresponds to the game where one can remove any two adjacent vertices of the graph. That is exactly the well-known game {\sc Arc Kayles} \cite{S78}. Recently, Adams {\it et al.} \cite{adams} studied the game {\sc Grim} that is exactly {\bf 0.6} on some graphs (players are allowed to remove any vertex of the graph, except if it is an isolated vertex). A scoring version of ${\bf 0.6}$ is also currently studied \cite{DGM}. Hence our definition is relevant with existing work. Note that the well-known game {\sc Node Kayles} cannot be seen as such an octal game even though on a path it is equivalent to ${\bf 0.137}$. Indeed, in {\sc Nodes Kayles}, when four vertices can be removed, they cannot induce a $P_4$. This cannot match our definition.

\begin{rem}
{\rm In the definition of octal games, if $b^i_3=1$, then the players can split a nonempty heap in exactly two nonempty heaps. Generalizations of octal games may then be defined with $b^i_j$ for $j \geq 4$ in order to allow the splitting of a nonempty heap into more than two nonempty heaps. However, our extension of octal games on graphs do not make this distinction: if $b^i_3=1$, then the players can disconnect the graph and leave as many components as they like. Thus this move is not a move that leaves a given number of components, but one which breaks the connectivity of a graph. This is still coherent with the definition of octal games on a row of counters since the path graph can only be split in two components, and allows us to include previously defined vertex deletion games, such as \textsc{Arc Kayles} and \textsc{Grim}.}
\end{rem}

\begin{rem} {\rm We ask for the $i$ removed vertices to form a connected component for two reasons. First, in traditional octal games, the counters are generally taken consecutively. The second reason is that if we remove this condition, then all subtraction games on graphs will be trivial. Indeed, it is always possible to remove a vertex of a connected graph and keep the graph connected. Therefore it is also always possible to remove $i$ vertices of the graph without disconnecting it if the vertices do not need to induce a connected graph. Thus playing a subtraction game on a graph would be equivalent to playing the same game on a path with the same number of vertices and we lose the interest of considering more complex structures. With our definition, subtraction games on graphs are not so straightforward.}
\end{rem}


In the rest of this paper, we focus on one octal game, namely ${\bf 0.33}$, for which we provide a detailed analysis on subdivided stars and bistars: by proving lemmas about reducibility of paths, we provide an equivalence between families of stars and bistars which allows us to determine their Grundy value.

\section{A study of the {\bf 0.33} game on subdivided stars}\label{sec:star}

If $n$ is an integer, we define the graph $P_n$ as the path on $n$ vertices, with $n-1$ edges.

A subdivided star is the tree obtained by subdividing each edge of a star $K_{1,k}$ (with $k \geq 0$) as many times as we want. Each of the subdivided edges will be called a path. A subdivided star is denoted by \sstar{\ell_1,\ldots,\ell_k}, where $\ell_i \geq 1$ is the number of vertices of the $i$th path. Figure~\ref{fig:exkpod} shows an example of such a graph.

The standard definition of subdivided stars actually requires $k \geq 3$ and thus excludes the paths, however we will need to consider the paths as base cases for subdivided stars and bistars. This is why we will consider the subdivided star \sstar{\ell_1} (resp. \sstar{\ell_1,\ell_2}) which is isomorphic to $P_{\ell_1+1}$ (resp. to $P_{\ell_1+\ell_2+1}$). Note that the star $K_{1,0}$ is isomorphic to $P_1$. For clarity, the notation as paths will be used whenever applicable.

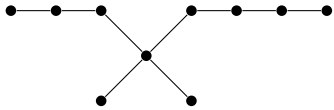
\begin{figure}[H]
	\centering
	\begin{tikzpicture}[scale=0.6]
	\node[noeud] (c) at (1,2) {};
	
	\node[noeud] (2) at (0,1) {};
	\node[noeud] (2b) at (2,1) {};
	\node[noeud] (5) at (0,3) {};
	\node[noeud] (6) at (-1,3) {};
	\node[noeud] (7) at (-2,3) {};
	\node[noeud] (5b) at (2,3) {};
	\node[noeud] (6b) at (3,3) {};
	\node[noeud] (7b) at (4,3) {};
	\node[noeud] (8b) at (5,3) {};
	
	\draw (2) to (c);
	\draw (2b) to (c);
	\draw (c) to (5) to (6) to (7);
	\draw (c) to (5b) to (6b) to (7b) to (8b);
	\end{tikzpicture}
	\caption{The subdivided star \sstar{1,1,3,4}.} 
	\label{fig:exkpod}
\end{figure}

In the {\bf 0.33} game played on a graph, players can remove a vertex or two adjacent vertices from the graph, provided that they do not disconnect the graph. Figure~\ref{fig:ex033star} shows the moves that are available for the first player on a subdivided star. Note that in every figure describing moves, the original position will be boxed.

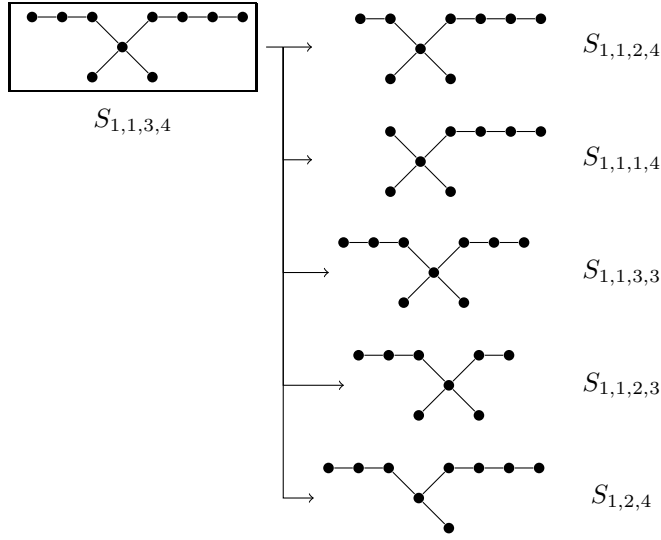
\begin{figure}[!h]
	\centering
	\begin{tikzpicture}
		\node (orig) at (0,0) {\fbox{
			\begin{tikzpicture}[scale=0.4]
			\node[noeud] (c) at (1,2) {};
			
			\node[noeud] (2) at (0,1) {};
			\node[noeud] (2b) at (2,1) {};
			\node[noeud] (5) at (0,3) {};
			\node[noeud] (6) at (-1,3) {};
			\node[noeud] (7) at (-2,3) {};
			\node[noeud] (5b) at (2,3) {};
			\node[noeud] (6b) at (3,3) {};
			\node[noeud] (7b) at (4,3) {};
			\node[noeud] (8b) at (5,3) {};
			
			\draw (2) to (c);
			\draw (2b) to (c);
			\draw (c) to (5) to (6) to (7);
			\draw (c) to (5b) to (6b) to (7b) to (8b);
			\end{tikzpicture}}
		};
		\node (playOnP3-1) at (4,0) {
				\begin{tikzpicture}[scale=0.4]
				\node[noeud] (c) at (1,2) {};
				
				\node[noeud] (2) at (0,1) {};
				\node[noeud] (2b) at (2,1) {};
				\node[noeud] (5) at (0,3) {};
				\node[noeud] (6) at (-1,3) {};
				\node (7) at (-2,3) {};
				\node[noeud] (5b) at (2,3) {};
				\node[noeud] (6b) at (3,3) {};
				\node[noeud] (7b) at (4,3) {};
				\node[noeud] (8b) at (5,3) {};
				
				\draw (2) to (c);
				\draw (2b) to (c);
				\draw (c) to (5) to (6);
				\draw (c) to (5b) to (6b) to (7b) to (8b);
				\end{tikzpicture}
		};
		\node (playOnP3-2) at (4,-1.5) {
			\begin{tikzpicture}[scale=0.4]
			\node[noeud] (c) at (1,2) {};
			
			\node[noeud] (2) at (0,1) {};
			\node[noeud] (2b) at (2,1) {};
			\node[noeud] (5) at (0,3) {};
			\node (6) at (-1,3) {};
			\node (7) at (-2,3) {};
			\node[noeud] (5b) at (2,3) {};
			\node[noeud] (6b) at (3,3) {};
			\node[noeud] (7b) at (4,3) {};
			\node[noeud] (8b) at (5,3) {};
			
			\draw (2) to (c);
			\draw (2b) to (c);
			\draw (c) to (5);
			\draw (c) to (5b) to (6b) to (7b) to (8b);
			\end{tikzpicture}
		};
		\node (playOnP4-1) at (4,-3) {
				\begin{tikzpicture}[scale=0.4]
				\node[noeud] (c) at (1,2) {};
				
				\node[noeud] (2) at (0,1) {};
				\node[noeud] (2b) at (2,1) {};
				\node[noeud] (5) at (0,3) {};
				\node[noeud] (6) at (-1,3) {};
				\node[noeud] (7) at (-2,3) {};
				\node[noeud] (5b) at (2,3) {};
				\node[noeud] (6b) at (3,3) {};
				\node[noeud] (7b) at (4,3) {};
				
				\draw (2) to (c);
				\draw (2b) to (c);
				\draw (c) to (5) to (6) to (7);
				\draw (c) to (5b) to (6b) to (7b);
				\end{tikzpicture}
		};
		\node (playOnP4-2) at (4,-4.5) {
			\begin{tikzpicture}[scale=0.4]
			\node[noeud] (c) at (1,2) {};
			
			\node[noeud] (2) at (0,1) {};
			\node[noeud] (2b) at (2,1) {};
			\node[noeud] (5) at (0,3) {};
			\node[noeud] (6) at (-1,3) {};
			\node[noeud] (7) at (-2,3) {};
			\node[noeud] (5b) at (2,3) {};
			\node[noeud] (6b) at (3,3) {};
			
			\draw (2) to (c);
			\draw (2b) to (c);
			\draw (c) to (5) to (6) to (7);
			\draw (c) to (5b) to (6b);
			\end{tikzpicture}
		};
		\node (playOnP1) at (4,-6) {
				\begin{tikzpicture}[scale=0.4]
				\node[noeud] (c) at (1,2) {};
				
				\node[noeud] (2b) at (2,1) {};
				\node[noeud] (5) at (0,3) {};
				\node[noeud] (6) at (-1,3) {};
				\node[noeud] (7) at (-2,3) {};
				\node[noeud] (5b) at (2,3) {};
				\node[noeud] (6b) at (3,3) {};
				\node[noeud] (7b) at (4,3) {};
				\node[noeud] (8b) at (5,3) {};
				
				\draw (2b) to (c);
				\draw (c) to (5) to (6) to (7);
				\draw (c) to (5b) to (6b) to (7b) to (8b);
				\end{tikzpicture}
		};
		
		\node at (0,-1) {\sstar{1,1,3,4}};
		\node at (6.5,0) {\sstar{1,1,2,4}};
		\node at (6.5,-1.5) {\sstar{1,1,1,4}};
		\node at (6.5,-3) {\sstar{1,1,3,3}};
		\node at (6.5,-4.5) {\sstar{1,1,2,3}};
		\node at (6.5,-6) {\sstar{1,2,4}};
		
		\draw (orig) to (2,0);
		\draw[->] (2,0) -- (playOnP3-1);
		\draw[->] (2,0) |- (playOnP3-2);
		\draw[->] (2,0) |- (playOnP4-1);
		\draw[->] (2,0) |- (playOnP4-2);
		\draw[->] (2,0) |- (playOnP1);
	\end{tikzpicture}
	\caption{The available moves for the first player in the {\bf 0.33} game played on the subdivided star \sstar{1,1,3,4}.}
	\label{fig:ex033star}
\end{figure}

The {\bf 0.33} game on paths and cycles has the same nim-sequence as the {\bf 0.33} game on heaps of counters:
\begin{prop}
	\label{prop:033pathsAndCycles}
	For any $n \geq 0$, $\grundy(P_n) = \grundy(C_n) = n \bmod 3$.
\end{prop}
In this section, we will prove a similar result for subdivided stars: every path of length $\ell$ can be reduced to a path of length $\ell \bmod 3$ without changing the Grundy value.

\begin{theorem}
	\label{thm:modkpodes}
	For all $\ell_1,\ldots,\ell_k$, we have $\grundy($\sstar{\ell_1,\ldots,\ell_k}$)=\grundy($\sstar{\ell_1 \bmod 3,\ldots,\ell_k \bmod 3}$)$.
\end{theorem}

To prove this theorem, it suffices to prove that a $P_3$ can be attached to the central vertex or attached to a leaf of a subdivided star without changing the Grundy value. This will follow from a series of lemmas. First we make an observation that will be useful for several proofs.

\begin{obs}
\label{lem:keylemma}
Let $P_n$ be a path with $n \geq 4$, and $x$ a vertex of $P_n$. Then a move in $P_n$ that removes $x$ has an equivalent move that does not remove $x$: removing the symmetric of $x$ leads to the same position.
\end{obs}

In particular, we will use this observation when $x$ is the central vertex of a star with one or two paths.

\begin{lemma}
\label{lem:grundyStars}
Let $\ell \geq 0$ and $S=$\sstar{1,1,\ell}. We have $\grundy(S)=|V(S)| \bmod 3=\ell \bmod 3$.
\end{lemma}

\begin{proof}
We use induction on $\ell$. First, suppose that one can remove the central vertex of $S$. This is only possible if $\ell=0$, thus $S=P_3$ and we are done.

Now, if $\ell\geq 1$, then one cannot remove the central vertex of $S$. In this case, up to three moves are available from $S$:
\begin{itemize}
	\item Removing one of the two leaves, leaving $P_{\ell+2}$ whose Grundy value is $(\ell+2) \bmod 3$;
	\item Removing one vertex from the path of length $\ell$, leaving a star whose Grundy value is $(\ell+2) \bmod 3$ by induction hypothesis;
	\item If $\ell \geq 2$, removing two vertices from the path of length $\ell$, leaving a star whose Grundy value is $(\ell+1) \bmod 3$ by induction hypothesis.
\end{itemize}
Thus, we have $\grundy(S)=\mex((\ell+1) \bmod 3, (\ell+2) \bmod 3)=\ell \bmod 3$. Note that if $\ell=1$ then all moves are equivalent and leave $P_3$, thus $\grundy(S)=\mex(\grundy(P_3))=\mex(0)=1$.
\end{proof}

\begin{lemma}
\label{lem:modkpodes}
A $P_3$ can be attached to any leaf or to the central vertex of a subdivided star without changing its Grundy value.
\end{lemma}

\begin{proof}
Let $S$ be a subdivided star, and $S'$ be the subdivided star obtained by attaching a $P_3$ to any leaf or to the central vertex of $S$. We show that $S + S'$ is a $\outcomeP$-position by proving that the second player can always play to a $\outcomeP$-position after the first player's move. The proof is by induction on $|V(S)|$.

Suppose first that the first player can remove the central vertex of $S$:
\begin{itemize}
\item If $S$ is empty (resp. $S=P_1$, $S=P_2$), then $S'=P_3$ (resp. $S'=P_4$, $S'=P_5$), and thus $S+S'$ is a $\outcomeP$-position since $\grundy(S)=\grundy(S')$;
\item If $S=P_3$, then either $S'=P_6$ and we are done, or $S'=$\sstar{1,1,3} and the result follows from Lemma~\ref{lem:grundyStars}.
\item If $S=S_{\ell}$ with $\ell \geq 4$ or $S=S_{1,\ell}$ with $\ell \geq 2$, then, as stated in Observation~\ref{lem:keylemma}, the second player will always be able to replicate the first player's move on $S'$, by playing the symmetrical move. By induction hypothesis, the new position will be a $\outcomeP$-position.
\end{itemize}

Suppose now that the first player cannot remove the central vertex of $S$:
\begin{itemize}
\item If the first player takes one vertex (resp. two vertices) from the attached $P_3$ in $S'$, then the second player takes two vertices (resp. one vertex) from it, leaving $S + S$ which is a $\outcomeP$-position.
\item If the first player plays elsewhere on $S'$, the second player answers by playing the same move on $S$. By induction hypothesis, the new position will be a $\outcomeP$-position.
\item If $S \neq P_m$, then the first player cannot remove the central vertex. In this case, for every first player's move on $S$, the second player can replicate it on $S'$, allowing us to invoke the induction hypothesis.
\end{itemize}
\end{proof}

Theorem~\ref{thm:modkpodes} then directly follows from Lemma~\ref{lem:modkpodes}.
Hence, all paths of length $3p$ can be removed, all paths of length $3p+1$ can be reduced to a single edge, and all paths of length $3p+2$ can be reduced to a path of length 2. If we want to know the Grundy value of a given subdivided star, it then suffices to study the Grundy values of the subdivided stars with paths of length 1 and 2 attached to their central vertex.

We are able to build a table of positions and their options: the subdivided star in row $i$ and column $j$, $j \leq i$, has $i$ paths attached to its central vertex, $j$ of them being of length 2. Figure~\ref{fig:tabpos} shows the first six rows of this table (the first two rows correspond to the empty graph and the subdivided star reduced to its central vertex, respectively):

\begin{figure}[!h]
\centering
\begin{tikzpicture}
	\draw [->,thick] (-1,1.5) -- (-1,-5.5);
	\draw [->,thick] (-1,1.5) -- (10.5,1.5);
	
	\draw (4.75,2.4) node {Number of paths of length 2 in the subdivided star};
	\draw (-1.8,-2) node[rotate=90]  {Number of paths in the subdivided star};
	
	\draw (-1.2,0) node {0};
	\draw (-1.2,-1) node {1};
	\draw (-1.2,-2) node {2};
	\draw (-1.2,-3) node {3};
	\draw (-1.2,-4) node {4};
	\draw (-1.2,-5) node {5};
	
	\draw (0,1.8) node {0};
	\draw (2,1.8) node {1};
	\draw (4,1.8) node {2};
	\draw (6,1.8) node {3};
	\draw (8,1.8) node {4};
	\draw (10,1.8) node {5};
	
	\node (empty) at (0,1) {$\emptyset$};
	\node (p1) at (0,0) {$P_1$};
	
	\node (p2) at (0,-1) {$P_2$};
	\node (p3) at (2,-1) {$P_3$};

	\node (p3b) at (0,-2) {$P_3$};
	\node (p4) at (2,-2) {$P_4$};
	\node (p5) at (4,-2) {$P_5$};
	
	\node (s111) at (0,-3) {\sstar{1,1,1}};
	\node (s112) at (2,-3) {\sstar{1,1,2}};
	\node (s122) at (4,-3) {\sstar{1,2,2}};
	\node (s222) at (6,-3) {\sstar{2,2,2}};
	
	\node (s1111) at (0,-4) {\sstar{1,1,1,1}};
	\node (s1112) at (2,-4) {\sstar{1,1,1,2}};
	\node (s1122) at (4,-4) {\sstar{1,1,2,2}};
	\node (s1222) at (6,-4) {\sstar{1,2,2,2}};
	\node (s2222) at (8,-4) {\sstar{2,2,2,2}};
	
	\node (s11111) at (0,-5) {\sstar{1,1,1,1,1}};
	\node (s11112) at (2,-5) {\sstar{1,1,1,1,2}};
	\node (s11122) at (4,-5) {\sstar{1,1,1,2,2}};
	\node (s11222) at (6,-5) {\sstar{1,1,2,2,2}};
	\node (s12222) at (8,-5) {\sstar{1,2,2,2,2}};
	\node (s22222) at (10,-5) {\sstar{2,2,2,2,2}};
	
	\draw [->] (p1) to (empty);
	\draw [->] (p2) to[out=120, in=-120] (empty);
	\draw [->] (p3b) to[out=120, in=-120] (p1);
	\draw [->] (p2) to (p1);
	\draw [->] (p3) to (p1);
	\draw [->] (p3) to (p2);
	\draw [->] (p3b) to (p2);
	\draw [->] (p4) to (p3b);
	\draw [->] (p4) to (p2);
	\draw [->] (p4) to (p3);
	\draw [->] (p5) to (p3);
	\draw [->] (p5) to (p4);
	\draw [->] (s111) to (p3b);
	\draw [->] (s112) to (p3b);
	\draw [->] (s112) to (p4);
	\draw [->] (s112) to (s111);
	\draw [->] (s122) to (p5);
	\draw [->] (s122) to (p4);
	\draw [->] (s122) to (s112);
	\draw [->] (s222) to (p5);
	\draw [->] (s222) to (s122);
	\draw [->] (s1111) to (s111);
	\draw [->] (s1112) to (s111);
	\draw [->] (s1112) to (s112);
	\draw [->] (s1112) to (s1111);
	\draw [->] (s1122) to (s112);
	\draw [->] (s1122) to (s122);
	\draw [->] (s1122) to (s1112);
	\draw [->] (s1222) to (s122);
	\draw [->] (s1222) to (s222);
	\draw [->] (s1222) to (s1122);
	\draw [->] (s2222) to (s222);
	\draw [->] (s2222) to (s1222);
	\draw [->] (s11111) to (s1111);
	\draw [->] (s11112) to (s1111);
	\draw [->] (s11112) to (s1112);
	\draw [->] (s11112) to (s11111);
	\draw [->] (s11122) to (s1112);
	\draw [->] (s11122) to (s1122);
	\draw [->] (s11122) to (s11112);
	\draw [->] (s11222) to (s1122);
	\draw [->] (s11222) to (s1222);
	\draw [->] (s11222) to (s11122);
	\draw [->] (s12222) to (s1222);
	\draw [->] (s12222) to (s2222);
	\draw [->] (s12222) to (s11222);
	\draw [->] (s22222) to (s2222);
	\draw [->] (s22222) to (s12222);
\end{tikzpicture}
\caption{The first six rows of the table of positions and their options.}
\label{fig:tabpos}
\end{figure}
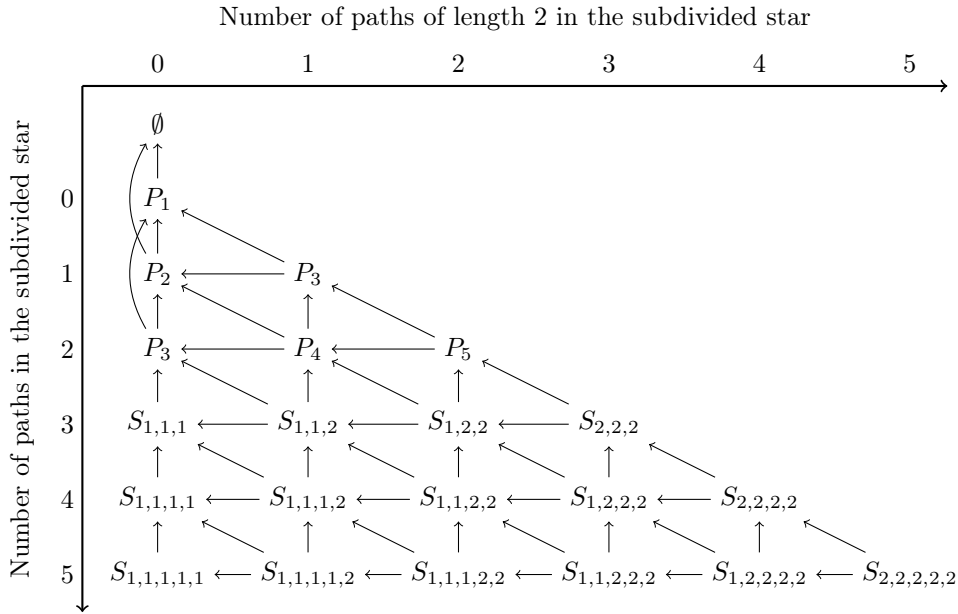

Since the Grundy value of the empty graph is 0, we can deduce the Grundy value of every star by proceeding inductively from the top lines:

\begin{theorem}
	\label{thm:grunStars}
	Figure~\ref{fig:tabgrun} shows the table of the Grundy values of subdivided stars after reduction of their paths modulo 3.
\end{theorem}

Except for the four first rows, the rows corresponding to stars with an odd number of paths are of the form $1203(12)^*$ whereas the rows corresponding to stars with an even number of paths are of the form $03120(30)^*$. Moreover, except for the four first columns, the columns with an even number of paths of length 2 are of the form $(01)^*$ whereas the columns with an odd number of paths of length 2 are of the form $(23)^*$.

\begin{figure}[!h]
	\centering
	\begin{tikzpicture}
	\draw [->,thick] (-1,1.5) -- (-1,-8.5);
	\draw [->,thick] (-1,1.5) -- (8.5,1.5);
	
	\draw (3.75,2.4) node {Number of paths of length 2 in the subdivided star};
	\draw (-2.3,-3.5) node[rotate=90]  {Number of paths in the subdivided star};
	
	\draw (-1.2,0) node {0};
	\draw (-1.2,-1) node {1};
	\draw (-1.2,-2) node {2};
	\draw (-1.2,-3) node {3};
	\draw (-1.2,-4) node {4};
	\draw (-1.2,-5) node {5};
	\draw (-1.5,-6) node {\ldots};
	\draw (-1.5,-7) node {$2p$};
	\draw (-1.5,-8) node {$2p+1$};
	
	\draw (0,1.8) node {0};
	\draw (1,1.8) node {1};
	\draw (2,1.8) node {2};
	\draw (3,1.8) node {3};
	\draw (4,1.8) node {4};
	\draw (5,1.8) node {5};
	\draw (6,1.8) node {\ldots};
	\draw (7,1.8) node {$2p$};
	\draw (8,1.8) node {$2p+1$};
	
	\node (empty) at (0,1) {0};
	\node (p1) at (0,0) {1};
	
	\node (p2) at (0,-1) {2};
	\node (p3) at (1,-1) {0};
	
	\node (p3b) at (0,-2) {0};
	\node (p4) at (1,-2) {1};
	\node (p5) at (2,-2) {2};
	
	\node (s111) at (0,-3) {1};
	\node (s112) at (1,-3) {2};
	\node (s122) at (2,-3) {0};
	\node (s222) at (3,-3) {1};
	
	\node (s1111) at (0,-4) {0};
	\node (s1112) at (1,-4) {3};
	\node (s1122) at (2,-4) {1};
	\node (s1222) at (3,-4) {2};
	\node (s2222) at (4,-4) {0};
	
	\node (s11111) at (0,-5) {1};
	\node (s11112) at (1,-5) {2};
	\node (s11122) at (2,-5) {0};
	\node (s11222) at (3,-5) {3};
	\node (s12222) at (4,-5) {1};
	\node (s22222) at (5,-5) {2};
	
	\draw [->] (p1) to (empty);
	\draw [->] (p2) to[out=120, in=-120] (empty);
	\draw [->] (p3b) to[out=120, in=-120] (p1);
	\draw [->] (p2) to (p1);
	\draw [->] (p3) to (p1);
	\draw [->] (p3) to (p2);
	\draw [->] (p3b) to (p2);
	\draw [->] (p4) to (p3b);
	\draw [->] (p4) to (p2);
	\draw [->] (p4) to (p3);
	\draw [->] (p5) to (p3);
	\draw [->] (p5) to (p4);
	\draw [->] (s111) to (p3b);
	\draw [->] (s112) to (p3b);
	\draw [->] (s112) to (p4);
	\draw [->] (s112) to (s111);
	\draw [->] (s122) to (p5);
	\draw [->] (s122) to (p4);
	\draw [->] (s122) to (s112);
	\draw [->] (s222) to (p5);
	\draw [->] (s222) to (s122);
	\draw [->] (s1111) to (s111);
	\draw [->] (s1112) to (s111);
	\draw [->] (s1112) to (s112);
	\draw [->] (s1112) to (s1111);
	\draw [->] (s1122) to (s112);
	\draw [->] (s1122) to (s122);
	\draw [->] (s1122) to (s1112);
	\draw [->] (s1222) to (s122);
	\draw [->] (s1222) to (s222);
	\draw [->] (s1222) to (s1122);
	\draw [->] (s2222) to (s222);
	\draw [->] (s2222) to (s1222);
	\draw [->] (s11111) to (s1111);
	\draw [->] (s11112) to (s1111);
	\draw [->] (s11112) to (s1112);
	\draw [->] (s11112) to (s11111);
	\draw [->] (s11122) to (s1112);
	\draw [->] (s11122) to (s1122);
	\draw [->] (s11122) to (s11112);
	\draw [->] (s11222) to (s1122);
	\draw [->] (s11222) to (s1222);
	\draw [->] (s11222) to (s11122);
	\draw [->] (s12222) to (s1222);
	\draw [->] (s12222) to (s2222);
	\draw [->] (s12222) to (s11222);
	\draw [->] (s22222) to (s2222);
	\draw [->] (s22222) to (s12222);
	
	\draw (0,-7) node {$0$};
	\draw (1,-7) node {$3$};
	\draw (2,-7) node {$1$};
	\draw (3,-7) node {$2$};
	\draw (4,-7) node {$0$};
	\draw (5,-7) node {$3$};
	\draw (6,-7) node {\ldots};
	\draw (7,-7) node {$0$};
	
	\draw (0,-8) node {$1$};
	\draw (1,-8) node {$2$};
	\draw (2,-8) node {$0$};
	\draw (3,-8) node {$3$};
	\draw (4,-8) node {$1$};
	\draw (5,-8) node {$2$};
	\draw (6,-8) node {\ldots};
	\draw (7,-8) node {$1$};
	\draw (8,-8) node {$2$};
	
	\draw [<-] (0.2,-7) -- (0.8,-7);
	\draw [<-] (1.2,-7) -- (1.8,-7);
	\draw [<-] (2.2,-7) -- (2.8,-7);
	\draw [<-] (3.2,-7) -- (3.8,-7);
	\draw [<-] (4.2,-7) -- (4.8,-7);
	
	\draw [<-] (0.2,-8) -- (0.8,-8);
	\draw [<-] (1.2,-8) -- (1.8,-8);
	\draw [<-] (2.2,-8) -- (2.8,-8);
	\draw [<-] (3.2,-8) -- (3.8,-8);
	\draw [<-] (4.2,-8) -- (4.8,-8);
	\draw [<-] (7.2,-8) -- (7.8,-8);
	
	\draw [<-] (0,-7.2) -- (0,-7.8);
	\draw [<-] (1,-7.2) -- (1,-7.8);
	\draw [<-] (2,-7.2) -- (2,-7.8);
	\draw [<-] (3,-7.2) -- (3,-7.8);
	\draw [<-] (4,-7.2) -- (4,-7.8);
	\draw [<-] (5,-7.2) -- (5,-7.8);
	\draw [<-] (7,-7.2) -- (7,-7.8);
	
	\draw [<-] (0.2,-7.2) -- (0.8,-7.8);
	\draw [<-] (1.2,-7.2) -- (1.8,-7.8);
	\draw [<-] (2.2,-7.2) -- (2.8,-7.8);
	\draw [<-] (3.2,-7.2) -- (3.8,-7.8);
	\draw [<-] (4.2,-7.2) -- (4.8,-7.8);
	\draw [<-] (7.2,-7.2) -- (7.8,-7.8);
	\end{tikzpicture}
	\caption{First six rows, and rows $2p$ and $2p+1$, of the table of Grundy values for the subdivided stars.}
	\label{fig:tabgrun}
\end{figure}
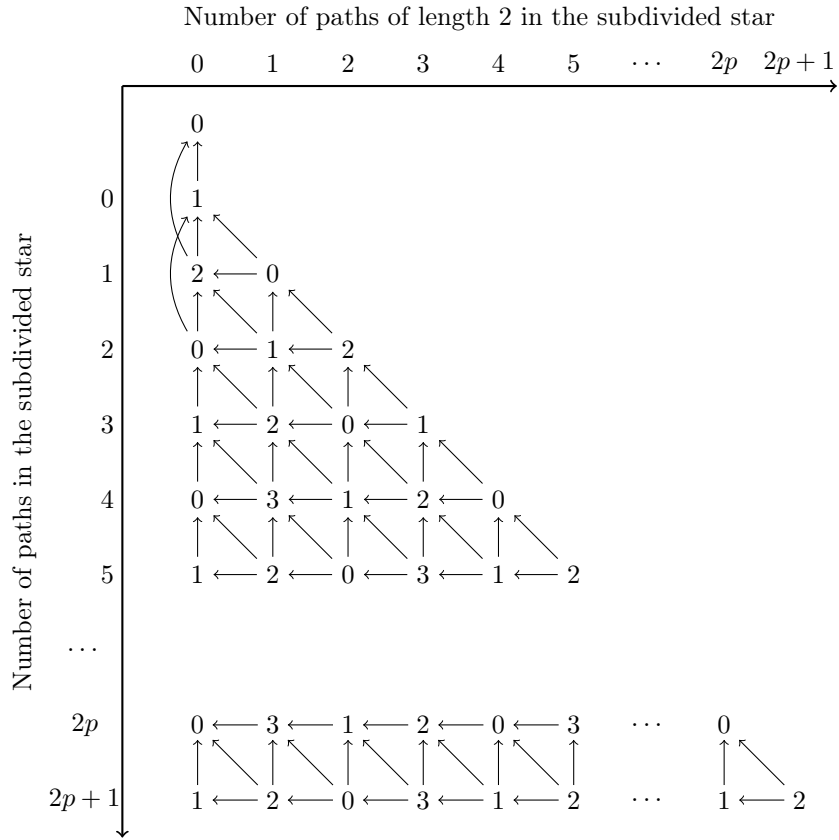


\section{The game {\bf 0.33} on subdivided bistars}\label{sec:bistar}

Let $S$ and $S'$ be two subdivided stars. The subdivided bistar \sbstar{S}{m}{S'} is the graph obtained by joining the central vertices of $S$ and $S'$ by a path of $m$ edges.
If $m=0$, then the subdivided bistar is a subdivided star. Likewise, if $m \geq 1$, $S=$\sstar{\ell_1,\ldots,\ell_k} and $S'=\emptyset$, then the subdivided bistar \sbstar{S}{m}{S'} is the subdivided star \sstar{\ell_1,\ldots,\ell_k,m-1}.
Figure~\ref{fig:exbipod} shows an example of a subdivided bistar.

For the sake of convenience, we will denote \sbstar{S}{1}{S'} by $S \bipa S'$.

\begin{figure}[H]
\centering
\begin{tikzpicture}[scale=0.6]
	\node[noeud] (c1) at (0,0) {};
	\node[noeud] (a11) at (-1,1) {};
	\node[noeud] (a21) at (-1,0) {};
	\node[noeud] (a22) at (-2,0) {};
	\node[noeud] (a23) at (-3,0) {};
	\node[noeud] (a31) at (-1,-1) {};
	\node[noeud] (a32) at (-2,-1) {};
	\draw (c1) -- (a11);
	\draw (c1) -- (a21);
	\draw (a23) -- (a21);
	\draw (c1) -- (a31);
	\draw (a32) -- (a31);
	
	\node[noeud] (c2) at (2,0) {};
	\node[noeud] (b11) at (3,0) {};
	\node[noeud] (b21) at (3,1) {};
	\node[noeud] (b22) at (4,1) {};
	\node[noeud] (b23) at (5,1) {};
	\node[noeud] (b24) at (6,1) {};
	\node[noeud] (b31) at (3,-1) {};
	\node[noeud] (b32) at (4,-1) {};
	\node[noeud] (b33) at (5,-1) {};
	\draw (c2) -- (b11);
	\draw (c2) -- (b21);
	\draw (b24) -- (b21);
	\draw (c2) -- (b31);
	\draw (b33) -- (b31);
	
	\node[noeud] (m) at (1,0) {};
	\draw (c1) -- (m);
	\draw (m) -- (c2);
\end{tikzpicture}
\caption[sbstar]{The subdivided bistar \sstar{1,2,3}\joint{2}\sstar{1,3,4}.}
\label{fig:exbipod}
\end{figure}

We notice that playing the {\bf 0.33} game on a subdivided bistar is similar to playing the {\bf 0.33} game on the two subdivided stars composing it with an "adjustment" depending on the length of the path linking the two stars, except for some small cases where one of the stars can be emptied so that one can play on the middle path.

This section is divided in two parts. In the first part, we will prove that every path of length $\ell$ in a subdivided bistar can be reduced to a path of lenth $\ell \bmod 3$ without changing the Grundy value:

\begin{theorem}
	\label{thm:modbipodes}
	For all $\ell_1,\ldots,\ell_k,\ell'_{1},\ldots,\ell'_{k'},m$, we have:
		\[
		\grundy(S_{\ell_1,\ldots,\ell_k}
		\begin{tikzpicture}[baseline=-4]\draw (0,0) node[circle,fill=black,minimum size=0,inner sep=1]{} -- (0.3,0) node[circle,fill=black, minimum size=0, inner sep=1] {} node[midway,above,scale=0.5] {$m$};\end{tikzpicture}
		S_{\ell'_{1},\ldots,\ell'_{k'}})
		=
		\grundy(S_{\ell_1 \bmod 3,\ldots,\ell_k \bmod 3}
		\begin{tikzpicture}[baseline=-4]\draw (0,0) node[circle,fill=black,minimum size=0,inner sep=1]{} -- (0.3,0) node[circle,fill=black, minimum size=0, inner sep=1] {} node[midway,above,scale=0.5] {$m\bmod3$};\end{tikzpicture}
		S_{\ell'_{1} \bmod 3,\ldots,\ell'_{k'} \bmod 3})
		\]
\end{theorem}

In the second part, we compute the Grundy value of a subdivided bistar, depending on the Grundy values of each of its two subdivided stars.

\subsection{Reducing the paths of a subdivided bistar}

In this section, we prove Theorem~\ref{thm:modbipodes}. We begin by proving the result for the middle path, before proving it for the paths of the two subdivided stars composing the bistar.

Note that we allow the length of the middle path to be 0, in which case the subdivided bistar is simply a subdivided star. Thus, if a subdivided bistar has a middle path of $3k$ edges, then it can be reduced to a subdivided star without changing its Grundy value.

\begin{lemma}
\label{lem:modbipodeschain}
	For all $\ell_1,\ldots,\ell_k,\ell'_{1},\ldots,\ell'_{k'},m$, we have:
		\[
		\grundy(S_{\ell_1,\ldots,\ell_k}
		\begin{tikzpicture}[baseline=-4]\draw (0,0) node[circle,fill=black,minimum size=0,inner sep=1]{} -- (0.3,0) node[circle,fill=black, minimum size=0, inner sep=1] {} node[midway,above,scale=0.5] {$m$};\end{tikzpicture}
		S_{\ell'_{1},\ldots,\ell'_{k'}})
		=
		\grundy(S_{\ell_1,\ldots,\ell_k}
		\begin{tikzpicture}[baseline=-4]\draw (0,0) node[circle,fill=black,minimum size=0,inner sep=1]{} -- (0.3,0) node[circle,fill=black, minimum size=0, inner sep=1] {} node[midway,above,scale=0.5] {$m\bmod3$};\end{tikzpicture}
		S_{\ell'_{1},\ldots,\ell'_{k'}})
		\]
\end{lemma}

\begin{proof}
It is enough to prove that adding three edges to the path does not change the Grundy value of the subdivided bistar.

Let $S$ and $S'$ be two subdivided stars. Let $B=$\sbstar{S}{m}{S'} and $B'=$\sbstar{S}{m+3}{S'}, $m \geq 0$.
We show that $B + B'$ is a $\outcomeP$-position by proving that for every first player's move, the second player always has an answer leading to a $\outcomeP$-position. We use induction on the size of $B$.

The first player can play on the middle path if and only if one of the two stars is either empty or reduced to a single vertex or a $P_2$. In this case, $B$ and $B'$ are subdivided stars, and the result follows from Lemma~\ref{lem:modkpodes}.


Assume now that both $S$ and $S'$ have at least two vertices. Hence, the first player is unable to play on the middle path and can play either on $S$ or $S'$. The second player will replicate the same move on the other subdivided bistar. By induction hypothesis, the result follows.
\end{proof}

In order to prove that the paths of the stars can be reduced, we need a few technical lemmas.

\begin{lemma}
\label{lem:P3empty}
Let $S$ be a subdivided star, and $B=S$\bipa\sstar{1,1}. We have $\grundy(S)=\grundy(B)$.
\end{lemma}

\begin{proof}
We show that $S + B$ is a $\outcomeP$-position by proving that for every first player's move, the second player always has an answer leading to a $\outcomeP$-position. We use induction on the size of $S$.

The cases where $S$ is empty or the first player can remove its central vertex are:
\begin{itemize}
\item $S$ is empty, thus $B = P_3$, which is a $\outcomeP$-position;
\item $S$ is a single vertex, thus $B=$\sstar{1,1,1}. We know by Lemma~\ref{lem:grundyStars} that $\grundy(B)=1=\grundy(S)$;
\item $S=P_2$, thus $B=$\sstar{1,1,2}. Considering Figure~\ref{fig:tabgrun}, we get $\grundy(B)=2=\grundy(S)$;
\item $S=P_3$, and in that case, either $S=$\sstar{1,1} or $S=S_2$:
	\begin{enumerate}
		\item $B=$\sstar{1,1}\bipa\sstar{1,1}. $B$ is a $\outcomeP$-position: the first player has only one available move, and from the resulting graph the second player can play to $P_3$ which is a $\outcomeP$-position. Both $B$ and $S$ being $\outcomeP$-positions, we have $\grundy(B)=\grundy(S)$.
		\item $B=$\sstar{1,1,3}. By Theorem~\ref{thm:modkpodes}, $\grundy(S)=\grundy(B)$.
	\end{enumerate}
\item $S=S_{\ell}$ with $\ell \geq 4$ or $S=S_{1,\ell}$ with $\ell \geq 2$. By Observation~\ref{lem:keylemma}, the second player will always be able to replicate the first player's move on $B$, by playing the symmetrical move. By induction hypothesis, the new position is a $\outcomeP$-position.
\end{itemize}

Figure~\ref{fig:P3emptyPROOF} depicts the cases where the first player does not take the central vertex of $S$, and completes the proof.

\begin{figure}[H]
\centering
\begin{tikzpicture}
	\node (orig) at (0,0) {\fbox{
	\begin{tikzpicture}[scale=0.5]
		\node[noeud] at (0,0) {};
	
		\draw (-0.6,0) node[scale=0.75] {$S$};
	
		\draw (0,0) -- (-0.75,0.5);
		\draw (0,0) -- (-0.75,-0.5);
		\draw (-0.75,0.5) arc (135:225:0.7);
		
		\draw (0.6,0) node {+};
	
		\node[noeud] at (2,0) {};
		\node[noeud] at (3,0) {};
		\node[noeud] at (4,1) {};
		\node[noeud] at (4,-1) {};
	
		\draw (1.4,0) node[scale=0.75] {$S$};
	
		\draw (2,0) -- (1.25,0.5);
		\draw (2,0) -- (1.25,-0.5);
		\draw (1.25,0.5) arc (135:225:0.7);
	
		\draw (2,0) -- (3,0);
		\draw (3,0) to (4,1);
		\draw (3,0) to (4,-1);
	\end{tikzpicture}}
	};
	
	\node (playOnS1) at (4,0) {
	\begin{tikzpicture}[scale=0.5]
		\node[noeud] at (0,0) {};
	
		\draw (-0.6,0) node[scale=0.75] {$S'$};
	
		\draw (0,0) -- (-0.75,0.5);
		\draw (0,0) -- (-0.75,-0.5);
		\draw (-0.75,0.5) arc (135:225:0.7);
		
		\draw (0.6,0) node {+};
	
		\node[noeud] at (2,0) {};
		\node[noeud] at (3,0) {};
		\node[noeud] at (4,1) {};
		\node[noeud] at (4,-1) {};
	
		\draw (1.4,0) node[scale=0.75] {$S$};
	
		\draw (2,0) -- (1.25,0.5);
		\draw (2,0) -- (1.25,-0.5);
		\draw (1.25,0.5) arc (135:225:0.7);
	
		\draw (2,0) -- (3,0);
		\draw (3,0) to (4,1);
		\draw (3,0) to (4,-1);
	\end{tikzpicture}
	};
	\node (playOnS2) at (4,-1.5) {
	\begin{tikzpicture}[scale=0.5]
		\node[noeud] at (0,0) {};
	
		\draw (-0.6,0) node[scale=0.75] {$S$};
	
		\draw (0,0) -- (-0.75,0.5);
		\draw (0,0) -- (-0.75,-0.5);
		\draw (-0.75,0.5) arc (135:225:0.7);
		
		\draw (0.6,0) node {+};
	
		\node[noeud] at (2,0) {};
		\node[noeud] at (3,0) {};
		\node[noeud] at (4,1) {};
		\node[noeud] at (4,-1) {};
	
		\draw (1.4,0) node[scale=0.75] {$S'$};
	
		\draw (2,0) -- (1.25,0.5);
		\draw (2,0) -- (1.25,-0.5);
		\draw (1.25,0.5) arc (135:225:0.7);
	
		\draw (2,0) -- (3,0);
		\draw (3,0) to (4,1);
		\draw (3,0) to (4,-1);
	\end{tikzpicture}
	};
	\node (playOnS3) at (8,-0.75) {
	\begin{tikzpicture}[scale=0.5]
		\node[noeud] at (0,0) {};
	
		\draw (-0.6,0) node[scale=0.75] {$S'$};
	
		\draw (0,0) -- (-0.75,0.5);
		\draw (0,0) -- (-0.75,-0.5);
		\draw (-0.75,0.5) arc (135:225:0.7);
		
		\draw (0.6,0) node {+};
	
		\node[noeud] at (2,0) {};
		\node[noeud] at (3,0) {};
		\node[noeud] at (4,1) {};
		\node[noeud] at (4,-1) {};
	
		\draw (1.4,0) node[scale=0.75] {$S'$};
	
		\draw (2,0) -- (1.25,0.5);
		\draw (2,0) -- (1.25,-0.5);
		\draw (1.25,0.5) arc (135:225:0.7);
	
		\draw (2,0) -- (3,0);
		\draw (3,0) to (4,1);
		\draw (3,0) to (4,-1);
	\end{tikzpicture}
	};
	\draw (11,-0.75) node {$\outcomeP$ by induction};
	\draw (11,-1.25) node {hypothesis};
	
	\node (bli) at (4,-3) {
	\begin{tikzpicture}[scale=0.5]
		\node[noeud] at (0,0) {};
	
		\draw (-0.6,0) node[scale=0.75] {$S$};
	
		\draw (0,0) -- (-0.75,0.5);
		\draw (0,0) -- (-0.75,-0.5);
		\draw (-0.75,0.5) arc (135:225:0.7);
		
		\draw (0.6,0) node {+};
	
		\node[noeud] at (2,0) {};
		\node[noeud] at (3,0) {};
		\node[noeud] at (4,1) {};
	
		\draw (1.4,0) node[scale=0.75] {$S$};
	
		\draw (2,0) -- (1.25,0.5);
		\draw (2,0) -- (1.25,-0.5);
		\draw (1.25,0.5) arc (135:225:0.7);
	
		\draw (2,0) -- (3,0);
		\draw (3,0) to (4,1);
	\end{tikzpicture}
	};
	\node (blir) at (8,-3) {
	\begin{tikzpicture}[scale=0.5]
		\node[noeud] at (0,0) {};
	
		\draw (-0.6,0) node[scale=0.75] {$S$};
	
		\draw (0,0) -- (-0.75,0.5);
		\draw (0,0) -- (-0.75,-0.5);
		\draw (-0.75,0.5) arc (135:225:0.7);
		
		\draw (0.6,0) node {+};
	
		\node[noeud] at (2,0) {};
	
		\draw (1.4,0) node[scale=0.75] {$S$};
	
		\draw (2,0) -- (1.25,0.5);
		\draw (2,0) -- (1.25,-0.5);
		\draw (1.25,0.5) arc (135:225:0.7);
	\end{tikzpicture}
	};
	\draw (11,-3) node {$\outcomeP$};
	
	\draw (orig) -- (2,0);
	\draw[->] (2,0) -- (playOnS1);
	\draw[->] (2,0) |- (playOnS2);
	\draw[->] (playOnS1) -- (playOnS3);
	\draw[->] (playOnS2) -- (playOnS3);
	\draw[->] (2,0) |- (bli);
	\draw[->] (bli) -- (blir);
\end{tikzpicture}
\caption{The inductive part of the proof of Lemma~\ref{lem:P3empty}}
\label{fig:P3emptyPROOF}
\end{figure}
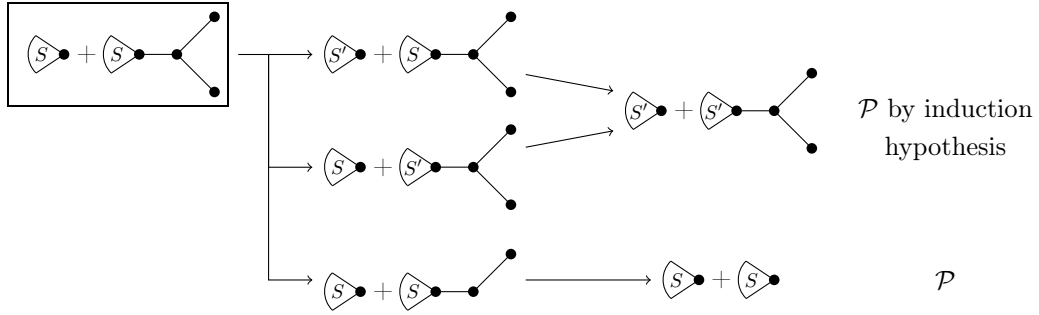
\end{proof}

Let $S$ be a subdivided star, we denote \sbstar{S}{2}{$\emptyset$} by $S\bipa$. We then have:

\begin{lemma}
\label{lem:P3onevertex}
Let $S$ be a subidvided star. We have $\grundy(S\bipa)=\grundy(S$\tripa\sstar{1,1}$)$.
\end{lemma}

\begin{proof}
Let $B=S$\tripa\sstar{1,1}.	
We show that $S\bipa + B$ is a $\outcomeP$-position by proving that for every first player's move, the second player always has an answer leading to a $\outcomeP$-position. We use induction on the size of $S$.

The cases where $S$ is empty, or where the first player can remove either the central vertex of $S$ or both the central vertex of $S$ and the vertex from the middle path of $S\bipa$ are:
\begin{itemize}
\item $S$ is empty, thus $S\bipa = P_1$ and $B=$\sstar{1,1,1}. We know by Lemma~\ref{lem:grundyStars} that $\grundy(B)=1=\grundy(S\bipa)$ so $S\bipa+B$ is a $\outcomeP$-position;
\item $S$ is a single vertex, thus $S\bipa = P_2$ and $B=$\sstar{1,1,2}. Considering Figure~\ref{fig:tabgrun}, we get $\grundy(B)=2=\grundy(S\bipa)$ so $S\bipa+B$ is a $\outcomeP$-position;
\item $S=P_2$, thus $S \bipa = P_3$ and $B=$\sstar{1,1,3} which by Lemma~\ref{lem:modkpodes} has the same Grundy value as \sstar{1,1}, \emph{i.e.} as $P_3$. Thus, $\grundy(B)=\grundy(S\bipa)$ so $S\bipa+B$ is a $\outcomeP$-position;
\item $S\bipa=$\sstar{1,1,1}, thus $B=$\sbstar{\sstar{1,1}}{2}{\sstar{1,1}}. Considering the table in Figure~\ref{fig:tabgrun}, we get $\grundy(S\bipa)=1$. It is easy to see that $\grundy(B)=1$, since only one move is available for the first player (removing one leaf vertex), which leaves \sstar{1,1,3} which is a $\outcomeP$-position. Thus $\grundy(S\bipa)=\grundy(B)$ so $S\bipa+B$ is a $\outcomeP$-position.
\item $S=S_{\ell}$ with $\ell \geq 4$ or $S=S_{1,\ell}$ with $\ell \geq 2$. By Observation~\ref{lem:keylemma}, the second player will always be able to replicate the first player's move on $B$, by playing the symmetrical move. By induction hypothesis, the new position is a $\outcomeP$-position.
\end{itemize}

Figure~\ref{fig:P3onevertexPROOF} depicts the cases where the first player takes neither the central vertex of $S$ nor both the central vertex of $S$ and the vertex from the middle path of $S\bipa$, and completes the proof.

\begin{figure}[H]
\centering
\begin{tikzpicture}
	\node (orig) at (0,0) {\fbox{
	\begin{tikzpicture}[scale=0.5]
		\node[noeud] at (-1,0) {};
		\node[noeud] at (0,0) {};
	
		\draw (-1.6,0) node[scale=0.75] {$S$};
	
		\draw (-1,0) -- (0,0);
		\draw (-1,0) -- (-1.75,0.5);
		\draw (-1,0) -- (-1.75,-0.5);
		\draw (-1.75,0.5) arc (135:225:0.7);
	
		\draw (0.6,0) node {+};
	
		\node[noeud] at (2,0) {};
		\node[noeud] at (3,0) {};
		\node[noeud] at (4,0) {};
		\node[noeud] at (5,1) {};
		\node[noeud] at (5,-1) {};
	
		\draw (1.4,0) node[scale=0.75] {$S$};
	
		\draw (2,0) -- (1.25,0.5);
		\draw (2,0) -- (1.25,-0.5);
		\draw (1.25,0.5) arc (135:225:0.7);
	
		\draw (2,0) -- (4,0);
		\draw (4,0) to (5,1);
		\draw (4,0) to (5,-1);
	\end{tikzpicture}}
	};
	\node (playOnS1) at (5,0) {
	\begin{tikzpicture}[scale=0.5]
		\node[noeud] at (-1,0) {};
		\node[noeud] at (0,0) {};
	
		\draw (-1.6,0) node[scale=0.75] {$S'$};
	
		\draw (-1,0) -- (0,0);
		\draw (-1,0) -- (-1.75,0.5);
		\draw (-1,0) -- (-1.75,-0.5);
		\draw (-1.75,0.5) arc (135:225:0.7);
	
		\draw (0.6,0) node {+};
	
		\node[noeud] at (2,0) {};
		\node[noeud] at (3,0) {};
		\node[noeud] at (4,0) {};
		\node[noeud] at (5,1) {};
		\node[noeud] at (5,-1) {};
	
		\draw (1.4,0) node[scale=0.75] {$S$};
	
		\draw (2,0) -- (1.25,0.5);
		\draw (2,0) -- (1.25,-0.5);
		\draw (1.25,0.5) arc (135:225:0.7);
	
		\draw (2,0) -- (4,0);
		\draw (4,0) to (5,1);
		\draw (4,0) to (5,-1);
	\end{tikzpicture}
	};
	\node (playOnS2) at (5,-1.5) {
	\begin{tikzpicture}[scale=0.5]
		\node[noeud] at (-1,0) {};
		\node[noeud] at (0,0) {};
	
		\draw (-1.6,0) node[scale=0.75] {$S$};
	
		\draw (-1,0) -- (0,0);
		\draw (-1,0) -- (-1.75,0.5);
		\draw (-1,0) -- (-1.75,-0.5);
		\draw (-1.75,0.5) arc (135:225:0.7);
	
		\draw (0.6,0) node {+};
	
		\node[noeud] at (2,0) {};
		\node[noeud] at (3,0) {};
		\node[noeud] at (4,0) {};
		\node[noeud] at (5,1) {};
		\node[noeud] at (5,-1) {};
	
		\draw (1.4,0) node[scale=0.75] {$S'$};
	
		\draw (2,0) -- (1.25,0.5);
		\draw (2,0) -- (1.25,-0.5);
		\draw (1.25,0.5) arc (135:225:0.7);
	
		\draw (2,0) -- (4,0);
		\draw (4,0) to (5,1);
		\draw (4,0) to (5,-1);
	\end{tikzpicture}
	};
	\node (playOnS3) at (10,-0.75) {
	\begin{tikzpicture}[scale=0.5]
		\node[noeud] at (-1,0) {};
		\node[noeud] at (0,0) {};
	
		\draw (-1.6,0) node[scale=0.75] {$S'$};
	
		\draw (-1,0) -- (0,0);
		\draw (-1,0) -- (-1.75,0.5);
		\draw (-1,0) -- (-1.75,-0.5);
		\draw (-1.75,0.5) arc (135:225:0.7);
	
		\draw (0.6,0) node {+};
	
		\node[noeud] at (2,0) {};
		\node[noeud] at (3,0) {};
		\node[noeud] at (4,0) {};
		\node[noeud] at (5,1) {};
		\node[noeud] at (5,-1) {};
	
		\draw (1.4,0) node[scale=0.75] {$S'$};
	
		\draw (2,0) -- (1.25,0.5);
		\draw (2,0) -- (1.25,-0.5);
		\draw (1.25,0.5) arc (135:225:0.7);
	
		\draw (2,0) -- (4,0);
		\draw (4,0) to (5,1);
		\draw (4,0) to (5,-1);
	\end{tikzpicture}
	};
	\draw (14,-0.75) node {$\outcomeP$ by induction};
	\draw (14,-1.25) node {hypothesis};
	
	\node (abli) at (5,-3) {
	\begin{tikzpicture}[scale=0.5]
		\node[noeud] at (-1,0) {};
	
		\draw (-1.6,0) node[scale=0.75] {$S$};
	
		\draw (-1,0) -- (-1.75,0.5);
		\draw (-1,0) -- (-1.75,-0.5);
		\draw (-1.75,0.5) arc (135:225:0.7);
	
		\draw (0,0) node {+};
	
		\node[noeud] at (2,0) {};
		\node[noeud] at (3,0) {};
		\node[noeud] at (4,0) {};
		\node[noeud] at (5,1) {};
		\node[noeud] at (5,-1) {};
	
		\draw (1.4,0) node[scale=0.75] {$S$};
	
		\draw (2,0) -- (1.25,0.5);
		\draw (2,0) -- (1.25,-0.5);
		\draw (1.25,0.5) arc (135:225:0.7);
	
		\draw (2,0) -- (4,0);
		\draw (4,0) to (5,1);
		\draw (4,0) to (5,-1);
	\end{tikzpicture}
	};
	\node (bbli) at (5,-4.5) {
	\begin{tikzpicture}[scale=0.5]
		\node[noeud] at (-1,0) {};
		\node[noeud] at (0,0) {};
	
		\draw (-1.6,0) node[scale=0.75] {$S$};
	
		\draw (-1,0) -- (0,0);
		\draw (-1,0) -- (-1.75,0.5);
		\draw (-1,0) -- (-1.75,-0.5);
		\draw (-1.75,0.5) arc (135:225:0.7);
	
		\draw (0.6,0) node {+};
	
		\node[noeud] at (2,0) {};
		\node[noeud] at (3,0) {};
		\node[noeud] at (4,0) {};
		\node[noeud] at (5,1) {};
		\node at (5,-1) {};
	
		\draw (1.4,0) node[scale=0.75] {$S$};
	
		\draw (2,0) -- (1.25,0.5);
		\draw (2,0) -- (1.25,-0.5);
		\draw (1.25,0.5) arc (135:225:0.7);
	
		\draw (2,0) -- (4,0);
		\draw (4,0) to (5,1);
	\end{tikzpicture}
	};
	\node (cbli) at (10,-3.75) {
	\begin{tikzpicture}[scale=0.5]
		\node[noeud] at (-1,0) {};
	
		\draw (-1.6,0) node[scale=0.75] {$S$};
	
		\draw (-1,0) -- (-1.75,0.5);
		\draw (-1,0) -- (-1.75,-0.5);
		\draw (-1.75,0.5) arc (135:225:0.7);
	
		\draw (0,0) node {+};
	
		\node[noeud] at (2,0) {};
		\node[noeud] at (3,0) {};
		\node[noeud] at (4,0) {};
		\node[noeud] at (5,1) {};
		\node at (5,-1) {};
	
		\draw (1.4,0) node[scale=0.75] {$S$};
	
		\draw (2,0) -- (1.25,0.5);
		\draw (2,0) -- (1.25,-0.5);
		\draw (1.25,0.5) arc (135:225:0.7);
	
		\draw (2,0) -- (4,0);
		\draw (4,0) to (5,1);
	\end{tikzpicture}
	};
	\draw (14,-3.5) node {$\outcomeP$ by};
	\draw (14,-4) node {Lemma~\ref{lem:modkpodes}};
	
	\draw (orig) -- (2.5,0);
	\draw[->] (2.5,0) -- (playOnS1);
	\draw[->] (2.5,0) |- (playOnS2);
	\draw[->] (playOnS1) -- (playOnS3);
	\draw[->] (playOnS2) -- (playOnS3);
	\draw[->] (2.5,0) |- (abli);
	\draw[->] (2.5,0) |- (bbli);
	\draw[->] (abli) -- (cbli);
	\draw[->] (bbli) -- (cbli);
\end{tikzpicture}
\caption{The inductive part of the proof of Lemma~\ref{lem:P3onevertex}.}
\label{fig:P3onevertexPROOF}
\end{figure}
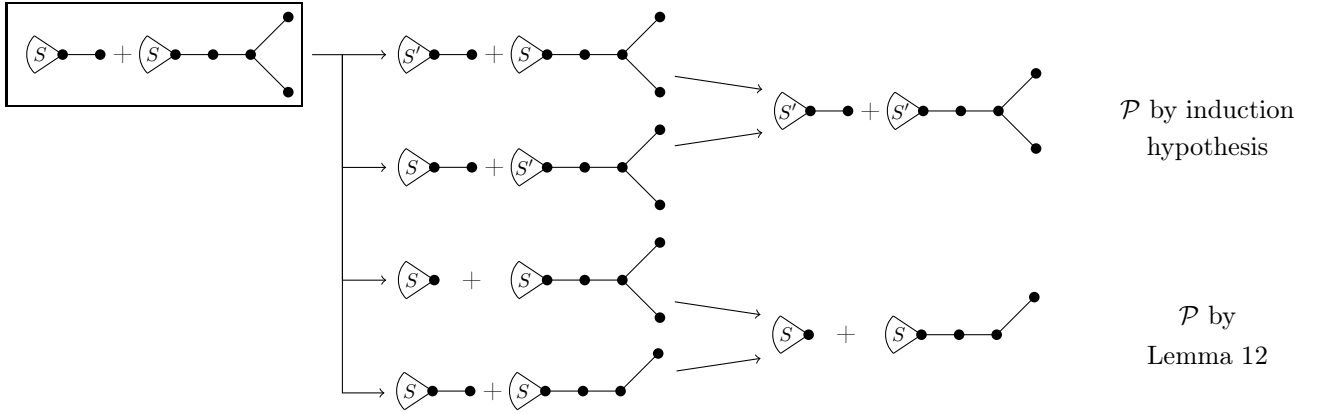
\end{proof}

We are now ready to prove that the paths of the two subdivided stars of a bistar can be reduced:

\begin{lemma}
\label{lem:modbipodesarms}
	For all $\ell_1,\ldots,\ell_k,\ell'_{1},\ldots,\ell'_{k'},m$, we have:
		\[
		\grundy(S_{\ell_1,\ldots,\ell_k}
		\begin{tikzpicture}[baseline=-4]\draw (0,0) node[circle,fill=black,minimum size=0,inner sep=1]{} -- (0.3,0) node[circle,fill=black, minimum size=0, inner sep=1] {} node[midway,above,scale=0.5] {$m$};\end{tikzpicture}
		S_{\ell'_{1},\ldots,\ell'_{k'}})
		=
		\grundy(S_{\ell_1 \bmod 3,\ldots,\ell_k \bmod 3}
		\begin{tikzpicture}[baseline=-4]\draw (0,0) node[circle,fill=black,minimum size=0,inner sep=1]{} -- (0.3,0) node[circle,fill=black, minimum size=0, inner sep=1] {} node[midway,above,scale=0.5] {$m$};\end{tikzpicture}
		S_{\ell'_{1} \bmod 3,\ldots,\ell'_{k'} \bmod 3})
		\]
\end{lemma}

\begin{proof}
Thanks to Lemma~\ref{lem:modbipodeschain}, we only have to prove the result on the subdivided bistars with a middle path of length 1 or 2.

Let $S$ and $S'$ be two subdivided stars, $B=$\sbstar{S}{i}{S'} with $i \in \{1,2\}$, and $B'$ the subdivided bistar obtained by attaching a $P_3$ to a leaf or to the central vertex of (without loss of generality) $S'$. We prove by induction on $|V(B)|$ that $\grundy(B)=\grundy(B')$.

First, we consider the cases where $S'$ is empty or the first player can remove its central vertex:
\begin{itemize}
\item If either $S$ or $S'$ is empty (resp. a single vertex), then $B$ is a subdivided star, and the result holds by Lemma~\ref{lem:modkpodes};
\item If $B=S$\bipa$P_2$ (resp. $B=S$\tripa$P_2$) and the first player empties $S'$ on $B$, then the second player is unable to replicate the move on $B'$. The strategy is then to take two vertices from the attached $P_3$. By Lemma~\ref{lem:P3empty} (resp. Lemma~\ref{lem:P3onevertex}), we have $\grundy(B)=\grundy(B')$.
\end{itemize}

Now, we consider the cases where the first player cannot take the central vertex of $S'$:
\begin{itemize}
	\item If $S'$ is a path with more than two vertices, and the $P_3$ is attached to its central vertex, then replicating the first player's move will always be possible and lead to a $\outcomeP$-position by induction hypothesis.
	\item If the first player takes one vertex (resp. two vertices) from the attached $P_3$ on $B'$, then the second player answers by taking two vertices (resp. one vertex) from it, leaving $B + B$ which is a $\outcomeP$-position.
	\item Otherwise, the first player plays either on $S$ or on $S'$ in either of the two bistars. Note that the first player cannot remove the central vertex of $S'$, since this case has already been treated above. The second player answers by replicating his move on the other bistar, which is always a legal move, allowing us to invoke the induction hypothesis.
\end{itemize}
\end{proof}

Theorem~\ref{thm:modbipodes} then follows from Lemmas~\ref{lem:modbipodeschain} and~\ref{lem:modbipodesarms}. As in the subdivided stars section, we are left with a limited number of bistars to study. The next subsection presents the study of the Grundy value of a subdivided bistar depending on the Grundy values of its subdivided stars.

\subsection{Computing the Grundy value of a subdivided bistar}

We will express the Grundy value of a subdivided star as a function of the Grundy values of its two stars.

By Lemma~\ref{lem:modbipodeschain}, it is enough to consider bistars whose middle path has length either 1 or 2. We consider these two cases separately.

\subsubsection{When the middle path is of length 1}

Playing on a subdivided bistar with a middle path of length 1 is almost equivalent to playing in the disjoint union of the two subdivided stars, except for small cases when some moves are not available in the bistar. We will see in what follows that except for some small cases, the Grundy value of the bistar is indeed the nim-sum of the Grundy values of the two stars.

We refine the equivalence relation $\equiv$ for subdivided stars as follows. Let $S$ and $S'$ be two subdivided stars. We say that $S$ and $S'$ are $\sim_1$-equivalent, denoted $S \sim_1 S'$, if and only if for any subdivided star $\hat{S}$, $S\bipa \hat{S} \equiv S' \bipa \hat{S}$.

Note that the Grundy value of a bistar $S\bipa S'$ only depends of the equivalence class under $\sim_1$ of $S$ and $S'$. The equivalence $\sim_1$ is a refinement of $\equiv$ since taking $\hat{S}=\emptyset$ we have $S\equiv S'$.

By Lemma~\ref{lem:P3empty}, we already know that $P_3\sim_1\emptyset$, and thus $S_2 \sim_1 \emptyset$ and \sstar{1,1}$\sim_1\emptyset$.

We will prove that there are actually eight equivalence classes for $\sim_1$:
\begin{itemize}
\item $\conestar=\{P_1,$\sstar{2,1},\sstar{2,2,2}$\}$ (these stars have Grundy value 1);
\item $\ctwostar=\{P_2$,\sstar{2,2}$\}$ (these stars have Grundy value 2);
\item $\ctwobox$: subdivided stars $S$ such that $\grundy(S)=2$ and $S$ contains one or three paths of length $2$;
\item $\cthreebox$: subdivided stars $S$ such that $\grundy(S)=3$ and S contains one or three paths of length $2$;
\item For $i\in \{0,1,2,3\}$, $\mathcal C_i$: subdivided stars $S$ with $\grundy(S)=i$ and $S$ is not in a previous class.
\end{itemize}
Figure~\ref{fig:tabEquivSim1} shows the equivalence classes of the subdivided stars.

\begin{figure}[!h]
	\centering
	\begin{tikzpicture}
	\draw [->,thick] (-1,1.5) -- (-1,-8.5);
	\draw [->,thick] (-1,1.5) -- (8.5,1.5);
	
	\draw (3.75,2.4) node {Number of paths of length 2 in the subdivided star};
	\draw (-2.3,-3.5) node[rotate=90]  {Number of paths in the subdivided star};
	
	\draw (-1.2,0) node {0};
	\draw (-1.2,-1) node {1};
	\draw (-1.2,-2) node {2};
	\draw (-1.2,-3) node {3};
	\draw (-1.2,-4) node {4};
	\draw (-1.2,-5) node {5};
	\draw (-1.5,-6) node {\ldots};
	\draw (-1.5,-7) node {$2p$};
	\draw (-1.5,-8) node {$2p+1$};
	
	\draw (0,1.8) node {0};
	\draw (1,1.8) node {1};
	\draw (2,1.8) node {2};
	\draw (3,1.8) node {3};
	\draw (4,1.8) node {4};
	\draw (5,1.8) node {5};
	\draw (6,1.8) node {\ldots};
	\draw (7,1.8) node {$2p$};
	\draw (8,1.8) node {$2p+1$};
	
	\node (empty) at (0,1) {0};
	\node (p1) at (0,0) {$1^*$};
	
	\node (p2) at (0,-1) {$2^*$};
	\node (p3) at (1,-1) {0};
	
	\node (p3b) at (0,-2) {0};
	\node (p4) at (1,-2) {$1^*$};
	\node (p5) at (2,-2) {$2^*$};
	
	\node (s111) at (0,-3) {1};
	\node (s112) at (1,-3) {$2^\Box$};
	\node (s122) at (2,-3) {0};
	\node (s222) at (3,-3) {$1^*$};
	
	\node (s1111) at (0,-4) {0};
	\node (s1112) at (1,-4) {$3^\Box$};
	\node (s1122) at (2,-4) {1};
	\node (s1222) at (3,-4) {$2^\Box$};
	\node (s2222) at (4,-4) {0};
	
	\node (s11111) at (0,-5) {1};
	\node (s11112) at (1,-5) {$2^\Box$};
	\node (s11122) at (2,-5) {0};
	\node (s11222) at (3,-5) {$3^\Box$};
	\node (s12222) at (4,-5) {1};
	\node (s22222) at (5,-5) {2};
	
	\draw [->] (p1) to (empty);
	\draw [->] (p2) to[out=120, in=-120] (empty);
	\draw [->] (p3b) to[out=120, in=-120] (p1);
	\draw [->] (p2) to (p1);
	\draw [->] (p3) to (p1);
	\draw [->] (p3) to (p2);
	\draw [->] (p3b) to (p2);
	\draw [->] (p4) to (p3b);
	\draw [->] (p4) to (p2);
	\draw [->] (p4) to (p3);
	\draw [->] (p5) to (p3);
	\draw [->] (p5) to (p4);
	\draw [->] (s111) to (p3b);
	\draw [->] (s112) to (p3b);
	\draw [->] (s112) to (p4);
	\draw [->] (s112) to (s111);
	\draw [->] (s122) to (p5);
	\draw [->] (s122) to (p4);
	\draw [->] (s122) to (s112);
	\draw [->] (s222) to (p5);
	\draw [->] (s222) to (s122);
	\draw [->] (s1111) to (s111);
	\draw [->] (s1112) to (s111);
	\draw [->] (s1112) to (s112);
	\draw [->] (s1112) to (s1111);
	\draw [->] (s1122) to (s112);
	\draw [->] (s1122) to (s122);
	\draw [->] (s1122) to (s1112);
	\draw [->] (s1222) to (s122);
	\draw [->] (s1222) to (s222);
	\draw [->] (s1222) to (s1122);
	\draw [->] (s2222) to (s222);
	\draw [->] (s2222) to (s1222);
	\draw [->] (s11111) to (s1111);
	\draw [->] (s11112) to (s1111);
	\draw [->] (s11112) to (s1112);
	\draw [->] (s11112) to (s11111);
	\draw [->] (s11122) to (s1112);
	\draw [->] (s11122) to (s1122);
	\draw [->] (s11122) to (s11112);
	\draw [->] (s11222) to (s1122);
	\draw [->] (s11222) to (s1222);
	\draw [->] (s11222) to (s11122);
	\draw [->] (s12222) to (s1222);
	\draw [->] (s12222) to (s2222);
	\draw [->] (s12222) to (s11222);
	\draw [->] (s22222) to (s2222);
	\draw [->] (s22222) to (s12222);
	
	\draw (0,-7) node {$0$};
	\draw (1,-7) node {$3^\Box$};
	\draw (2,-7) node {$1$};
	\draw (3,-7) node {$2^\Box$};
	\draw (4,-7) node {$0$};
	\draw (5,-7) node {$3$};
	\draw (6,-7) node {\ldots};
	\draw (7,-7) node {$0$};
	
	\draw (0,-8) node {$1$};
	\draw (1,-8) node {$2^\Box$};
	\draw (2,-8) node {$0$};
	\draw (3,-8) node {$3^\Box$};
	\draw (4,-8) node {$1$};
	\draw (5,-8) node {$2$};
	\draw (6,-8) node {\ldots};
	\draw (7,-8) node {$1$};
	\draw (8,-8) node {$2$};
	
	\draw [<-] (0.2,-7) -- (0.8,-7);
	\draw [<-] (1.2,-7) -- (1.8,-7);
	\draw [<-] (2.2,-7) -- (2.8,-7);
	\draw [<-] (3.2,-7) -- (3.8,-7);
	\draw [<-] (4.2,-7) -- (4.8,-7);
	
	\draw [<-] (0.2,-8) -- (0.8,-8);
	\draw [<-] (1.2,-8) -- (1.8,-8);
	\draw [<-] (2.2,-8) -- (2.8,-8);
	\draw [<-] (3.2,-8) -- (3.8,-8);
	\draw [<-] (4.2,-8) -- (4.8,-8);
	\draw [<-] (7.2,-8) -- (7.8,-8);
	
	\draw [<-] (0,-7.2) -- (0,-7.8);
	\draw [<-] (1,-7.2) -- (1,-7.8);
	\draw [<-] (2,-7.2) -- (2,-7.8);
	\draw [<-] (3,-7.2) -- (3,-7.8);
	\draw [<-] (4,-7.2) -- (4,-7.8);
	\draw [<-] (5,-7.2) -- (5,-7.8);
	\draw [<-] (7,-7.2) -- (7,-7.8);
	
	\draw [<-] (0.2,-7.2) -- (0.8,-7.8);
	\draw [<-] (1.2,-7.2) -- (1.8,-7.8);
	\draw [<-] (2.2,-7.2) -- (2.8,-7.8);
	\draw [<-] (3.2,-7.2) -- (3.8,-7.8);
	\draw [<-] (4.2,-7.2) -- (4.8,-7.8);
	\draw [<-] (7.2,-7.2) -- (7.8,-7.8);
	\end{tikzpicture}
	\caption{First six rows, and rows $2p$ and $2p+1$, of the table of equivalence classes for $\sim_1$ of the subdivided stars. Stars belonging to resp. $\conestar$, $\ctwostar$, $\ctwobox$, $\cthreebox$ are depicted by resp. $1^*$, $2^*$, $2^\Box$, $3^\Box$, while the $\mathcal C_i$'s are depicted by $i$.}
	\label{fig:tabEquivSim1}
\end{figure}

\begin{theorem}\label{thm:equiv1}
The equivalence classes for $\sim_1$ are exactly the sets $\mathcal C_0$, $\mathcal C_1$, $\conestar$, $\mathcal C_2$, $\ctwostar$, $\ctwobox$,  $\mathcal C_3$ and $\cthreebox$. Moreover, Table~\ref{tab:prod1} describes how the Grundy value of $S\bipa S'$ can be computed depending on the equivalence class of $S$ and $S'$.
\end{theorem}

\begin{table}[H]
\begin{center}
\begin{math}
\begin{array}{c|c|c|c|c|c|c|c|c}
&\mathcal C_0 & \mathcal C_1 & \conestar & \mathcal C_2 & \ctwostar & \ctwobox & \mathcal C_3 & \cthreebox \\ \hline
\mathcal C_0 & \nimsum &  \nimsum & \nimsum & \nimsum & \nimsum & \nimsum & \nimsum & \nimsum \\ \hline
\mathcal C_1 & \nimsum &  \nimsum & \nimsum & \nimsum & \nimsum & \nimsum & \nimsum & \nimsum \\ \hline
\conestar & \nimsum &  \nimsum & 2 & \nimsum & 0 & \nimsum & \nimsum & \nimsum \\ \hline
\mathcal C_2 & \nimsum &  \nimsum & \nimsum & \nimsum & \nimsum & \nimsum & \nimsum & \nimsum \\ \hline
\ctwostar & \nimsum &  \nimsum & 0 & \nimsum & 1 & 1 & \nimsum & 0 \\ \hline
\ctwobox & \nimsum &  \nimsum & \nimsum & \nimsum & 1 & \nimsum & \nimsum & \nimsum \\ \hline
\mathcal C_3 & \nimsum &  \nimsum & \nimsum & \nimsum & \nimsum & \nimsum & \nimsum & \nimsum \\ \hline
\cthreebox & \nimsum &  \nimsum & \nimsum & \nimsum & 0 & \nimsum & \nimsum & \nimsum \\ \hline
\end{array}
\end{math}
\end{center}
\caption[prod1]{Computing the Grundy value of $S \bipa S'$ depending on the equivalence class of $S$ and $S'$.}
\label{tab:prod1}
\end{table}

We will need some technical lemmas before proving the theorem:

\begin{lemma}\label{lem:equivstar}
	We have:
\begin{enumerate}
\item $P_1 \sim_1$ \sstar{2,1}
\item $P_2 \sim_1$ \sstar{2,2}
\item \sstar{1,1} $\sim_1$ \sstar{2,2,1}
\item \sstar{2,1} $\sim_1$ \sstar{2,2,2}.
\end{enumerate} 
Therefore, any two elements in $\conestar$ (resp. $\ctwostar$) are $\sim_1$-equivalent.
\end{lemma}

\begin{proof}
Each of these equivalences will be proved in the same way: for an equivalence $S \sim_1 S'$, we prove that for every subdivided star $\hat{S}$, $\grundy(\hat{S} \bipa S) = \grundy(\hat{S} \bipa S')$. We will use induction on the size of $\hat{S}$. The base cases will be when $|\hat{S}| \in \{0,1,2\}$, that is to say when the first player is able to take the central vertex of $\hat{S}$. Each of these cases corresponds to a subdivided star, whose Grundy value is given in Figure~\ref{fig:tabgrun}.
In the inductive part, we need to prove that for every move on $\hat{S}\bipa S + \hat{S} \bipa S'$ by the first player, the second player has a move leading to a $\outcomeP$-position. In every case, if the first player plays on $\hat{S}$, then the second player can replicate the move, allowing us to invoke the induction hypothesis.
Thus, we will only consider the moves on $S$ or $S'$ in each case.

\noindent\textbf{Case 1 :} $P_1 \sim_1$ \sstar{2,1}

Figure~\ref{fig:pod0EQUIV1pod21} shows the possible moves on $P_1$ or \sstar{2,1}, and the answer leading to a $\outcomeP$-position (for readability, we write $S$ instead of $\hat{S}$ in the figure).

\begin{figure}[!h]
\centering
\begin{tikzpicture}
	\node (orig) at (0,0) {\fbox{
	\begin{tikzpicture}[scale=0.5]
		\node[noeud] (1) at (-1,0) {};
		\node[noeud] (2) at (-2,0) {};
		\draw (-2,0) -- (-2.75,0.5);
		\draw (-2,0) -- (-2.75,-0.5);
		\draw (-2.75,0.5) arc (135:225:0.7);
		\node (2b) at (-2.65,0) {$S$};
		\draw (1) -- (2);
		
		\draw (-0.4,0) node {+};
		
		\draw (1,0) -- (0.25,0.5);
		\draw (1,0) -- (0.25,-0.5);
		\draw (0.25,0.5) arc (135:225:0.7);
		\node (3b) at (0.35,0) {$S$};
		\node[noeud] (3) at (1,0) {};
		\node[noeud] (4) at (2,0) {};
		\node[noeud] (6) at (3,1) {};
		\node[noeud] (5) at (3,-1) {};
		\node[noeud] (7) at (4,1) {};
		\draw (3) -- (4);
		\draw (4) to (6);
		\draw (4) to (5);
		\draw (6) -- (7);
	\end{tikzpicture}}
	};
	
	\node (removeOnePod0) at (4.5,0) {
	\begin{tikzpicture}[scale=0.5]
		\node[noeud] (2) at (-2,0) {};
		\draw (-2,0) -- (-2.75,0.5);
		\draw (-2,0) -- (-2.75,-0.5);
		\draw (-2.75,0.5) arc (135:225:0.7);
		\node (2b) at (-2.65,0) {$S$};
		
		\draw (-1,0) node {+};
		
		\draw (1,0) -- (0.25,0.5);
		\draw (1,0) -- (0.25,-0.5);
		\draw (0.25,0.5) arc (135:225:0.7);
		\node (3b) at (0.35,0) {$S$};
		\node[noeud] (3) at (1,0) {};
		\node[noeud] (4) at (2,0) {};
		\node[noeud] (6) at (3,1) {};
		\node[noeud] (5) at (3,-1) {};
		\node[noeud] (7) at (4,1) {};
		\draw (3) -- (4);
		\draw (4) to (6);
		\draw (4) to (5);
		\draw (6) -- (7);
	\end{tikzpicture}
	};
	\node (removeOnePod0B) at (9,0) {
	\begin{tikzpicture}[scale=0.5]
		\node[noeud] (2) at (-2,0) {};
		\draw (-2,0) -- (-2.75,0.5);
		\draw (-2,0) -- (-2.75,-0.5);
		\draw (-2.75,0.5) arc (135:225:0.7);
		\node (2b) at (-2.65,0) {$S$};
		
		\draw (-1,0) node {+};
		
		\draw (1,0) -- (0.25,0.5);
		\draw (1,0) -- (0.25,-0.5);
		\draw (0.25,0.5) arc (135:225:0.7);
		\node (3b) at (0.35,0) {$S$};
		\node[noeud] (3) at (1,0) {};
		\node[noeud] (4) at (2,0) {};
		\node[noeud] (6) at (3,1) {};
		\node (5) at (3,-1) {};
		\node[noeud] (7) at (4,1) {};
		\draw (3) -- (4);
		\draw (4) to (6);
		\draw (6) -- (7);
	\end{tikzpicture}
	};
	\draw (13,0.25) node {$\outcomeP$ by};
	\draw (13,-0.25) node {$\emptyset \sim_1 P_3$};
	
	\node (playToPod2) at (4.5,-1.5) {
	\begin{tikzpicture}[scale=0.5]
		\node[noeud] (1) at (-1,0) {};
		\node[noeud] (2) at (-2,0) {};
		\draw (-2,0) -- (-2.75,0.5);
		\draw (-2,0) -- (-2.75,-0.5);
		\draw (-2.75,0.5) arc (135:225:0.7);
		\node (2b) at (-2.65,0) {$S$};
		\draw (1) -- (2);
		
		\draw (-0.4,0) node {+};
		
		\draw (1,0) -- (0.25,0.5);
		\draw (1,0) -- (0.25,-0.5);
		\draw (0.25,0.5) arc (135:225:0.7);
		\node (3b) at (0.35,0) {$S$};
		\node[noeud] (3) at (1,0) {};
		\node[noeud] (4) at (2,0) {};
		\node[noeud] (6) at (3,1) {};
		\node (5) at (3,-1) {};
		\node[noeud] (7) at (4,1) {};
		\draw (3) -- (4);
		\draw (4) to (6);
		\draw (6) -- (7);
	\end{tikzpicture}
	};
	\node (playToPod1) at (4.5,-3) {
	\begin{tikzpicture}[scale=0.5]
		\node[noeud] (1) at (-1,0) {};
		\node[noeud] (2) at (-2,0) {};
		\draw (-2,0) -- (-2.75,0.5);
		\draw (-2,0) -- (-2.75,-0.5);
		\draw (-2.75,0.5) arc (135:225:0.7);
		\node (2b) at (-2.65,0) {$S$};
		\draw (1) -- (2);
		
		\draw (-0.4,0) node {+};
		
		\draw (1,0) -- (0.25,0.5);
		\draw (1,0) -- (0.25,-0.5);
		\draw (0.25,0.5) arc (135:225:0.7);
		\node (3b) at (0.35,0) {$S$};
		\node[noeud] (3) at (1,0) {};
		\node[noeud] (4) at (2,0) {};
		\node (6) at (3,1) {};
		\node[noeud] (5) at (3,-1) {};
		\node (7) at (4,1) {};
		\draw (3) -- (4);
		\draw (4) to (5);
	\end{tikzpicture}
	};
	\node (playToPod0) at (9,-2.25) {
	\begin{tikzpicture}[scale=0.5]
		\node[noeud] (1) at (-1,0) {};
		\node[noeud] (2) at (-2,0) {};
		\draw (-2,0) -- (-2.75,0.5);
		\draw (-2,0) -- (-2.75,-0.5);
		\draw (-2.75,0.5) arc (135:225:0.7);
		\node (2b) at (-2.65,0) {$S$};
		\draw (1) -- (2);
		
		\draw (-0.4,0) node {+};
		
		\draw (1,0) -- (0.25,0.5);
		\draw (1,0) -- (0.25,-0.5);
		\draw (0.25,0.5) arc (135:225:0.7);
		\node (3b) at (0.35,0) {$S$};
		\node[noeud] (3) at (1,0) {};
		\node[noeud] (4) at (2,0) {};
		\node (6) at (3,1) {};
		\node (5) at (3,-1) {};
		\node (7) at (4,1) {};
		\draw (3) -- (4);
	\end{tikzpicture}
	};
	\draw (13,-2.25) node {$\outcomeP$};
	
	\node (removeOnePod21) at (4.5,-4.5) {
	\begin{tikzpicture}[scale=0.5]
		\node[noeud] (1) at (-1,0) {};
		\node[noeud] (2) at (-2,0) {};
		\draw (-2,0) -- (-2.75,0.5);
		\draw (-2,0) -- (-2.75,-0.5);
		\draw (-2.75,0.5) arc (135:225:0.7);
		\node (2b) at (-2.65,0) {$S$};
		\draw (1) -- (2);
		
		\draw (-0.4,0) node {+};
		
		\draw (1,0) -- (0.25,0.5);
		\draw (1,0) -- (0.25,-0.5);
		\draw (0.25,0.5) arc (135:225:0.7);
		\node (3b) at (0.35,0) {$S$};
		\node[noeud] (3) at (1,0) {};
		\node[noeud] (4) at (2,0) {};
		\node[noeud] (6) at (3,1) {};
		\node[noeud] (5) at (3,-1) {};
		\node (7) at (4,1) {};
		\draw (3) -- (4);
		\draw (4) to (6);
		\draw (4) to (5);
	\end{tikzpicture}
	};
	\node (removeOnePod21B) at (9,-4.5) {
	\begin{tikzpicture}[scale=0.5]
		\node[noeud] (2) at (-2,0) {};
		\draw (-2,0) -- (-2.75,0.5);
		\draw (-2,0) -- (-2.75,-0.5);
		\draw (-2.75,0.5) arc (135:225:0.7);
		\node (2b) at (-2.65,0) {$S$};
		
		\draw (-1,0) node {+};
		
		\draw (1,0) -- (0.25,0.5);
		\draw (1,0) -- (0.25,-0.5);
		\draw (0.25,0.5) arc (135:225:0.7);
		\node (3b) at (0.35,0) {$S$};
		\node[noeud] (3) at (1,0) {};
		\node[noeud] (4) at (2,0) {};
		\node[noeud] (6) at (3,1) {};
		\node[noeud] (5) at (3,-1) {};
		\node (7) at (4,1) {};
		\draw (3) -- (4);
		\draw (4) to (6);
		\draw (4) to (5);
	\end{tikzpicture}
	};
	\draw (13,-4.25) node {$\outcomeP$ by};
	\draw (13,-4.75) node {$\emptyset \sim_1$ \sstar{1,1}};
	
	\node (playToPm) at (4.5,-6) {
	\begin{tikzpicture}[scale=0.5]
		\node (2b) at (-2,0) {$P_{m-1}$};
		\draw (-0.6,0) node {+};
		
		\draw (1,0) -- (0.25,0.5);
		\draw (1,0) -- (0.25,-0.5);
		\draw (0.25,0.5) arc (135:225:0.7);
		\node (3b) at (0.4,0) {$S$};
		\node[noeud] (3) at (1,0) {};
		\node[noeud] (4) at (2,0) {};
		\node[noeud] (6) at (3,1) {};
		\node[noeud] (5) at (3,-1) {};
		\node[noeud] (7) at (4,1) {};
		\draw (3) -- (4);
		\draw (4) to (6);
		\draw (4) to (5);
		\draw (6) -- (7);
	\end{tikzpicture}
	};
	\draw (4.5,-7) node {(if the first player can};
	\draw (4.5,-7.5) node {remove the central vertex,};
	\draw (4.5,-8) node {then $S=P_m$ with $m \geq 3$)};
	\node (playToPmB) at (9,-6) {
	\begin{tikzpicture}[scale=0.5]
		\node (2b) at (-2,0) {$P_{m-1}$};
		\draw (-0.75,0) node {+};
		
		\node (3b) at (0.5,0) {$P_{m+2}$};
	\end{tikzpicture}
	};
	\draw (13,-6) node {$\outcomeP$};
	
	\draw (orig) -- (2.25,0) [->] (2.25,0) -- (removeOnePod0);
	\draw[->] (removeOnePod0) -- (removeOnePod0B);
	\draw (orig) -- (2.25,0) [->] (2.25,0) |- (playToPod1);
	\draw (orig) -- (2.25,0) [->] (2.25,0) |- (playToPod2);
	\draw[->] (playToPod1) -- (playToPod0);
	\draw[->] (playToPod2) -- (playToPod0);
	\draw (orig) -- (2.25,0) [->] (2.25,0) |- (removeOnePod21);
	\draw[->] (removeOnePod21) -- (removeOnePod21B);
	\draw (orig) -- (2.25,0) [->] (2.25,0) |- (playToPm);
	\draw[->] (playToPm) -- (playToPmB);
\end{tikzpicture}
\caption{The inductive part of the proof that $P_1 \sim_1$\sstar{2,1}.}
\label{fig:pod0EQUIV1pod21}
\end{figure}
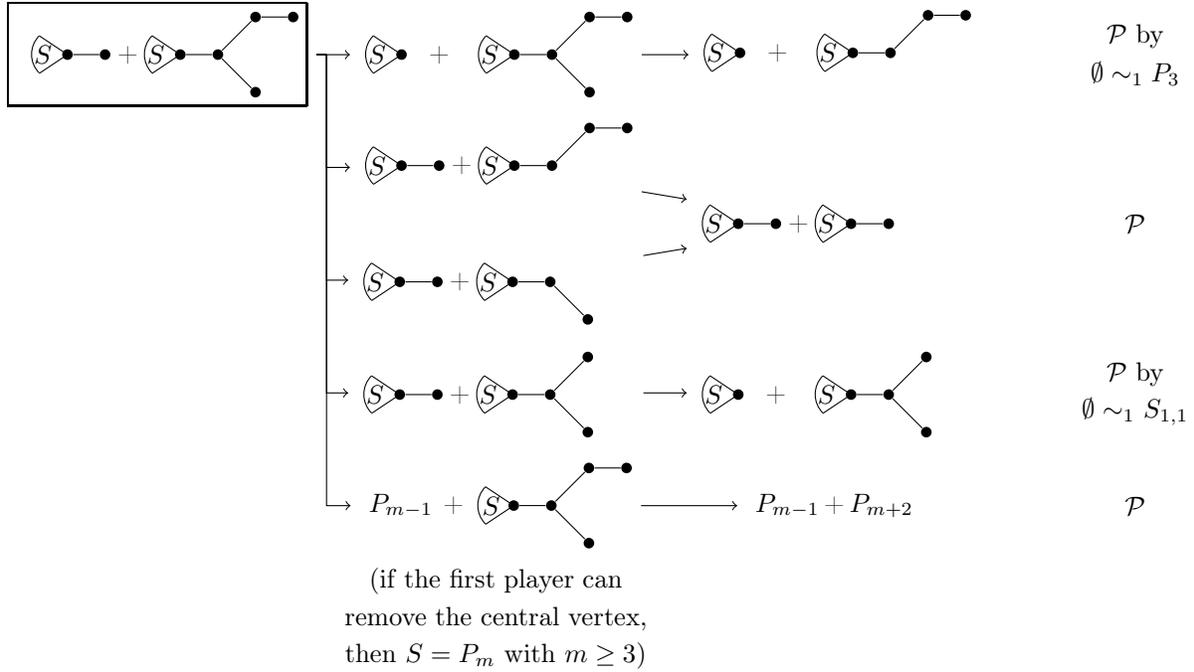

\noindent\textbf{Case 2 :} $P_2 \sim_1$ \sstar{2,2}

Figure~\ref{fig:pod1EQUIV1pod22} shows the possible moves on $P_2$ or \sstar{2,2}, and the answer leading to a $\outcomeP$-position (for readability, we write $S$ instead of $\hat{S}$ in the figure).

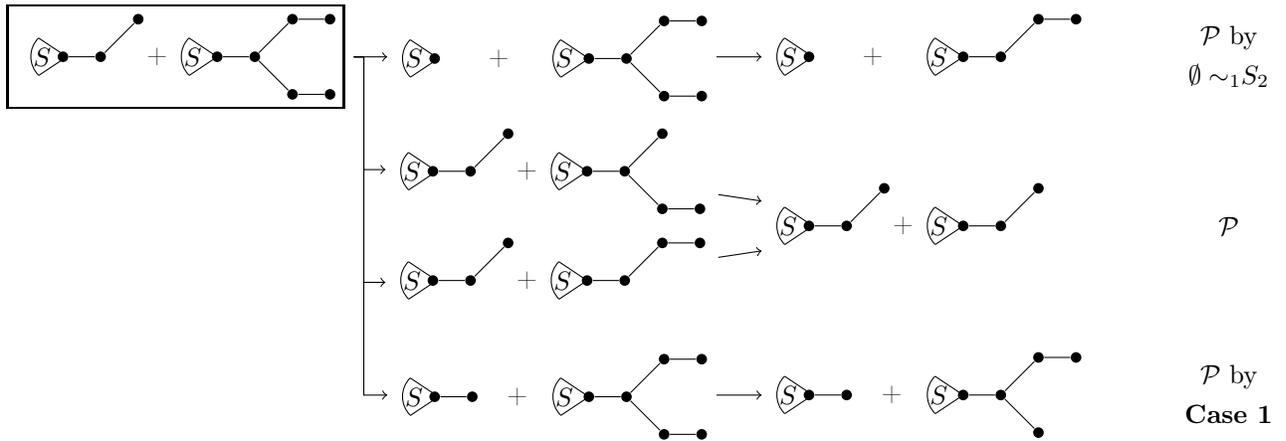
\begin{figure}[!h]
\centering
\begin{tikzpicture}
	\node (orig) at (0,0) {\fbox{
	\begin{tikzpicture}[scale=0.5]
		\node[noeud] (1) at (-1.1,1) {};
		\node[noeud] (2) at (-2.1,0) {};
		\node[noeud] (2b) at (-3.1,0) {};
		\node (2c) at (-3.65,0) {$S$};
		\draw (-3,0) -- (-3.75,0.5);
		\draw (-3,0) -- (-3.75,-0.5);
		\draw (-3.75,0.5) arc (135:225:0.7);
		\draw (1) -- (2);
		\draw (2) -- (2b);
		
		\draw (-0.6,0) node {+};
		
		\draw (1,0) -- (0.25,0.5);
		\draw (1,0) -- (0.25,-0.5);
		\draw (0.25,0.5) arc (135:225:0.7);
		\node (3b) at (0.35,0) {$S$};
		\node[noeud] (3) at (1,0) {};
		\node[noeud] (4) at (2,0) {};
		\node[noeud] (6) at (3,1) {};
		\node[noeud] (5) at (3,-1) {};
		\node[noeud] (5b) at (4,-1) {};
		\node[noeud] (7) at (4,1) {};
		\draw (3) -- (4);
		\draw (4) to (6);
		\draw (4) to (5);
		\draw (6) -- (7);
		\draw (5) -- (5b);
	\end{tikzpicture}}
	};
	
	\node (removeTwoPod1) at (5,0) {
	\begin{tikzpicture}[scale=0.5]
		\node (1) at (-1.1,1) {};
		\node (2) at (-2.1,0) {};
		\node[noeud] (2b) at (-3.1,0) {};
		\node (2c) at (-3.65,0) {$S$};
		\draw (-3,0) -- (-3.75,0.5);
		\draw (-3,0) -- (-3.75,-0.5);
		\draw (-3.75,0.5) arc (135:225:0.7);
		
		\draw (-1.4,0) node {+};
		
		\draw (1,0) -- (0.25,0.5);
		\draw (1,0) -- (0.25,-0.5);
		\draw (0.25,0.5) arc (135:225:0.7);
		\node (3b) at (0.35,0) {$S$};
		\node[noeud] (3) at (1,0) {};
		\node[noeud] (4) at (2,0) {};
		\node[noeud] (6) at (3,1) {};
		\node[noeud] (5) at (3,-1) {};
		\node[noeud] (5b) at (4,-1) {};
		\node[noeud] (7) at (4,1) {};
		\draw (3) -- (4);
		\draw (4) to (6);
		\draw (4) to (5);
		\draw (6) -- (7);
		\draw (5) -- (5b);
	\end{tikzpicture}
	};
	\node (removeTwoPod1B) at (10,0) {
	\begin{tikzpicture}[scale=0.5]
		\node (1) at (-1.1,1) {};
		\node (2) at (-2.1,0) {};
		\node[noeud] (2b) at (-3.1,0) {};
		\node (2c) at (-3.65,0) {$S$};
		\draw (-3,0) -- (-3.75,0.5);
		\draw (-3,0) -- (-3.75,-0.5);
		\draw (-3.75,0.5) arc (135:225:0.7);
		
		\draw (-1.4,0) node {+};
		
		\draw (1,0) -- (0.25,0.5);
		\draw (1,0) -- (0.25,-0.5);
		\draw (0.25,0.5) arc (135:225:0.7);
		\node (3b) at (0.35,0) {$S$};
		\node[noeud] (3) at (1,0) {};
		\node[noeud] (4) at (2,0) {};
		\node[noeud] (6) at (3,1) {};
		\node (5) at (3,-1) {};
		\node (5b) at (4,-1) {};
		\node[noeud] (7) at (4,1) {};
		\draw (3) -- (4);
		\draw (4) to (6);
		\draw (6) -- (7);
	\end{tikzpicture}
	};
	\draw (14,0.25) node {$\outcomeP$ by};
	\draw (14,-0.25) node {$\emptyset \sim_1$\sstar{2}};
	
	\node (removeOnePod22) at (5,-1.5) {
	\begin{tikzpicture}[scale=0.5]
		\node[noeud] (1) at (-1.1,1) {};
		\node[noeud] (2) at (-2.1,0) {};
		\node[noeud] (2b) at (-3.1,0) {};
		\node (2c) at (-3.65,0) {$S$};
		\draw (-3,0) -- (-3.75,0.5);
		\draw (-3,0) -- (-3.75,-0.5);
		\draw (-3.75,0.5) arc (135:225:0.7);
		\draw (1) -- (2);
		\draw (2) -- (2b);
		
		\draw (-0.6,0) node {+};
		
		\draw (1,0) -- (0.25,0.5);
		\draw (1,0) -- (0.25,-0.5);
		\draw (0.25,0.5) arc (135:225:0.7);
		\node (3b) at (0.35,0) {$S$};
		\node[noeud] (3) at (1,0) {};
		\node[noeud] (4) at (2,0) {};
		\node[noeud] (6) at (3,1) {};
		\node[noeud] (5) at (3,-1) {};
		\node[noeud] (5b) at (4,-1) {};
		\node (7) at (4,1) {};
		\draw (3) -- (4);
		\draw (4) to (6);
		\draw (4) to (5);
		\draw (5) -- (5b);
	\end{tikzpicture}
	};
	\node (removeTwoPod22) at (5,-3) {
	\begin{tikzpicture}[scale=0.5]
		\node[noeud] (1) at (-1.1,1) {};
		\node[noeud] (2) at (-2.1,0) {};
		\node[noeud] (2b) at (-3.1,0) {};
		\node (2c) at (-3.65,0) {$S$};
		\draw (-3,0) -- (-3.75,0.5);
		\draw (-3,0) -- (-3.75,-0.5);
		\draw (-3.75,0.5) arc (135:225:0.7);
		\draw (1) -- (2);
		\draw (2) -- (2b);
		
		\draw (-0.6,0) node {+};
		
		\draw (1,0) -- (0.25,0.5);
		\draw (1,0) -- (0.25,-0.5);
		\draw (0.25,0.5) arc (135:225:0.7);
		\node (3b) at (0.35,0) {$S$};
		\node[noeud] (3) at (1,0) {};
		\node[noeud] (4) at (2,0) {};
		\node[noeud] (6) at (3,1) {};
		\node (5) at (3,-1) {};
		\node (5b) at (4,-1) {};
		\node[noeud] (7) at (4,1) {};
		\draw (3) -- (4);
		\draw (4) to (6);
		\draw (6) -- (7);
	\end{tikzpicture}
	};
	\node (removePod22) at (10,-2.25) {
	\begin{tikzpicture}[scale=0.5]
		\node[noeud] (1) at (-1.1,1) {};
		\node[noeud] (2) at (-2.1,0) {};
		\node[noeud] (2b) at (-3.1,0) {};
		\node (2c) at (-3.65,0) {$S$};
		\draw (-3,0) -- (-3.75,0.5);
		\draw (-3,0) -- (-3.75,-0.5);
		\draw (-3.75,0.5) arc (135:225:0.7);
		\draw (1) -- (2);
		\draw (2) -- (2b);
		
		\draw (-0.6,0) node {+};
		
		\draw (1,0) -- (0.25,0.5);
		\draw (1,0) -- (0.25,-0.5);
		\draw (0.25,0.5) arc (135:225:0.7);
		\node (3b) at (0.35,0) {$S$};
		\node[noeud] (3) at (1,0) {};
		\node[noeud] (4) at (2,0) {};
		\node[noeud] (6) at (3,1) {};
		\node (5) at (3,-1) {};
		\node (5b) at (4,-1) {};
		\node (7) at (4,1) {};
		\draw (3) -- (4);
		\draw (4) to (6);
	\end{tikzpicture}
	};
	\draw (14,-2.25) node {$\outcomeP$};
	
	\node (removeOnePod1) at (5,-4.5) {
	\begin{tikzpicture}[scale=0.5]
		\node (1) at (-1.1,1) {};
		\node[noeud] (2) at (-2.1,0) {};
		\node[noeud] (2b) at (-3.1,0) {};
		\node (2c) at (-3.65,0) {$S$};
		\draw (-3,0) -- (-3.75,0.5);
		\draw (-3,0) -- (-3.75,-0.5);
		\draw (-3.75,0.5) arc (135:225:0.7);
		\draw (2) -- (2b);
		
		\draw (-0.9,0) node {+};
		
		\draw (1,0) -- (0.25,0.5);
		\draw (1,0) -- (0.25,-0.5);
		\draw (0.25,0.5) arc (135:225:0.7);
		\node (3b) at (0.35,0) {$S$};
		\node[noeud] (3) at (1,0) {};
		\node[noeud] (4) at (2,0) {};
		\node[noeud] (6) at (3,1) {};
		\node[noeud] (5) at (3,-1) {};
		\node[noeud] (5b) at (4,-1) {};
		\node[noeud] (7) at (4,1) {};
		\draw (3) -- (4);
		\draw (4) to (6);
		\draw (4) to (5);
		\draw (6) -- (7);
		\draw (5) -- (5b);
	\end{tikzpicture}
	};
	\node (removeOnePod1B) at (10,-4.5) {
	\begin{tikzpicture}[scale=0.5]
		\node (1) at (-1.1,1) {};
		\node[noeud] (2) at (-2.1,0) {};
		\node[noeud] (2b) at (-3.1,0) {};
		\node (2c) at (-3.65,0) {$S$};
		\draw (-3,0) -- (-3.75,0.5);
		\draw (-3,0) -- (-3.75,-0.5);
		\draw (-3.75,0.5) arc (135:225:0.7);
		\draw (2) -- (2b);
		
		\draw (-0.9,0) node {+};
		
		\draw (1,0) -- (0.25,0.5);
		\draw (1,0) -- (0.25,-0.5);
		\draw (0.25,0.5) arc (135:225:0.7);
		\node (3b) at (0.35,0) {$S$};
		\node[noeud] (3) at (1,0) {};
		\node[noeud] (4) at (2,0) {};
		\node[noeud] (6) at (3,1) {};
		\node[noeud] (5) at (3,-1) {};
		\node (5b) at (4,-1) {};
		\node[noeud] (7) at (4,1) {};
		\draw (3) -- (4);
		\draw (4) to (6);
		\draw (4) to (5);
		\draw (6) -- (7);
	\end{tikzpicture}
	};
	\draw (14,-4.25) node {$\outcomeP$ by};
	\draw (14,-4.75) node {\textbf{Case 1}};
	
	\draw (orig) -- (2.5,0) [->] (2.5,0) -- (removeTwoPod1);
	\draw[->] (removeTwoPod1) -- (removeTwoPod1B);
	\draw (orig) -- (2.5,0) [->] (2.5,0) |- (removeOnePod22);
	\draw (orig) -- (2.5,0) [->] (2.5,0) |- (removeTwoPod22);
	\draw[->] (removeOnePod22) -- (removePod22);
	\draw[->] (removeTwoPod22) -- (removePod22);
	\draw (orig) -- (2.5,0) [->] (2.5,0) |- (removeOnePod1);
	\draw[->] (removeOnePod1) -- (removeOnePod1B);
\end{tikzpicture}
\caption{The inductive part of the proof that $P_2 \sim_1$ \sstar{2,2}.}
\label{fig:pod1EQUIV1pod22}
\end{figure}

\noindent\textbf{Case 3 :} \sstar{1,1} $\sim_1$ \sstar{2,2,1}

Figure~\ref{fig:pod11EQUIV1pod221} shows the possible moves on \sstar{1,1} or \sstar{2,2,1}, and the answer leading to a $\outcomeP$-position (for readability, we write $S$ instead of $\hat{S}$ in the figure).

\begin{figure}[!h]
\centering
\begin{tikzpicture}
	\node (orig) at (0,0) {\fbox{
	\begin{tikzpicture}[scale=0.5]
		\node[noeud] (1) at (-1.1,1) {};
		\node[noeud] (1b) at (-1.1,-1) {};
		\node[noeud] (2) at (-2.1,0) {};
		\node[noeud] (2b) at (-3.1,0) {};
		\node (2c) at (-3.65,0) {$S$};
		\draw (-3,0) -- (-3.75,0.5);
		\draw (-3,0) -- (-3.75,-0.5);
		\draw (-3.75,0.5) arc (135:225:0.7);
		\draw (1) -- (2);
		\draw (2) -- (2b);
		\draw (2) -- (1b);
		
		\draw (-0.6,0) node {+};
		
		\draw (1,0) -- (0.25,0.5);
		\draw (1,0) -- (0.25,-0.5);
		\draw (0.25,0.5) arc (135:225:0.7);
		\node (3b) at (0.35,0) {$S$};
		\node[noeud] (3) at (1,0) {};
		\node[noeud] (4) at (2,0) {};
		\node[noeud] (6) at (3,1) {};
		\node[noeud] (5) at (3,-1) {};
		\node[noeud] (5b) at (4,-1) {};
		\node[noeud] (7) at (4,1) {};
		\node[noeud] (8) at (3,0) {};
		\draw (3) -- (4);
		\draw (8) -- (4);
		\draw (4) to (6);
		\draw (4) to (5);
		\draw (6) -- (7);
		\draw (5) -- (5b);
	\end{tikzpicture}}
	};
	
	\node (removeTwoPod1) at (5,0) {
	\begin{tikzpicture}[scale=0.5]
		\node[noeud] (1) at (-1.1,1) {};
		\node[noeud] (1b) at (-1.1,-1) {};
		\node[noeud] (2) at (-2.1,0) {};
		\node[noeud] (2b) at (-3.1,0) {};
		\node (2c) at (-3.65,0) {$S$};
		\draw (-3,0) -- (-3.75,0.5);
		\draw (-3,0) -- (-3.75,-0.5);
		\draw (-3.75,0.5) arc (135:225:0.7);
		\draw (1) -- (2);
		\draw (2) -- (2b);
		\draw (2) -- (1b);
		
		\draw (-0.6,0) node {+};
		
		\draw (1,0) -- (0.25,0.5);
		\draw (1,0) -- (0.25,-0.5);
		\draw (0.25,0.5) arc (135:225:0.7);
		\node (3b) at (0.35,0) {$S$};
		\node[noeud] (3) at (1,0) {};
		\node[noeud] (4) at (2,0) {};
		\node[noeud] (6) at (3,1) {};
		\node[noeud] (5) at (3,-1) {};
		\node[noeud] (5b) at (4,-1) {};
		\node[noeud] (7) at (4,1) {};
		\node (8) at (3,0) {};
		\draw (3) -- (4);
		\draw (4) to (6);
		\draw (4) to (5);
		\draw (6) -- (7);
		\draw (5) -- (5b);
	\end{tikzpicture}
	};
	\node (removeTwoPod1B) at (10,0) {
	\begin{tikzpicture}[scale=0.5]
		\node[noeud] (1) at (-1.1,1) {};
		\node[noeud] (1b) at (-1.1,-1) {};
		\node[noeud] (2) at (-2.1,0) {};
		\node[noeud] (2b) at (-3.1,0) {};
		\node (2c) at (-3.65,0) {$S$};
		\draw (-3,0) -- (-3.75,0.5);
		\draw (-3,0) -- (-3.75,-0.5);
		\draw (-3.75,0.5) arc (135:225:0.7);
		\draw (1) -- (2);
		\draw (2) -- (2b);
		\draw (2) -- (1b);
		
		\draw (-0.6,0) node {+};
		
		\draw (1,0) -- (0.25,0.5);
		\draw (1,0) -- (0.25,-0.5);
		\draw (0.25,0.5) arc (135:225:0.7);
		\node (3b) at (0.35,0) {$S$};
		\node[noeud] (3) at (1,0) {};
		\node[noeud] (4) at (2,0) {};
		\node[noeud] (6) at (3,1) {};
		\node (5) at (3,-1) {};
		\node (5b) at (4,-1) {};
		\node[noeud] (7) at (4,1) {};
		\node (8) at (3,0) {};
		\draw (3) -- (4);
		\draw (4) to (6);
		\draw (6) -- (7);
	\end{tikzpicture}
	};
	\draw (14,0.25) node {$\outcomeP$ by};
	\draw (14,-0.25) node {\sstar{1,1}$\sim_1$\sstar{2}};
	
	\node (removeOnePod22) at (5,-1.5) {
	\begin{tikzpicture}[scale=0.5]
		\node[noeud] (1) at (-1.1,1) {};
		\node[noeud] (1b) at (-1.1,-1) {};
		\node[noeud] (2) at (-2.1,0) {};
		\node[noeud] (2b) at (-3.1,0) {};
		\node (2c) at (-3.65,0) {$S$};
		\draw (-3,0) -- (-3.75,0.5);
		\draw (-3,0) -- (-3.75,-0.5);
		\draw (-3.75,0.5) arc (135:225:0.7);
		\draw (1) -- (2);
		\draw (2) -- (2b);
		\draw (2) -- (1b);
		
		\draw (-0.6,0) node {+};
		
		\draw (1,0) -- (0.25,0.5);
		\draw (1,0) -- (0.25,-0.5);
		\draw (0.25,0.5) arc (135:225:0.7);
		\node (3b) at (0.35,0) {$S$};
		\node[noeud] (3) at (1,0) {};
		\node[noeud] (4) at (2,0) {};
		\node[noeud] (6) at (3,1) {};
		\node[noeud] (5) at (3,-1) {};
		\node[noeud] (5b) at (4,-1) {};
		\node (7) at (4,1) {};
		\node[noeud] (8) at (3,0) {};
		\draw (3) -- (4);
		\draw (8) -- (4);
		\draw (4) to (6);
		\draw (4) to (5);
		\draw (5) -- (5b);
	\end{tikzpicture}
	};
	\node (removeTwoPod22) at (5,-3) {
	\begin{tikzpicture}[scale=0.5]
		\node[noeud] (1) at (-1.1,1) {};
		\node[noeud] (1b) at (-1.1,-1) {};
		\node[noeud] (2) at (-2.1,0) {};
		\node[noeud] (2b) at (-3.1,0) {};
		\node (2c) at (-3.65,0) {$S$};
		\draw (-3,0) -- (-3.75,0.5);
		\draw (-3,0) -- (-3.75,-0.5);
		\draw (-3.75,0.5) arc (135:225:0.7);
		\draw (1) -- (2);
		\draw (2) -- (2b);
		\draw (2) -- (1b);
		
		\draw (-0.6,0) node {+};
		
		\draw (1,0) -- (0.25,0.5);
		\draw (1,0) -- (0.25,-0.5);
		\draw (0.25,0.5) arc (135:225:0.7);
		\node (3b) at (0.35,0) {$S$};
		\node[noeud] (3) at (1,0) {};
		\node[noeud] (4) at (2,0) {};
		\node[noeud] (6) at (3,1) {};
		\node (5) at (3,-1) {};
		\node (5b) at (4,-1) {};
		\node[noeud] (7) at (4,1) {};
		\node[noeud] (8) at (3,0) {};
		\draw (3) -- (4);
		\draw (8) -- (4);
		\draw (4) to (6);
		\draw (6) -- (7);
	\end{tikzpicture}
	};
	\node (removePod22) at (10,-2.25) {
	\begin{tikzpicture}[scale=0.5]
		\node[noeud] (1) at (-1.1,1) {};
		\node[noeud] (1b) at (-1.1,-1) {};
		\node[noeud] (2) at (-2.1,0) {};
		\node[noeud] (2b) at (-3.1,0) {};
		\node (2c) at (-3.65,0) {$S$};
		\draw (-3,0) -- (-3.75,0.5);
		\draw (-3,0) -- (-3.75,-0.5);
		\draw (-3.75,0.5) arc (135:225:0.7);
		\draw (1) -- (2);
		\draw (2) -- (2b);
		\draw (2) -- (1b);
		
		\draw (-0.6,0) node {+};
		
		\draw (1,0) -- (0.25,0.5);
		\draw (1,0) -- (0.25,-0.5);
		\draw (0.25,0.5) arc (135:225:0.7);
		\node (3b) at (0.35,0) {$S$};
		\node[noeud] (3) at (1,0) {};
		\node[noeud] (4) at (2,0) {};
		\node[noeud] (6) at (3,1) {};
		\node (5) at (3,-1) {};
		\node (5b) at (4,-1) {};
		\node (7) at (4,1) {};
		\node[noeud] (8) at (3,0) {};
		\draw (3) -- (4);
		\draw (8) -- (4);
		\draw (4) to (6);
	\end{tikzpicture}
	};
	\draw (14,-2.25) node {$\outcomeP$};
	
	\node (removeOnePod1) at (5,-4.5) {
	\begin{tikzpicture}[scale=0.5]
		\node[noeud] (1) at (-1.1,1) {};
		\node (1b) at (-1.1,-1) {};
		\node[noeud] (2) at (-2.1,0) {};
		\node[noeud] (2b) at (-3.1,0) {};
		\node (2c) at (-3.65,0) {$S$};
		\draw (-3,0) -- (-3.75,0.5);
		\draw (-3,0) -- (-3.75,-0.5);
		\draw (-3.75,0.5) arc (135:225:0.7);
		\draw (1) -- (2);
		\draw (2) -- (2b);
		
		\draw (-0.6,0) node {+};
		
		\draw (1,0) -- (0.25,0.5);
		\draw (1,0) -- (0.25,-0.5);
		\draw (0.25,0.5) arc (135:225:0.7);
		\node (3b) at (0.35,0) {$S$};
		\node[noeud] (3) at (1,0) {};
		\node[noeud] (4) at (2,0) {};
		\node[noeud] (6) at (3,1) {};
		\node[noeud] (5) at (3,-1) {};
		\node[noeud] (5b) at (4,-1) {};
		\node[noeud] (7) at (4,1) {};
		\node[noeud] (8) at (3,0) {};
		\draw (3) -- (4);
		\draw (8) -- (4);
		\draw (4) to (6);
		\draw (4) to (5);
		\draw (6) -- (7);
		\draw (5) -- (5b);
	\end{tikzpicture}
	};
	\node (removeOnePod1B) at (10,-4.5) {
	\begin{tikzpicture}[scale=0.5]
		\node[noeud] (1) at (-1.1,1) {};
		\node (1b) at (-1.1,-1) {};
		\node[noeud] (2) at (-2.1,0) {};
		\node[noeud] (2b) at (-3.1,0) {};
		\node (2c) at (-3.65,0) {$S$};
		\draw (-3,0) -- (-3.75,0.5);
		\draw (-3,0) -- (-3.75,-0.5);
		\draw (-3.75,0.5) arc (135:225:0.7);
		\draw (1) -- (2);
		\draw (2) -- (2b);
		
		\draw (-0.6,0) node {+};
		
		\draw (1,0) -- (0.25,0.5);
		\draw (1,0) -- (0.25,-0.5);
		\draw (0.25,0.5) arc (135:225:0.7);
		\node (3b) at (0.35,0) {$S$};
		\node[noeud] (3) at (1,0) {};
		\node[noeud] (4) at (2,0) {};
		\node[noeud] (6) at (3,1) {};
		\node[noeud] (5) at (3,-1) {};
		\node[noeud] (5b) at (4,-1) {};
		\node[noeud] (7) at (4,1) {};
		\node (8) at (3,0) {};
		\draw (3) -- (4);
		\draw (4) to (6);
		\draw (4) to (5);
		\draw (6) -- (7);
		\draw (5) -- (5b);
	\end{tikzpicture}
	};
	\draw (14,-4.25) node {$\outcomeP$ by};
	\draw (14,-4.75) node {\textbf{Case 2}};
	
	\draw (orig) -- (2.5,0) [->] (2.5,0) |- (removeTwoPod1);
	\draw[->] (removeTwoPod1) -- (removeTwoPod1B);
	\draw (orig) -- (2.5,0) [->] (2.5,0) |- (removeOnePod22);
	\draw (orig) -- (2.5,0) [->] (2.5,0) |- (removeTwoPod22);
	\draw[->] (removeOnePod22) -- (removePod22);
	\draw[->] (removeTwoPod22) -- (removePod22);
	\draw (orig) -- (2.5,0) [->] (2.5,0) |- (removeOnePod1);
	\draw[->] (removeOnePod1) -- (removeOnePod1B);
\end{tikzpicture}
\caption{The inductive part of the proof that \sstar{1,1} $\sim_1$ \sstar{2,2,1}.}
\label{fig:pod11EQUIV1pod221}
\end{figure}
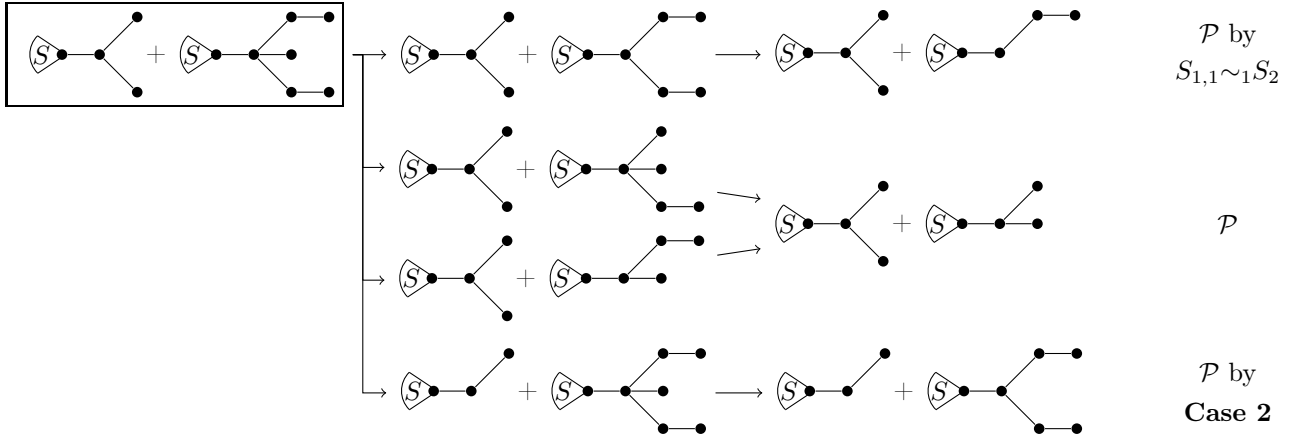

\noindent\textbf{Case 4 :} \sstar{2,1} $\sim_1$ \sstar{2,2,2}

Figure~\ref{fig:pod21EQUIV1pod222} shows the possible moves on \sstar{2,1} or \sstar{2,2,2}, and the answer leading to a $\outcomeP$-position (for readability, we write $S$ instead of $\hat{S}$ in the figure).

\begin{figure}[!h]
\centering
\begin{tikzpicture}
	\node (orig) at (0,0) {\fbox{
	\begin{tikzpicture}[scale=0.5]
		\node[noeud] (1) at (-2.1,1) {};
		\node[noeud] (1a) at (-1.1,1) {};
		\node[noeud] (1b) at (-2.1,-1) {};
		\node[noeud] (2) at (-3.1,0) {};
		\node[noeud] (2b) at (-4.1,0) {};
		\node (2c) at (-4.65,0) {$S$};
		\draw (-4,0) -- (-4.75,0.5);
		\draw (-4,0) -- (-4.75,-0.5);
		\draw (-4.75,0.5) arc (135:225:0.7);
		\draw (1) -- (2);
		\draw (2) -- (2b);
		\draw (2) -- (1b);
		\draw (1) -- (1a);
		
		\draw (-0.6,0) node {+};
		
		\draw (1,0) -- (0.25,0.5);
		\draw (1,0) -- (0.25,-0.5);
		\draw (0.25,0.5) arc (135:225:0.7);
		\node (3b) at (0.35,0) {$S$};
		\node[noeud] (3) at (1,0) {};
		\node[noeud] (4) at (2,0) {};
		\node[noeud] (6) at (3,1) {};
		\node[noeud] (5) at (3,-1) {};
		\node[noeud] (5b) at (4,-1) {};
		\node[noeud] (7) at (4,1) {};
		\node[noeud] (8) at (3,0) {};
		\node[noeud] (8b) at (4,0) {};
		\draw (3) -- (4);
		\draw (8) -- (4);
		\draw (8) -- (8b);
		\draw (4) to (6);
		\draw (4) to (5);
		\draw (6) -- (7);
		\draw (5) -- (5b);
	\end{tikzpicture}}
	};
	
	\node (removeTwoFrom21) at (5.5,0) {
	\begin{tikzpicture}[scale=0.5]
		\node (1) at (-2.1,1) {};
		\node (1a) at (-1.1,1) {};
		\node[noeud] (1b) at (-2.1,-1) {};
		\node[noeud] (2) at (-3.1,0) {};
		\node[noeud] (2b) at (-4.1,0) {};
		\node (2c) at (-4.65,0) {$S$};
		\draw (-4,0) -- (-4.75,0.5);
		\draw (-4,0) -- (-4.75,-0.5);
		\draw (-4.75,0.5) arc (135:225:0.7);
		\draw (2) -- (2b);
		\draw (2) -- (1b);
		
		\draw (-1.1,0) node {+};
		
		\draw (1,0) -- (0.25,0.5);
		\draw (1,0) -- (0.25,-0.5);
		\draw (0.25,0.5) arc (135:225:0.7);
		\node (3b) at (0.35,0) {$S$};
		\node[noeud] (3) at (1,0) {};
		\node[noeud] (4) at (2,0) {};
		\node[noeud] (6) at (3,1) {};
		\node[noeud] (5) at (3,-1) {};
		\node[noeud] (5b) at (4,-1) {};
		\node[noeud] (7) at (4,1) {};
		\node[noeud] (8) at (3,0) {};
		\node[noeud] (8b) at (4,0) {};
		\draw (3) -- (4);
		\draw (8) -- (4);
		\draw (8) -- (8b);
		\draw (4) to (6);
		\draw (4) to (5);
		\draw (6) -- (7);
		\draw (5) -- (5b);
	\end{tikzpicture}
	};
	\node (removeTwoFrom21B) at (11,0) {
	\begin{tikzpicture}[scale=0.5]
		\node (1) at (-2.1,1) {};
		\node (1a) at (-1.1,1) {};
		\node[noeud] (1b) at (-2.1,-1) {};
		\node[noeud] (2) at (-3.1,0) {};
		\node[noeud] (2b) at (-4.1,0) {};
		\node (2c) at (-4.65,0) {$S$};
		\draw (-4,0) -- (-4.75,0.5);
		\draw (-4,0) -- (-4.75,-0.5);
		\draw (-4.75,0.5) arc (135:225:0.7);
		\draw (2) -- (2b);
		\draw (2) -- (1b);
		
		\draw (-1.1,0) node {+};
		
		\draw (1,0) -- (0.25,0.5);
		\draw (1,0) -- (0.25,-0.5);
		\draw (0.25,0.5) arc (135:225:0.7);
		\node (3b) at (0.35,0) {$S$};
		\node[noeud] (3) at (1,0) {};
		\node[noeud] (4) at (2,0) {};
		\node[noeud] (6) at (3,1) {};
		\node[noeud] (5) at (3,-1) {};
		\node[noeud] (5b) at (4,-1) {};
		\node[noeud] (7) at (4,1) {};
		\draw (3) -- (4);
		\draw (4) to (6);
		\draw (4) to (5);
		\draw (6) -- (7);
		\draw (5) -- (5b);
	\end{tikzpicture}
	};
	\draw (14.5,0.25) node {$\outcomeP$ by};
	\draw (14.5,-0.25) node {\textbf{Case 2}};
	
	\node (removeOneFrom222) at (5.5,-1.5) {
	\begin{tikzpicture}[scale=0.5]
		\node[noeud] (1) at (-2.1,1) {};
		\node[noeud] (1a) at (-1.1,1) {};
		\node[noeud] (1b) at (-2.1,-1) {};
		\node[noeud] (2) at (-3.1,0) {};
		\node[noeud] (2b) at (-4.1,0) {};
		\node (2c) at (-4.65,0) {$S$};
		\draw (-4,0) -- (-4.75,0.5);
		\draw (-4,0) -- (-4.75,-0.5);
		\draw (-4.75,0.5) arc (135:225:0.7);
		\draw (1) -- (2);
		\draw (2) -- (2b);
		\draw (2) -- (1b);
		\draw (1) -- (1a);
		
		\draw (-0.6,0) node {+};
		
		\draw (1,0) -- (0.25,0.5);
		\draw (1,0) -- (0.25,-0.5);
		\draw (0.25,0.5) arc (135:225:0.7);
		\node (3b) at (0.35,0) {$S$};
		\node[noeud] (3) at (1,0) {};
		\node[noeud] (4) at (2,0) {};
		\node[noeud] (6) at (3,1) {};
		\node[noeud] (5) at (3,-1) {};
		\node (5b) at (4,-1) {};
		\node[noeud] (7) at (4,1) {};
		\node[noeud] (8) at (3,0) {};
		\node[noeud] (8b) at (4,0) {};
		\draw (3) -- (4);
		\draw (8) -- (4);
		\draw (8) -- (8b);
		\draw (4) to (6);
		\draw (4) to (5);
		\draw (6) -- (7);
	\end{tikzpicture}
	};
	\node (removeTwoFrom222) at (5.5,-3) {
	\begin{tikzpicture}[scale=0.5]
		\node[noeud] (1) at (-2.1,1) {};
		\node[noeud] (1a) at (-1.1,1) {};
		\node[noeud] (1b) at (-2.1,-1) {};
		\node[noeud] (2) at (-3.1,0) {};
		\node[noeud] (2b) at (-4.1,0) {};
		\node (2c) at (-4.65,0) {$S$};
		\draw (-4,0) -- (-4.75,0.5);
		\draw (-4,0) -- (-4.75,-0.5);
		\draw (-4.75,0.5) arc (135:225:0.7);
		\draw (1) -- (2);
		\draw (2) -- (2b);
		\draw (2) -- (1b);
		\draw (1) -- (1a);
		
		\draw (-0.6,0) node {+};
		
		\draw (1,0) -- (0.25,0.5);
		\draw (1,0) -- (0.25,-0.5);
		\draw (0.25,0.5) arc (135:225:0.7);
		\node (3b) at (0.35,0) {$S$};
		\node[noeud] (3) at (1,0) {};
		\node[noeud] (4) at (2,0) {};
		\node[noeud] (6) at (3,1) {};
		\node[noeud] (5) at (3,-1) {};
		\node[noeud] (5b) at (4,-1) {};
		\node[noeud] (7) at (4,1) {};
		\draw (3) -- (4);
		\draw (4) to (6);
		\draw (4) to (5);
		\draw (6) -- (7);
		\draw (5) -- (5b);
	\end{tikzpicture}
	};
	\node (removeFrom222) at (11,-2.25) {
	\begin{tikzpicture}[scale=0.5]
		\node[noeud] (1) at (-2.1,1) {};
		\node[noeud] (1a) at (-1.1,1) {};
		\node[noeud] (1b) at (-2.1,-1) {};
		\node[noeud] (2) at (-3.1,0) {};
		\node[noeud] (2b) at (-4.1,0) {};
		\node (2c) at (-4.65,0) {$S$};
		\draw (-4,0) -- (-4.75,0.5);
		\draw (-4,0) -- (-4.75,-0.5);
		\draw (-4.75,0.5) arc (135:225:0.7);
		\draw (1) -- (2);
		\draw (2) -- (2b);
		\draw (2) -- (1b);
		\draw (1) -- (1a);
		
		\draw (-0.6,0) node {+};
		
		\draw (1,0) -- (0.25,0.5);
		\draw (1,0) -- (0.25,-0.5);
		\draw (0.25,0.5) arc (135:225:0.7);
		\node (3b) at (0.35,0) {$S$};
		\node[noeud] (3) at (1,0) {};
		\node[noeud] (4) at (2,0) {};
		\node[noeud] (6) at (3,1) {};
		\node[noeud] (5) at (3,-1) {};
		\node[noeud] (7) at (4,1) {};
		\draw (3) -- (4);
		\draw (4) to (6);
		\draw (4) to (5);
		\draw (6) -- (7);
	\end{tikzpicture}
	};
	\draw (14.5,-2.25) node {$\outcomeP$};
	
	\node (playTo11) at (5.5,-4.5) {
	\begin{tikzpicture}[scale=0.5]
		\node[noeud] (1) at (-2.1,1) {};
		\node[noeud] (1b) at (-2.1,-1) {};
		\node[noeud] (2) at (-3.1,0) {};
		\node[noeud] (2b) at (-4.1,0) {};
		\node (2c) at (-4.65,0) {$S$};
		\draw (-4,0) -- (-4.75,0.5);
		\draw (-4,0) -- (-4.75,-0.5);
		\draw (-4.75,0.5) arc (135:225:0.7);
		\draw (1) -- (2);
		\draw (2) -- (2b);
		\draw (2) -- (1b);
		
		\draw (-0.85,0) node {+};
		
		\draw (1,0) -- (0.25,0.5);
		\draw (1,0) -- (0.25,-0.5);
		\draw (0.25,0.5) arc (135:225:0.7);
		\node (3b) at (0.35,0) {$S$};
		\node[noeud] (3) at (1,0) {};
		\node[noeud] (4) at (2,0) {};
		\node[noeud] (6) at (3,1) {};
		\node[noeud] (5) at (3,-1) {};
		\node[noeud] (5b) at (4,-1) {};
		\node[noeud] (7) at (4,1) {};
		\node[noeud] (8) at (3,0) {};
		\node[noeud] (8b) at (4,0) {};
		\draw (3) -- (4);
		\draw (8) -- (4);
		\draw (8) -- (8b);
		\draw (4) to (6);
		\draw (4) to (5);
		\draw (6) -- (7);
		\draw (5) -- (5b);
	\end{tikzpicture}
	};
	\node (playTo11B) at (11,-4.5) {
	\begin{tikzpicture}[scale=0.5]
		\node[noeud] (1) at (-2.1,1) {};
		\node[noeud] (1b) at (-2.1,-1) {};
		\node[noeud] (2) at (-3.1,0) {};
		\node[noeud] (2b) at (-4.1,0) {};
		\node (2c) at (-4.65,0) {$S$};
		\draw (-4,0) -- (-4.75,0.5);
		\draw (-4,0) -- (-4.75,-0.5);
		\draw (-4.75,0.5) arc (135:225:0.7);
		\draw (1) -- (2);
		\draw (2) -- (2b);
		\draw (2) -- (1b);
		
		\draw (-0.85,0) node {+};
		
		\draw (1,0) -- (0.25,0.5);
		\draw (1,0) -- (0.25,-0.5);
		\draw (0.25,0.5) arc (135:225:0.7);
		\node (3b) at (0.35,0) {$S$};
		\node[noeud] (3) at (1,0) {};
		\node[noeud] (4) at (2,0) {};
		\node[noeud] (6) at (3,1) {};
		\node[noeud] (5) at (3,-1) {};
		\node[noeud] (5b) at (4,-1) {};
		\node[noeud] (7) at (4,1) {};
		\node[noeud] (8) at (3,0) {};
		\draw (3) -- (4);
		\draw (8) -- (4);
		\draw (4) to (6);
		\draw (4) to (5);
		\draw (6) -- (7);
		\draw (5) -- (5b);
	\end{tikzpicture}
	};
	\draw (14.5,-4.25) node {$\outcomeP$ by};
	\draw (14.5,-4.75) node {\textbf{Case 3}};
	
	\node (playTo2) at (5.5,-6) {
	\begin{tikzpicture}[scale=0.5]
		\node[noeud] (1) at (-2.1,1) {};
		\node[noeud] (1a) at (-1.1,1) {};
		\node (1b) at (-2.1,-1) {};
		\node[noeud] (2) at (-3.1,0) {};
		\node[noeud] (2b) at (-4.1,0) {};
		\node (2c) at (-4.65,0) {$S$};
		\draw (-4,0) -- (-4.75,0.5);
		\draw (-4,0) -- (-4.75,-0.5);
		\draw (-4.75,0.5) arc (135:225:0.7);
		\draw (1) -- (2);
		\draw (2) -- (2b);
		\draw (1) -- (1a);
		
		\draw (-0.6,0) node {+};
		
		\draw (1,0) -- (0.25,0.5);
		\draw (1,0) -- (0.25,-0.5);
		\draw (0.25,0.5) arc (135:225:0.7);
		\node (3b) at (0.35,0) {$S$};
		\node[noeud] (3) at (1,0) {};
		\node[noeud] (4) at (2,0) {};
		\node[noeud] (6) at (3,1) {};
		\node[noeud] (5) at (3,-1) {};
		\node[noeud] (5b) at (4,-1) {};
		\node[noeud] (7) at (4,1) {};
		\node[noeud] (8) at (3,0) {};
		\node[noeud] (8b) at (4,0) {};
		\draw (3) -- (4);
		\draw (8) -- (4);
		\draw (8) -- (8b);
		\draw (4) to (6);
		\draw (4) to (5);
		\draw (6) -- (7);
		\draw (5) -- (5b);
	\end{tikzpicture}
	};
	\node (playTo2B) at (11,-6) {
	\begin{tikzpicture}[scale=0.5]
		\node[noeud] (1) at (-2.1,1) {};
		\node[noeud] (1a) at (-1.1,1) {};
		\node (1b) at (-2.1,-1) {};
		\node[noeud] (2) at (-3.1,0) {};
		\node[noeud] (2b) at (-4.1,0) {};
		\node (2c) at (-4.65,0) {$S$};
		\draw (-4,0) -- (-4.75,0.5);
		\draw (-4,0) -- (-4.75,-0.5);
		\draw (-4.75,0.5) arc (135:225:0.7);
		\draw (1) -- (2);
		\draw (2) -- (2b);
		\draw (1) -- (1a);
		
		\draw (-0.6,0) node {+};
		
		\draw (1,0) -- (0.25,0.5);
		\draw (1,0) -- (0.25,-0.5);
		\draw (0.25,0.5) arc (135:225:0.7);
		\node (3b) at (0.35,0) {$S$};
		\node[noeud] (3) at (1,0) {};
		\node[noeud] (4) at (2,0) {};
		\node[noeud] (6) at (3,1) {};
		\node[noeud] (5) at (3,-1) {};
		\node[noeud] (5b) at (4,-1) {};
		\node[noeud] (7) at (4,1) {};
		\node[noeud] (8) at (3,0) {};
		\draw (3) -- (4);
		\draw (8) -- (4);
		\draw (4) to (6);
		\draw (4) to (5);
		\draw (6) -- (7);
		\draw (5) -- (5b);
	\end{tikzpicture}
	};
	\draw (14.5,-5.5) node {$\outcomeP$ by};
	\draw (14.5,-6) node {\textbf{Case 3}};
	\draw (14.5,-6.5) node {since};
	\draw (14.5,-7) node {\sstar{1,1}$\sim_1$\sstar{2}};
	
	\draw (orig) -- (2.75,0) [->] (2.75,0) -- (removeTwoFrom21);
	\draw[->] (removeTwoFrom21) -- (removeTwoFrom21B);
	\draw (orig) -- (2.75,0) [->] (2.75,0) |- (removeOneFrom222);
	\draw (orig) -- (2.75,0) [->] (2.75,0) |- (removeTwoFrom222);
	\draw[->] (removeOneFrom222) -- (removeFrom222);
	\draw[->] (removeTwoFrom222) -- (removeFrom222);
	\draw (orig) -- (2.75,0) [->] (2.75,0) |- (playTo11);
	\draw[->] (playTo11) -- (playTo11B);
	\draw (orig) -- (2.75,0) [->] (2.75,0) |- (playTo2);
	\draw[->] (playTo2) -- (playTo2B);
\end{tikzpicture}
\caption{The inductive part of the proof that \sstar{2,1} $\sim_1$ \sstar{2,2,2}.}
\label{fig:pod21EQUIV1pod222}
\end{figure}
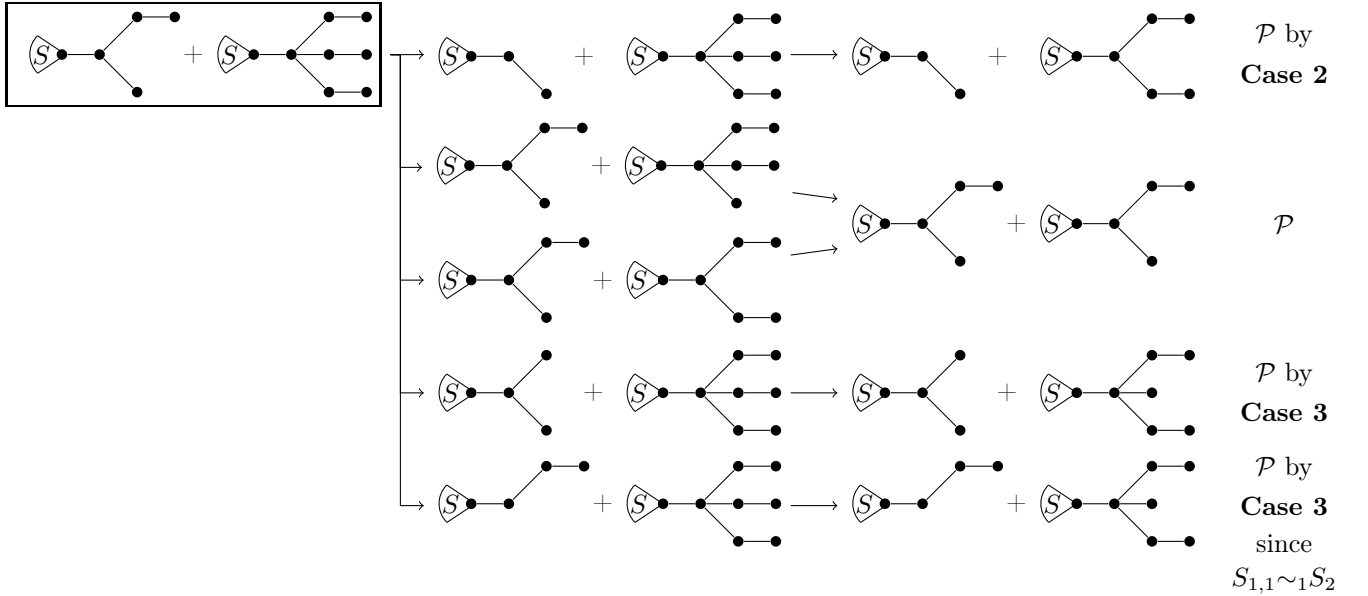
\end{proof}

\ \\

\begin{lemma}\label{lem:onestarone}
Let $S$ be a subdivided star not belonging to $\conestar\cup\ctwostar$. Then \sstar{1,1,1}$\bipa S \equiv P_1 \bipa S$.
\end{lemma}

\begin{proof}
We use induction on $|S|$. The base cases are the subdivided stars having an option in $\conestar\cup\ctwostar$:
\begin{enumerate}
\item $S=\emptyset$. In this case, $\grundy($\sstar{1,1,1}$\bipa S) = \grundy($\sstar{1,1,1}$) = 1 = \grundy(P_1) = \grundy(P_1 \bipa S)$.
\item $S=$\sstar{1,1}. In this case, $\grundy($\sstar{1,1,1}$\bipa S) = \grundy($\sstar{1,1,1}$)$ (by Lemma~\ref{lem:equivstar}) $= \grundy(P_1 \bipa S)$.
\item $S=$\sstar{2,1,1}. In this case, $\grundy($\sstar{1,1,1}$\bipa S) = 3 = \grundy($\sstar{2,1,1,1}$) = \grundy(P_1 \bipa S)$.
\item $S=$\sstar{2,2,1}. In this case, $\grundy($\sstar{1,1,1}$\bipa S) = \grundy($\sstar{1,1,1}$\bipa$\sstar{1,1}$)$ (by Lemma~\ref{lem:equivstar}) $= 1 = \grundy($\sstar{2,2,1,1}$) = \grundy(P_1 \bipa S)$.
\item $S=$\sstar{2,2,2,1}. In this case, $\grundy($\sstar{1,1,1}$\bipa S) = 3 = \grundy($\sstar{2,2,2,1,1}$) = \grundy(P_1 \bipa S)$.
\item $S=$\sstar{2,2,2,2}. In this case, $\grundy($\sstar{1,1,1}$\bipa S) = 1 = \grundy($\sstar{2,2,2,2,1}$) = \grundy(P_1 \bipa S)$.
\end{enumerate}
Although tedious, all these values can be computed by considering the Grundy values of the sets $\opt(S_{1,1,1}\bipa S)$ and $\opt(P_1\bipa S)$.

We now prove that if $S$ is a subdivided star not belonging to $\conestar\cup\ctwostar$ and not having an option in $\conestar\cup\ctwostar$, then \sstar{1,1,1}$\bipa S \equiv P_1 \bipa S$.
We note that the first player can neither empty $S$ nor take its central vertex. We show that for every first player's move on \sstar{1,1,1}$\bipa S + P_1 \bipa S$, the second player can always move to a $\outcomeP$-position. If the first player plays from $S$ to $S'$, then $S' \not\in \conestar\cup\ctwostar$, thus if the second player replicates the move, we can invoke the induction hypothesis. Figure~\ref{fig:onestarone} shows the case where the first player does not play on $S$, completing the proof.

\begin{figure}[H]
\centering
\begin{tikzpicture}
	\node (orig) at (0,0) {\fbox{
	\begin{tikzpicture}[scale=0.5]
		\node[noeud] (1) at (-1.1,0) {};
		\node[noeud] (2) at (-2.1,0) {};
		\draw (-2,0) -- (-2.75,0.5);
		\draw (-2,0) -- (-2.75,-0.5);
		\draw (-2.75,0.5) arc (135:225:0.7);
		\node (2b) at (-2.65,0) {$S$};
		\draw (1) -- (2);
		
		\draw (-0.4,0) node {+};
		
		\draw (1,0) -- (0.25,0.5);
		\draw (1,0) -- (0.25,-0.5);
		\draw (0.25,0.5) arc (135:225:0.7);
		\node (3b) at (0.35,0) {$S$};
		\node[noeud] (3) at (1,0) {};
		\node[noeud] (4) at (2,0) {};
		\node[noeud] (6) at (3,1) {};
		\node[noeud] (5) at (3,-1) {};
		\node[noeud] (7) at (3,0) {};
		\draw (3) to (4);
		\draw (4) to (5);
		\draw (4) to (6);
		\draw (4) to (7);
	\end{tikzpicture}}
	};
	
	\node (playOnPod0) at (4.5,0) {
	\begin{tikzpicture}[scale=0.5]
		\node[noeud] (2) at (-2.1,0) {};
		\draw (-2,0) -- (-2.75,0.5);
		\draw (-2,0) -- (-2.75,-0.5);
		\draw (-2.75,0.5) arc (135:225:0.7);
		\node (2b) at (-2.65,0) {$S$};
		
		\draw (-0.4,0) node {+};
		
		\draw (1,0) -- (0.25,0.5);
		\draw (1,0) -- (0.25,-0.5);
		\draw (0.25,0.5) arc (135:225:0.7);
		\node (3b) at (0.35,0) {$S$};
		\node[noeud] (3) at (1,0) {};
		\node[noeud] (4) at (2,0) {};
		\node[noeud] (6) at (3,1) {};
		\node[noeud] (5) at (3,-1) {};
		\node[noeud] (7) at (3,0) {};
		\draw (3) to (4);
		\draw (4) to (5);
		\draw (4) to (6);
		\draw (4) to (7);
	\end{tikzpicture}
	};
	\node (playOnPod111) at (4.5,-1.5) {
	\begin{tikzpicture}[scale=0.5]
		\node[noeud] (1) at (-1.1,0) {};
		\node[noeud] (2) at (-2.1,0) {};
		\draw (-2,0) -- (-2.75,0.5);
		\draw (-2,0) -- (-2.75,-0.5);
		\draw (-2.75,0.5) arc (135:225:0.7);
		\node (2b) at (-2.65,0) {$S$};
		\draw (1) -- (2);
		
		\draw (-0.4,0) node {+};
		
		\draw (1,0) -- (0.25,0.5);
		\draw (1,0) -- (0.25,-0.5);
		\draw (0.25,0.5) arc (135:225:0.7);
		\node (3b) at (0.35,0) {$S$};
		\node[noeud] (3) at (1,0) {};
		\node[noeud] (4) at (2,0) {};
		\node[noeud] (6) at (3,1) {};
		\node[noeud] (5) at (3,-1) {};
		\draw (3) to (4);
		\draw (4) to (5);
		\draw (4) to (6);
	\end{tikzpicture}
	};
	\node (playOnPods) at (9,-0.75) {
	\begin{tikzpicture}[scale=0.5]
		\node[noeud] (2) at (-2.1,0) {};
		\draw (-2,0) -- (-2.75,0.5);
		\draw (-2,0) -- (-2.75,-0.5);
		\draw (-2.75,0.5) arc (135:225:0.7);
		\node (2b) at (-2.65,0) {$S$};
		
		\draw (-0.4,0) node {+};
		
		\draw (1,0) -- (0.25,0.5);
		\draw (1,0) -- (0.25,-0.5);
		\draw (0.25,0.5) arc (135:225:0.7);
		\node (3b) at (0.35,0) {$S$};
		\node[noeud] (3) at (1,0) {};
		\node[noeud] (4) at (2,0) {};
		\node[noeud] (6) at (3,1) {};
		\node[noeud] (5) at (3,-1) {};
		\draw (3) to (4);
		\draw (4) to (5);
		\draw (4) to (6);
	\end{tikzpicture}
	};
	\draw (13,-0.5) node {$\outcomeP$ by};
	\draw (13,-1) node {$\emptyset \sim_1$\sstar{1,1}};
	
	\draw (orig) -- (2.25,0) [->] (2.25,0) |- (playOnPod0);
	\draw (orig) -- (2.25,0) [->] (2.25,0) |- (playOnPod111);
	\draw[->] (playOnPod0) -- (playOnPods);
	\draw[->] (playOnPod111) -- (playOnPods);
\end{tikzpicture}
\caption{The inductive part of the proof for Lemma~\ref{lem:onestarone}.}
\label{fig:onestarone}
\end{figure}
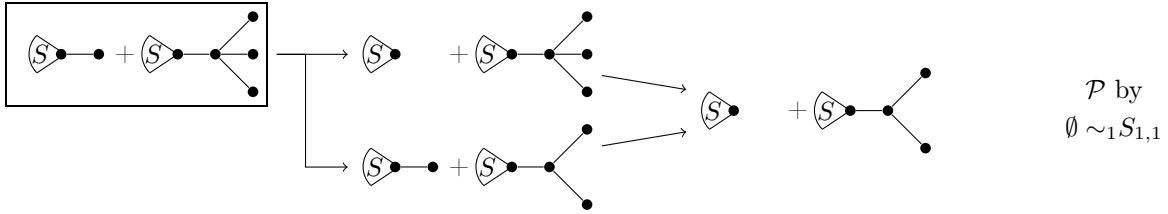
\end{proof}

\begin{proof}[Proof of Theorem~\ref{thm:equiv1}]

We prove by induction on the total number of vertices of $S$ and $S'$ that if $S$ and $S'$ are in the same set $\mathcal C_0$, $\mathcal C_1$, $\conestar$, $\mathcal C_2$, $\ctwostar$, $\ctwobox$,  $\mathcal C_3$ or $\cthreebox$, then they are $\sim_1$-equivalent.

By Lemma~\ref{lem:equivstar}, this is true if $S$ and $S'$ are in $\conestar$ or in $\ctwostar$. This is also true if $\{S,S'\}=\{\emptyset,$\sstar{1,1}$\}$ by Lemma~\ref{lem:P3empty} or if  $\{S,S'\}=\{\emptyset,P_3\}$ since it is the same as attaching a $P_3$ to the central vertex of a subdivided star.

Furthermore, one can check that the rows and columns for $\conestar$ and $\ctwostar$ in Table~\ref{tab:prod1} are correct. For that, it suffices to prove it for one representant of $\conestar$ ($P_1$) and one representant of $\ctwostar$ ($P_2$). Attaching $P_1$ to any subdivided star $\hat{S}$ results in the subdivided star listed directly under $\hat{S}$ in Figure~\ref{fig:tabpos}, while attaching $P_2$ results in the subdivided star listed diagonally to the right and below. For example, for \sstar{1,1,2,2}, attaching $P_1$ results in \sstar{1,1,1,2,2}, while attaching $P_2$ results in \sstar{1,1,2,2,2}. Comparing the Grundy values of the individual stars and the resulting bistar, one can verify that the table columns for $\conestar$ and $\ctwostar$ are correct. Identifying the elements of the various sets in Figure~\ref{fig:tabEquivSim1}, we see that below any element of $\conestar$ is a bisected star with Grundy value 2, and similarly, below any element of $\ctwostar$ is a bisected star with Grundy value 0. All other stars are either part of a 0-1 pattern (going down the columns), or part of a 2-3 pattern, which fits the computation of the Grundy values via the nim-sum, since $3 \oplus 1 = 2$. This verifies the result for the $\conestar$ column. Likewise, one can verify the column $\ctwostar$.

Suppose now that $S$ and $S'$ belong to the same set $C$, with $C\neq \conestar$ and $C\neq \ctwostar$. Thus both $S$ and $S'$ are either empty or not a path. We prove by induction on the size of $\hat{S}$ that $S\bipa \hat{S} \equiv S'\bipa \hat{S}$ for any subdivided star $\hat{S}$. This is true if $\hat{S}=\emptyset$ (since $\grundy(S)=\grundy(S')$) or if $\hat{S}\in \conestar\cup \ctwostar$  (as discussed before).
Hence we can assume that $\hat{S}\notin \conestar\cup\ctwostar$ and $\hat{S}$ is not a path. We will prove that $S\bipa \hat{S} + S'\bipa \hat{S}$ is a $\outcomeP$-position. The first player cannot play both in $S$ and $\hat{S}$ nor both in $S'$ and $\hat{S}$ since $\hat{S}$ is not a path. If the first player plays in $\hat{S}$, leading to $\hat{S}'$ in one of the two games, the first player cannot take the central vertex (since $\hat{S}$ is not a path). Hence the second player can reply to $S\bipa \hat{S}' + S'\bipa \hat{S}'$ which is a $\outcomeP$-position by induction hypothesis.
Otherwise, the first player plays in $S$ or in $S'$. By symmetry, we can assume that the first player plays in $S$, leading to a game $T\bipa \hat{S} + S'\bipa \hat{S}$. We have to find an answer from that game to a $\outcomeP$-position.

\begin{enumerate}[{\it (i)}]
	\item If there is a move from $T$ to $T'$ with $T'$ in the same set as $S$, then the second player plays to $T'\bipa \hat{S} + S'\bipa \hat{S}$ (this is always possible since if the move from $T$ to $T'$ is taking the central vertex and $T'$ is not empty, it means that $T$ is a path which is neither $P_3$ nor $P_4$, a contradiction). By induction, $T\bipa \hat{S} + S'\bipa \hat{S}$ is a $\outcomeP$-position.
	
	\item If there is a move from $S'$ to $T'$ with $T$ and $T'$ in the same set, then the second player plays to $T\bipa \hat{S}+ T'\bipa \hat{S}$ (again, this is always possible since $S'$ is not a path), which is a $\outcomeP$-position by induction hypothesis.
\end{enumerate}

Assume that none of these two cases occurs. If $\grundy(S)=3$ then we are in case {\em (ii)}. If $\grundy(S)=0$ then we are in case {\em (i)}. Hence we have $\grundy(S)\in \{1,2\}$.
If $\grundy(S)=1$, then $S,S'\in \mathcal C_1$. If $\grundy(T)=0$ then we are in case {\em (ii)}. Otherwise, $\grundy(T)>1$, and there is always a move from $T$ to $T'\in \mathcal C_1$ and we are in case {\em (ii)}.
Hence $\grundy(S)=2$. If $\grundy(T)=0$ or if $T\in \mathcal C_1$, then we are in case {\em (ii)}. If $\grundy(T)=3$, we are in case {\em (i)}. Hence the only remaining case is $T\in \conestar$. Then there is a move from $S'$ to $T'$ with $T'\in \mathcal C_1$. By induction, $T'\bipa \hat{S} \equiv $\sstar{1,1,1}$\bipa \hat{S}$ (indeed, the number of vertices in $S$ and $S'$ is strictly greater than the number of vertices in $T'$ and \sstar{1,1,1} since $S$ has at least five vertices). By Lemma~\ref{lem:onestarone}, \sstar{1,1,1}$\bipa \hat{S} \equiv P_1\bipa \hat{S} \equiv T\bipa \hat{S}$ (since $\hat{S}\notin \conestar\cup \ctwostar$). Thus $T\bipa \hat{S} \equiv T'\bipa \hat{S}$ and $T\bipa \hat{S}+ T' \bipa \hat{S}$ is a $\outcomeP$-position.

To compute Table~\ref{tab:prod1}, it is enough to consider one representant of each class, for instance $\emptyset$, $P_1$, $P_2$, \sstar{1,1,1}, \sstar{2,1,1}, \sstar{2,1,1,1}, \sstar{2,2,2,2,2}, \sstar{2,2,2,2,2,1}, respectively, and compute their Grundy value.
\end{proof}

\subsubsection{When the middle path is of length 2}

The situation in that case will be more complicated than in the previous case.
We similarly define an equivalence relation $\sim_2$. Let $S$ and $S'$ be two subdivided stars. We say that $S$ and $S'$ are $\sim_2$-equivalent, denoted $S \sim_2 S'$, if and only if for any subdivided star $\hat{S}$, $S\tripa \hat{S} \equiv S' \tripa \hat{S}$.

By Lemma~\ref{lem:P3onevertex}, we already know that $P_3\sim_2\emptyset$, and thus $S_2 \sim_2 \emptyset$ and \sstar{1,1}$\sim_2\emptyset$.

We will prove that there are exactly ten equivalence classes for $\sim_2$:
\begin{itemize}
\item $\dzerostar$: subdivided stars $S$ such that $\grundy(S)=0$ and $S$ contains zero or two paths of length $2$, plus \sstar{2};
\item $\donestar=\{P_1,$\sstar{2,1},\sstar{2,2,2}$\}$ (these stars have Grundy value 1);
\item $\donebox$: subdivided stars $S$ such that $\grundy(S)=1$, $S$ contains zero or two paths of length $2$ and $S \neq P_1$;
\item $\dtwostar=\{P_2$,\sstar{2,2}$\}$ (these stars have Grundy value 2);
\item $\dtwobox$: subdivided stars $S$ such that $\grundy(S)=2$ and $S$ contains one or three paths of length $2$;
\item $\dthreebox$: subdivided stars $S$ such that $\grundy(S)=3$ and $S$ contains one or three paths of length $2$;
\item For $i\in \{0,1,2,3\}$, $\mathcal D_i$: subdivided stars $S$ with $\grundy(S)=i$ and $S$ is not in a previous class.
\end{itemize}
Figure~\ref{fig:tabEquivSim2} shows the equivalence classes of the subdivided stars.

\begin{figure}[!h]
	\centering
	\begin{tikzpicture}
	\draw [->,thick] (-1,1.5) -- (-1,-8.5);
	\draw [->,thick] (-1,1.5) -- (8.5,1.5);
	
	\draw (3.75,2.4) node {Number of paths of length 2 in the subdivided star};
	\draw (-2.3,-3.5) node[rotate=90]  {Number of paths in the subdivided star};
	
	\draw (-1.2,0) node {0};
	\draw (-1.2,-1) node {1};
	\draw (-1.2,-2) node {2};
	\draw (-1.2,-3) node {3};
	\draw (-1.2,-4) node {4};
	\draw (-1.2,-5) node {5};
	\draw (-1.5,-6) node {\ldots};
	\draw (-1.5,-7) node {$2p$};
	\draw (-1.5,-8) node {$2p+1$};
	
	\draw (0,1.8) node {0};
	\draw (1,1.8) node {1};
	\draw (2,1.8) node {2};
	\draw (3,1.8) node {3};
	\draw (4,1.8) node {4};
	\draw (5,1.8) node {5};
	\draw (6,1.8) node {\ldots};
	\draw (7,1.8) node {$2p$};
	\draw (8,1.8) node {$2p+1$};
	
	\node (empty) at (0,1) {$0^*$};
	\node (p1) at (0,0) {$1^*$};
	
	\node (p2) at (0,-1) {$2^*$};
	\node (p3) at (1,-1) {$0^*$};
	
	\node (p3b) at (0,-2) {$0^*$};
	\node (p4) at (1,-2) {$1^*$};
	\node (p5) at (2,-2) {$2^*$};
	
	\node (s111) at (0,-3) {$1^\Box$};
	\node (s112) at (1,-3) {$2^\Box$};
	\node (s122) at (2,-3) {$0^*$};
	\node (s222) at (3,-3) {$1^*$};
	
	\node (s1111) at (0,-4) {$0^*$};
	\node (s1112) at (1,-4) {$3^\Box$};
	\node (s1122) at (2,-4) {$1^\Box$};
	\node (s1222) at (3,-4) {$2^\Box$};
	\node (s2222) at (4,-4) {0};
	
	\node (s11111) at (0,-5) {$1^\Box$};
	\node (s11112) at (1,-5) {$2^\Box$};
	\node (s11122) at (2,-5) {$0^*$};
	\node (s11222) at (3,-5) {$3^\Box$};
	\node (s12222) at (4,-5) {1};
	\node (s22222) at (5,-5) {2};
	
	\draw [->] (p1) to (empty);
	\draw [->] (p2) to[out=120, in=-120] (empty);
	\draw [->] (p3b) to[out=120, in=-120] (p1);
	\draw [->] (p2) to (p1);
	\draw [->] (p3) to (p1);
	\draw [->] (p3) to (p2);
	\draw [->] (p3b) to (p2);
	\draw [->] (p4) to (p3b);
	\draw [->] (p4) to (p2);
	\draw [->] (p4) to (p3);
	\draw [->] (p5) to (p3);
	\draw [->] (p5) to (p4);
	\draw [->] (s111) to (p3b);
	\draw [->] (s112) to (p3b);
	\draw [->] (s112) to (p4);
	\draw [->] (s112) to (s111);
	\draw [->] (s122) to (p5);
	\draw [->] (s122) to (p4);
	\draw [->] (s122) to (s112);
	\draw [->] (s222) to (p5);
	\draw [->] (s222) to (s122);
	\draw [->] (s1111) to (s111);
	\draw [->] (s1112) to (s111);
	\draw [->] (s1112) to (s112);
	\draw [->] (s1112) to (s1111);
	\draw [->] (s1122) to (s112);
	\draw [->] (s1122) to (s122);
	\draw [->] (s1122) to (s1112);
	\draw [->] (s1222) to (s122);
	\draw [->] (s1222) to (s222);
	\draw [->] (s1222) to (s1122);
	\draw [->] (s2222) to (s222);
	\draw [->] (s2222) to (s1222);
	\draw [->] (s11111) to (s1111);
	\draw [->] (s11112) to (s1111);
	\draw [->] (s11112) to (s1112);
	\draw [->] (s11112) to (s11111);
	\draw [->] (s11122) to (s1112);
	\draw [->] (s11122) to (s1122);
	\draw [->] (s11122) to (s11112);
	\draw [->] (s11222) to (s1122);
	\draw [->] (s11222) to (s1222);
	\draw [->] (s11222) to (s11122);
	\draw [->] (s12222) to (s1222);
	\draw [->] (s12222) to (s2222);
	\draw [->] (s12222) to (s11222);
	\draw [->] (s22222) to (s2222);
	\draw [->] (s22222) to (s12222);
	
	\draw (0,-7) node {$0^*$};
	\draw (1,-7) node {$3^\Box$};
	\draw (2,-7) node {$1^\Box$};
	\draw (3,-7) node {$2^\Box$};
	\draw (4,-7) node {$0$};
	\draw (5,-7) node {$3$};
	\draw (6,-7) node {\ldots};
	\draw (7,-7) node {$0$};
	
	\draw (0,-8) node {$1^\Box$};
	\draw (1,-8) node {$2^\Box$};
	\draw (2,-8) node {$0^*$};
	\draw (3,-8) node {$3^\Box$};
	\draw (4,-8) node {$1$};
	\draw (5,-8) node {$2$};
	\draw (6,-8) node {\ldots};
	\draw (7,-8) node {$1$};
	\draw (8,-8) node {$2$};
	
	\draw [<-] (0.2,-7) -- (0.8,-7);
	\draw [<-] (1.2,-7) -- (1.8,-7);
	\draw [<-] (2.2,-7) -- (2.8,-7);
	\draw [<-] (3.2,-7) -- (3.8,-7);
	\draw [<-] (4.2,-7) -- (4.8,-7);
	
	\draw [<-] (0.2,-8) -- (0.8,-8);
	\draw [<-] (1.2,-8) -- (1.8,-8);
	\draw [<-] (2.2,-8) -- (2.8,-8);
	\draw [<-] (3.2,-8) -- (3.8,-8);
	\draw [<-] (4.2,-8) -- (4.8,-8);
	\draw [<-] (7.2,-8) -- (7.8,-8);
	
	\draw [<-] (0,-7.2) -- (0,-7.8);
	\draw [<-] (1,-7.2) -- (1,-7.8);
	\draw [<-] (2,-7.2) -- (2,-7.8);
	\draw [<-] (3,-7.2) -- (3,-7.8);
	\draw [<-] (4,-7.2) -- (4,-7.8);
	\draw [<-] (5,-7.2) -- (5,-7.8);
	\draw [<-] (7,-7.2) -- (7,-7.8);
	
	\draw [<-] (0.2,-7.2) -- (0.8,-7.8);
	\draw [<-] (1.2,-7.2) -- (1.8,-7.8);
	\draw [<-] (2.2,-7.2) -- (2.8,-7.8);
	\draw [<-] (3.2,-7.2) -- (3.8,-7.8);
	\draw [<-] (4.2,-7.2) -- (4.8,-7.8);
	\draw [<-] (7.2,-7.2) -- (7.8,-7.8);
	\end{tikzpicture}
	\caption{First six rows, and rows $2p$ and $2p+1$, of the table of equivalence classes for $\sim_2$ of the subdivided stars. Stars belonging to resp. $\dzerostar$, $\donestar$, $\donebox$, $\dtwostar$, $\dtwobox$, $\dthreebox$ are depicted by resp. $0^*$, $1^*$, $1^\Box$, $2^*$, $2^\Box$, $3^\Box$, while the $\mathcal D_i$'s are depicted by $i$.}
	\label{fig:tabEquivSim2}
\end{figure}

\begin{theorem}\label{thm:equiv2}
The equivalence classes for $\sim_2$ are exactly the sets $\mathcal D_0$, $\dzerostar$, $\mathcal D_1$, $\donestar$, $\donebox$, $\mathcal D_2$, $\dtwostar$, $\dtwobox$,  $\mathcal D_3$ and $\dthreebox$. Moreover, Table~\ref{tab:prod2} describes how the Grundy value of $S\tripa S'$ can be computed depending on the equivalence class of $S$ and $S'$.
\end{theorem}

\begin{table}
\begin{center}
\begin{math}
\begin{array}{c|c|c|c|c|c|c|c|c|c|c}
&\mathcal D_0 & \dzerostar & \mathcal D_1 & \donestar & \donebox & \mathcal D_2 & \dtwostar & \dtwobox & \mathcal D_3 & \dthreebox \\ \hline
\mathcal D_0 & \nimsum & \nimsum_1 & \nimsum & 2 & \nimsum_1 & \nimsum & 0 & \nimsum_1 & \nimsum & \nimsum_1 \\ \hline
\dzerostar & \nimsum_1 & \nimsum_1 & \nimsum_1 & 2 & \nimsum_1 & \nimsum_1 & 0 & \nimsum_1 & \nimsum_1 & \nimsum_1 \\ \hline
\mathcal D_1 & \nimsum & \nimsum_1 & \nimsum & 3 & \nimsum_1 & \nimsum & 1 & \nimsum_1 & \nimsum & \nimsum_1 \\ \hline
\donestar & 2 & 2 & 3 & 0 & 3 & 0 & 1 & 1 & 1 & 0  \\ \hline
\donebox & \nimsum_1 & \nimsum_1 & \nimsum_1 & 3 & \nimsum_1 & \nimsum_1 & 1 & \nimsum_1 & \nimsum_1 & \nimsum_1 \\ \hline
\mathcal D_2 & \nimsum & \nimsum_1 & \nimsum & 0 & \nimsum_1 & \nimsum & 2 & \nimsum_1 & \nimsum &  \nimsum_1 \\ \hline
\dtwostar & 0 & 0 & 1 & 1 & 1 & 2 & 2 & 2 & 3 & 3 \\ \hline
\dtwobox & \nimsum_1 & \nimsum_1 & \nimsum_1 & 1 & \nimsum_1 & \nimsum_1 & 2 & 0 & \nimsum_1 & 1 \\ \hline
\mathcal D_3 & \nimsum & \nimsum_1 & \nimsum & 1 & \nimsum_1 & \nimsum & 3 & \nimsum_1 & \nimsum & \nimsum_1 \\ \hline
\dthreebox & \nimsum_1 & \nimsum_1 & \nimsum_1 & 0 & \nimsum_1 & \nimsum_1 & 3 & 1 & \nimsum_1 & 0 \\ \hline
\end{array}
\end{math}
\end{center}
\caption[prod2]{Computing the Grundy value of $S \tripa S'$ depending on the equivalence class of $S$ and $S'$. Recall that $\nimsum$ denotes the nim-sum. Moreover, $x \nimsum_1 y$ stands for $x \nimsum y \nimsum 1$.}
\label{tab:prod2}
\end{table}

Notice that when the two subdivided stars are of sufficiently large order, they are in the classes $\mathcal D_0,\mathcal D_1,\mathcal D_2,\mathcal D_3$, and the Grundy value of the bistar is given by the nim-sum of the Grundy values of the two stars. For most of the smallest subdivided stars, $\grundy(S\tripa S') = \grundy(S) \oplus \grundy(S') \oplus 1$.


The following lemma proves the equivalence for $\donestar$ and $\dtwostar$. Its proof is not included, since it is similar to the proof of Lemma~\ref{lem:equivstar}.

\begin{lemma}
\label{lem:equivstar2}
We have:
\begin{enumerate}
\item $P_1 \sim_2$ \sstar{2,1}
\item $P_2 \sim_2$ \sstar{2,2}
\item \sstar{1,1} $\sim_2$ \sstar{2,2,1}
\item \sstar{2,1} $\sim_2$ \sstar{2,2,2}.
\end{enumerate} 
Therefore, any two elements in $\donestar$ (resp. $\dtwostar$) are $\sim_2$-equivalent.
\end{lemma}

We can now prove Theorem~\ref{thm:equiv2}:

\begin{proof}[Proof of Theorem~\ref{thm:equiv2}]
Rather than proving the validity of equivalence classes and then deducing the table, we prove by induction on the total number of vertices in $S$ and $S'$ that the Grundy value of $S \tripa S'$ is given by Table~\ref{tab:prod2}.

One can check that the rows and columns for $\donestar$ and $\dtwostar$ in Table~\ref{tab:prod2} are correct: it suffices to prove it for one representant for $\donestar$ (say $P_1$) and for $\dtwostar$ (say $P_2$). This is possible since if $S,S' \in \donestar,\dtwostar$, then they are $\sim_2$-equivalent by Lemma~\ref{lem:equivstar2}.
For any subdivided star $\hat{S}$, $\hat{S} \tripa P_1$ is $\hat{S}$ with a path of length 2 attached to its central vertex. Thus, for every class, we only need to look at the value diagonally to the right and below in Figure~\ref{fig:tabgrun}. One can check that if $\grundy(\hat{S})=0$, then $\grundy(\hat{S} \tripa P_1)=2$, if $\hat{S} \in \donestar,\mathcal D_2, \dthreebox$, then $\grundy(\hat{S} \tripa P_1)=0$, if $\hat{S} \in \mathcal D_1,\donebox$, then $\grundy(\hat{S} \tripa P_1)=3$, if $\hat{S} \in \dtwostar,\dtwobox,\mathcal D_3$, then $\grundy(\hat{S} \tripa P_1)=1$.
For any subdivided star $\hat{S}$, $\hat{S} \tripa P_2$ is $\hat{S}$ with a path of length 3 attached to its central vertex. Thus, $\grundy(\hat{S} \tripa P_2) = \grundy(\hat{S})$.

Now we study the Grundy value of $S \tripa S'$ depending on the class of $S$ and $S'$. We can suppose that $S,S' \not\in \donestar,\dtwostar$, and that neither $S$ nor $S'$ are \sstar{1,1} or $P_3$ (since, by Lemma~\ref{lem:P3onevertex}, \sstar{1,2}$\sim_2 \emptyset$; and $P_3 \sim_2 \emptyset$ by Lemma~\ref{lem:modkpodes}).
We can find the Grundy values of the options of $S$ and $S'$ thanks to Figure~\ref{fig:tabEquivSim2}. None of the options of $S$ and $S'$ involves taking their central vertex. We can verify Table~\ref{tab:prod2} by computing the Grundy value of $S \tripa S'$ depending on the Grundy values of their options, by using the induction hypothesis:
\begin{center}
$\grundy(S \tripa S') = \mex( \grundy( T \tripa S' ) , \grundy( S \tripa T' ) | T$ option of $S$, $T'$ option of $S' )$
\end{center}
In order to prove that the equivalence classes are correct, we need to check that the Grundy value of $S \tripa S'$ does not change with the classes of the options of $S$ and $S'$. Indeed, two subdivided stars belonging to the same class can have different options.

We will prove two cases, the other ones being similar.

\noindent\textbf{Case 1:} $S \in \mathcal D_1$ and $S' \in \mathcal D_3$

In this case, $S$ always has three different options, but these options are not the same depending on $S$. $S$ always has an option in $\mathcal D_0$, and it can have two options either in $\dtwobox$ and $\dthreebox$ or in $\mathcal D_2$ and $\mathcal D_3$.
$S'$ has three options, which are in $\mathcal D_1,\mathcal D_2$ and $\mathcal D_3$.

These possible options of $S$ and $S'$ are shown in Figure~\ref{fig:t1tripat2-c1c3}. On the left are the possible options of $S$, and on the right are the possible options of $S'$. The notation $\mathcal D_i \tripa \mathcal D_j$ expresses the fact that the two subdivided stars $T$ and $S'$ (resp. $S$ and $T'$) are in the classes $\mathcal D_i$ and $\mathcal D_j$, and that the subdivided bistar is smaller than $S \tripa S'$, allowing us to invoke the induction hypothesis.

\begin{figure}[H]
\centering
\begin{tikzpicture}
	\node (orig) at (0,0) {\fbox{
	\begin{tikzpicture}
		\draw (0,0) node {$\mathcal D_1 \tripa \mathcal D_3$};
	\end{tikzpicture}}
	};
	
	\node (right1) at (3,-1.5) {
	\begin{tikzpicture}
		\draw (0,0) node {$\mathcal D_1 \tripa \mathcal D_0$};
		\draw (0.75,0) node {;};
		\draw (1.5,0) node {$\mathcal D_1 \tripa \mathcal D_1$};
		\draw (2.25,0) node {;};
		\draw (3,0) node {$\mathcal D_1 \tripa \mathcal D_2$};
	\end{tikzpicture}
	};
	
	\node (left1) at (-3,-1) {
	\begin{tikzpicture}
		\draw (0,0) node {$\mathcal D_0 \tripa \mathcal D_3$};
		\draw (0.75,0) node {;};
		\draw (1.5,0) node {$\dtwobox \tripa \mathcal D_3$};
		\draw (2.25,0) node {;};
		\draw (3,0) node {$\dthreebox \tripa \mathcal D_3$};
	\end{tikzpicture}
	};
	\node (left2) at (-3,-2) {
	\begin{tikzpicture}
		\draw (0,0) node {$\mathcal D_0 \tripa \mathcal D_3$};
		\draw (0.75,0) node {;};
		\draw (1.5,0) node {$\mathcal D_2 \tripa \mathcal D_3$};
		\draw (2.25,0) node {;};
		\draw (3,0) node {$\mathcal D_3 \tripa \mathcal D_3$};
	\end{tikzpicture}
	};
	
	\draw[->] (orig) |- (right1);
	\draw[->] (orig) |- (left1);
	\draw[->] (orig) |- (left2);
\end{tikzpicture}
\caption[Possible options]{The possible options of $S \tripa S'$ when $S_1 \in \mathcal C_1$ and $S_2 \in \mathcal C_3$.}
\label{fig:t1tripat2-c1c3}
\end{figure}
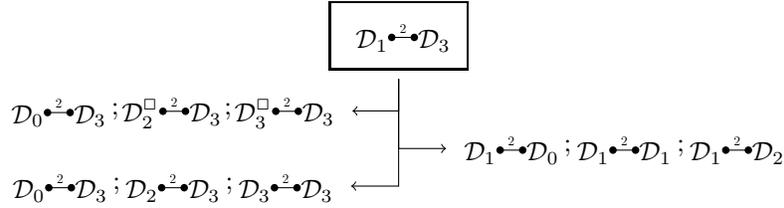

Now, we can compute the Grundy value of $S \tripa S'$. First, we compute this value in the case where the options of $S$ are in $\dtwobox$ and $\dthreebox$:

\begin{center}
	\begin{tabular}{c c l}
	$\grundy(S \tripa S')$ & $=$ & 
	$\mex( \grundy( \mathcal D_0 \tripa \mathcal D_3 ) , \grundy( \dtwobox \tripa \mathcal D_3 ) , \grundy( \dthreebox \tripa \mathcal D_3 ) , \grundy ( \mathcal D_1 \tripa \mathcal D_0 ) , \grundy ( \mathcal D_1 \tripa \mathcal D_1 ) , \grundy ( \mathcal D_1 \tripa \mathcal D_2 ) )$ \\
	& $=$ & $\mex( 3,0,1,1,0,3 )$ (by induction hypothesis) \\
	& $=$ & 2
	\end{tabular}
\end{center}

Now, we compute this value in the case where the options of $S$ are in $\mathcal D_2$ and $\mathcal D_3$:

\begin{center}
	\begin{tabular}{c c l}
		$\grundy(S \tripa S')$ & $=$ & 
		$\mex( \grundy( \mathcal D_0 \tripa \mathcal D_3 ) , \grundy( \mathcal D_2 \tripa \mathcal D_3 ) , \grundy( \mathcal D_3 \tripa \mathcal D_3 ) , \grundy ( \mathcal D_1 \tripa \mathcal D_0 ) , \grundy ( \mathcal D_1 \tripa \mathcal D_1 ) , \grundy ( \mathcal D_1 \tripa \mathcal D_2 ) )$ \\
		& $=$ & $\mex( 3,1,0,1,0,3 )$ (by induction hypothesis) \\
		& $=$ & 2
	\end{tabular}
\end{center}

The Grundy value being the same in both cases, we can conclude that $\grundy(S \tripa S')=2$.

\noindent\textbf{Case 2:} $S \in \mathcal D_0$ and $S' \in \mathcal D_2$

In this case, the possible options of $S$ and $S'$ are shown in Figure~\ref{fig:t1tripat2-c0c2}. On the left are the options of $S$, and on the right are the options of $S'$. Below each possible bistar is the Grundy value of the bistar, thanks to the induction hypothesis. By computing the $\mex$ value of each of the six sets of options, we always find the value 2.
Thus, $\grundy(S \tripa S')=2$.

\begin{figure}[H]
\centering
\begin{tikzpicture}
	\node (orig) at (0,0) {\fbox{
	\begin{tikzpicture}
		\draw (0,0) node {$\mathcal D_0 \tripa \mathcal D_2$};
	\end{tikzpicture}}
	};
	
	\node (right1) at (3,-2) {
	\begin{tikzpicture}
		\draw (0,0) node {$\mathcal D_0 \tripa \mathcal D_0$};
		\draw (0.75,0) node {;};
		\draw (1.5,0) node {$\mathcal D_0 \tripa \mathcal D_1$};
		\draw (0,-0.5) node {0};
		\draw (1.5,-0.5) node {1};
	\end{tikzpicture}
	};
	\node (right2) at (3,-3) {
	\begin{tikzpicture}
		\draw (0,0) node {$\mathcal D_0 \tripa \mathcal D_0$};
		\draw (0.75,0) node {;};
		\draw (1.5,0) node {$\mathcal D_0 \tripa \mathcal D_1$};
		\draw (2.25,0) node {;};
		\draw (3,0) node {$\mathcal D_0 \tripa \mathcal D_3$};
		\draw (0,-0.5) node {0};
		\draw (1.5,-0.5) node {1};
		\draw (3,-0.5) node {3};
	\end{tikzpicture}
	};
	
	\node (left1) at (-3,-1) {
	\begin{tikzpicture}
		\draw (0,0) node {$\donestar \tripa \mathcal D_2$};
		\draw (0.75,0) node {;};
		\draw (1.5,0) node {$\dtwobox \tripa \mathcal D_2$};
		\draw (0,-0.5) node {0};
		\draw (1.5,-0.5) node {1};
	\end{tikzpicture}
	};
	\node (left2) at (-3,-2) {
	\begin{tikzpicture}
		\draw (0,0) node {$\mathcal D_1 \tripa \mathcal D_2$};
		\draw (0.75,0) node {;};
		\draw (1.5,0) node {$\dtwobox \tripa \mathcal D_2$};
		\draw (2.25,0) node {;};
		\draw (3,0) node {$\dthreebox \tripa \mathcal D_2$};
		\draw (0,-0.5) node {3};
		\draw (1.5,-0.5) node {1};
		\draw (3,-0.5) node {0};
	\end{tikzpicture}
	};
	\node (left3) at (-3,-3) {
	\begin{tikzpicture}
		\draw (0,0) node {$\mathcal D_1 \tripa \mathcal D_2$};
		\draw (0.75,0) node {;};
		\draw (1.5,0) node {$\mathcal D_2 \tripa \mathcal D_2$};
		\draw (2.25,0) node {;};
		\draw (3,0) node {$\mathcal D_3 \tripa \mathcal D_2$};
		\draw (0,-0.5) node {3};
		\draw (1.5,-0.5) node {0};
		\draw (3,-0.5) node {1};
	\end{tikzpicture}
	};
	\node (left4) at (-3,-4) {
	\begin{tikzpicture}
		\draw (0,0) node {$\mathcal D_2 \tripa \mathcal D_2$};
		\draw (0.75,0) node {;};
		\draw (1.5,0) node {$\mathcal D_3 \tripa \mathcal D_2$};
		\draw (0,-0.5) node {0};
		\draw (1.5,-0.5) node {1};
	\end{tikzpicture}
	};
	
	\draw[->] (orig) |- (right1);
	\draw[->] (orig) |- (right2);
	\draw[->] (orig) |- (left1);
	\draw[->] (orig) |- (left2);
	\draw[->] (orig) |- (left3);
	\draw[->] (orig) |- (left4);
\end{tikzpicture}
\caption[Possible options]{The possible options of $S \tripa S'$ when $S \in \mathcal D_0$ and $S' \in \mathcal D_2$.}
\label{fig:t1tripat2-c0c2}
\end{figure}
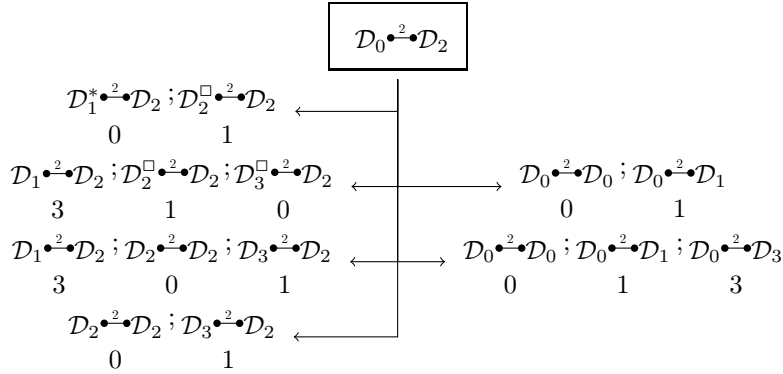

\ \newline

Overall, there are 36 cases to consider. As they are all similar to the two we already considered, we only present the possible classes of the options of $S$ in Figure~\ref{fig:optionsT}. The full proof can be found in~\cite{halFullProof}.

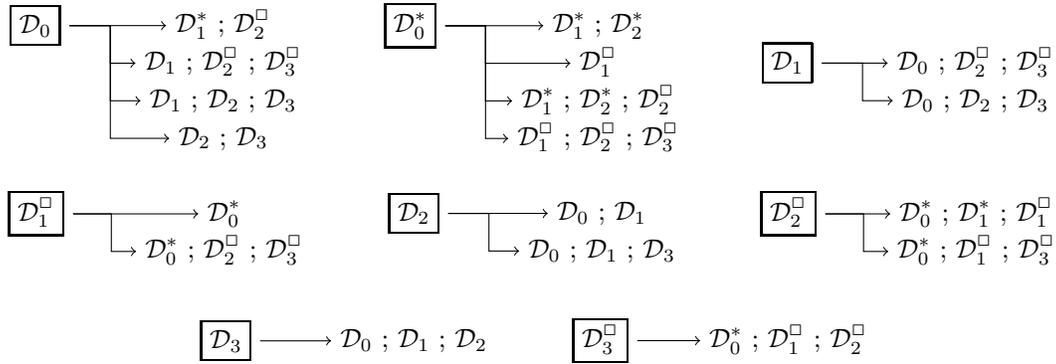
\begin{figure}[H]
\centering
\begin{tikzpicture}
	\node (c0) at (0,0) {
	\begin{tikzpicture}
		\node (0) at (0,0) {\fbox{$\mathcal D_0$}};
		\node (1) at (2.5,0) {$\donestar$ ; $\dtwobox$};
		\node (2) at (2.5,-0.5) {$\mathcal D_1$ ; $\dtwobox$ ; $\dthreebox$};
		\node (3) at (2.5,-1) {$\mathcal D_1$ ; $\mathcal D_2$ ; $\mathcal D_3$};
		\node (4) at (2.5,-1.5) {$\mathcal D_2$ ; $\mathcal D_3$};
		\draw[->] (0) -- (1,0) -- (1);
		\draw[->] (0) -- (1,0) |- (2);
		\draw[->] (0) -- (1,0) |- (3);
		\draw[->] (0) -- (1,0) |- (4);
	\end{tikzpicture}
	};
	\node (c0*) at (5,0) {
	\begin{tikzpicture}
		\node (0) at (0,0) {\fbox{$\dzerostar$}};
		\node (1) at (2.5,0) {$\donestar$ ; $\dtwostar$};
		\node (2) at (2.5,-0.5) {$\donebox$};
		\node (3) at (2.5,-1) {$\donestar$ ; $\dtwostar$ ; $\dtwobox$};
		\node (4) at (2.5,-1.5) {$\donebox$ ; $\dtwobox$ ; $\dthreebox$};
		\draw[->] (0) -- (1,0) -- (1);
		\draw[->] (0) -- (1,0) |- (2);
		\draw[->] (0) -- (1,0) |- (3);
		\draw[->] (0) -- (1,0) |- (4);
	\end{tikzpicture}
	};
	\node (c1) at (10,0) {
	\begin{tikzpicture}
		\node (0) at (0,0) {\fbox{$\mathcal D_1$}};
		\node (1) at (2.5,0) {$\mathcal D_0$ ; $\dtwobox$ ; $\dthreebox$};
		\node (2) at (2.5,-0.5) {$\mathcal D_0$ ; $\mathcal D_2$ ; $\mathcal D_3$};
		\draw[->] (0) -- (1,0) -- (1);
		\draw[->] (0) -- (1,0) |- (2);
	\end{tikzpicture}
	};
	\node (c1b) at (0,-2) {
	\begin{tikzpicture}
		\node (0) at (0,0) {\fbox{$\donebox$}};
		\node (1) at (2.5,0) {$\dzerostar$};
		\node (2) at (2.5,-0.5) {$\dzerostar$ ; $\dtwobox$ ; $\dthreebox$};
		\draw[->] (0) -- (1,0) -- (1);
		\draw[->] (0) -- (1,0) |- (2);
	\end{tikzpicture}
	};
	\node (c2) at (5,-2) {
	\begin{tikzpicture}
		\node (0) at (0,0) {\fbox{$\mathcal D_2$}};
		\node (1) at (2.5,0) {$\mathcal D_0$ ; $\mathcal D_1$};
		\node (2) at (2.5,-0.5) {$\mathcal D_0$ ; $\mathcal D_1$ ; $\mathcal D_3$};
		\draw[->] (0) -- (1,0) -- (1);
		\draw[->] (0) -- (1,0) |- (2);
	\end{tikzpicture}
	};
	\node (c2b) at (10,-2) {
	\begin{tikzpicture}
		\node (0) at (0,0) {\fbox{$\dtwobox$}};
		\node (1) at (2.5,0) {$\dzerostar$ ; $\donestar$ ; $\donebox$};
		\node (2) at (2.5,-0.5) {$\dzerostar$ ; $\donebox$ ; $\dthreebox$};
		\draw[->] (0) -- (1,0) -- (1);
		\draw[->] (0) -- (1,0) |- (2);
	\end{tikzpicture}
	};
	\node (c3) at (2.5,-3.5) {
	\begin{tikzpicture}
		\node (0) at (0,0) {\fbox{$\mathcal D_3$}};
		\node (1) at (2.5,0) {$\mathcal D_0$ ; $\mathcal D_1$ ; $\mathcal D_2$};
		\draw[->] (0) -- (1,0) -- (1);
	\end{tikzpicture}
	};
	\node (c3b) at (7.5,-3.5) {
	\begin{tikzpicture}
		\node (0) at (0,0) {\fbox{$\dthreebox$}};
		\node (1) at (2.5,0) {$\dzerostar$ ; $\donebox$ ; $\dtwobox$};
		\draw[->] (0) -- (1,0) -- (1);
	\end{tikzpicture}
	};
\end{tikzpicture}
\caption{The classes of the possible options of $S$ depending on its class.}
\label{fig:optionsT}
\end{figure}

Going through all the cases allows to prove the correctness of Table~\ref{tab:prod2}.
\end{proof}

This concludes our study of subdivided bistars.

\section{Conclusion}

In this paper, we introduced a general definition of octal games on graphs, capturing some existing take-away games on graphs. We then focused on one of the simplest octal games, {\bf 0.33}, on some subclasses of trees, namely subdivided stars and bistars.

We proved that for subdivided stars and bistars, as in paths, one can reduce the length of the paths to their length modulo 3. Thanks to this result, we have computed the exact Grundy value of any subdivided star, and exihibited a periodic behaviour. We have extended these results to bistars for which one can also reduce the lengths of any path modulo 3. Using operators and equivalence classes similar to the nim-sum and Grundy classes, we could then compute the Grundy value of a subdivided bistar using values of the two stars composing it.

However, the reduction of paths modulo 3 cannot be generalized to trees:

\begin{obs}
Attaching a $P_3$ to a vertex of a bistar which is not one of the central vertices of the stars may change the Grundy value (and even the outcome) of the game. Indeed, the bistar of Figure~\ref{fig:contreexemple} is an $\outcomeN$-position, but attaching a $P_3$ to $u$ changes it into an $\outcomeP$-position.
\end{obs}

\begin{figure}[H]
\centering
\begin{tikzpicture}[scale=0.75]
	\node[noeud] at (-1,-1) {};
	\node[noeud] at (-1,1) {};
	\node[noeud] at (0,0) {};
	\node[noeud] at (1,0) {};
	\node[noeud] at (2,0) {};
	\node[noeud] at (3,-1) {};
	\node[noeud] at (3,1) {};
	\draw (3,1.25) node {$u$};
	\node[noeud] at (4,1) {};

	\draw (0,0) to (-1,-1);
	\draw (0,0) to (-1,1);
	\draw (0,0) -- (2,0);
	\draw (2,0) to (3,-1);
	\draw (2,0) to (3,1);
	\draw (4,1) -- (3,1);
\end{tikzpicture}
\caption{Counter-example for trees: attaching a $P_3$ to $u$ changes the outcome.}
\label{fig:contreexemple}
\end{figure}
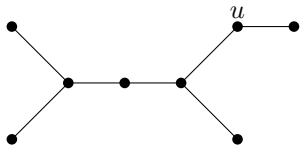

\begin{proof}
	The bistar is an $\outcomeN$-position: removing $u$ and the leaf attached to it leaves \sstar{1,1,3} which is equivalent to \sstar{1,1} by Theorem~\ref{thm:modkpodes}, which is a $\outcomeP$-position.
	
	Attaching a $P_3$ to $u$ changes the outcome: by a straightforward case analysis, one can check that every move leaves a $\outcomeN$-position.
\end{proof}

Actually, we conjecture that the Grundy value of trees for the {\bf 0.33} game is not even bounded.

\begin{conj}
	For all $n \geq 4$, there exists a tree $T$ such that $\grundy_{{\bf 0.33}}(T)=n$.
\end{conj}

This conjecture might even be true in the class of caterpillars. A
feeble argument to illustrate this intuition comes from our
computations. We may provide examples of caterpillars with Grundy
values as large as 11. Figure~\ref{fig:large_grundy} depicts a
caterpillar with a Grundy value of 10 (checked by computer).

\begin{figure}[ht]
	\begin{center}
		\begin{tikzpicture}[scale=.4]
		\node[noeud] at (0,0) {};
		\foreach \i in {1,...,36}{
			\draw (\i-1,0) -- (\i,0);
			\node[noeud] at (\i,0) {};
		}
		\foreach \i in { 2,4,6,8,10,12,14,18,20,22,24,26,28,30,34}{
			\node[noeud] at (\i,-1) {};
			\draw (\i,0) -- (\i,-1);
		}
		\end{tikzpicture}
	\end{center}
	\caption{A caterpillar with a Grundy value of 10.}
	\label{fig:large_grundy}
\end{figure}
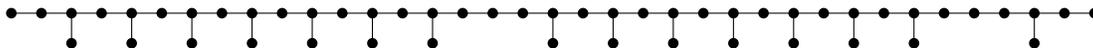

\begin{conj}
	For all $n \geq 4$, there exists a caterpillar $C$ such that $\grundy_{\textbf{0.33}}(C)=n$.
\end{conj}

However, some of our results can be generalized to other octal games on subdivided stars, see \cite{futurpapier}.

Finally, we would like to mention that it would certainly be interesting to consider the misère version of the 0.33 game on graphs.

\begin{center}
	\section*{References}
\end{center}


\begin{thebibliography}{elsarticle-harv}
	
\bibitem{winningways}
Berkelamp, E. R., Conway, J. H. and Guy, R. K. (2001-2004). {\em Winning Ways for Your Mathematical Plays} (2nd ed., 4 volumes). Wellesley, MA: A K Peters.

\bibitem{lip}
Albert, M., Nowakowski, R. and Wolfe, D. (2007). {\em Lessons in Play: An Introduction to Combinatorial Game Theory.} CRC Press.

\bibitem{cgt}
Siegel, A. N. (2013). {\em Combinatorial Game Theory.} Graduate Studies in Mathematics (Vol. 146). American Mathematical Society.

\bibitem{Guy96} Guy, R. K. (1996). Unsolved Problems in Combinatorial Games. In R.~Nowakowki (Ed.), {\em Games of No Chance}. MSRI Book Series (vol. 29). Cambridge: Cambridge University Press..

\bibitem{althofer}
Althöfer, I. and Bültermann, J. (1995). Superlinear period lengths in some subtraction games. {\em Theoretical Computer Science}, 148(1), 111--119. doi:{\tt10.1016/0304-3975(95)00019-S}

\bibitem{fleischer}
Fleischer, R. and Trippen, G. (2006). Kayles on the way to the stars. In H.~Jaap~van~den~Herik et al. (Eds.), {\em Computers and Games. 4th International Conference, CG, 2004}. LNCS Book Series (vol. 3846). Berlin Heidelberg: Springer-Verlag. doi:{\tt10.1007/11674399\_16}

\bibitem{S78}  Schaeffer, T. J. (1978). On the complexity of some two-person perfect-information games. { \em J. Comput. System Sci.}, 16(2), 185--225. doi:{\tt10.1016/0022-0000(78)90045-4}

\bibitem{adams}
Adams, R., Dixon, J., Elder, J., Peabody, J., Vega, O. and Will, K. (2015). {\em Combinatorial Analysis of a Subtraction Game on Graphs}. Retrieved from arXiv:{\tt1507.05673}.

\bibitem{edgegeo}
Fraenkel, A. S., Scheinerman, E. R. and Ullman, D. (1993). Undirected edge Geography. {\em Theoretical Computer Science, 112}(2), 371-381.

\bibitem{ottaway1}
Nowakowski, R. J. and Ottaway, P. (2005). Vertex deletion games with parity rules. {\em Integers: Electronic Journal of Combinatorial Number Theory, 5}(2), A15.

\bibitem{ottaway2}
Harding, P. and Ottaway, P. (2014). Edge deletion games with parity rules. {\em Integers, 14}, G1.

\bibitem{picarete}
Blanc, L., Duchêne, E. and Gravier, S. (2006). A Deletion Game on Graphs: “Le Pic ar\^ete”. {\em Integers: Electronic Journal of Combinatorial Number Theory, 6}(G02), G02.

\bibitem{H15} Huggan, M. (2015). {\em Impartial Intersection Restriction Games} (master's thesis). Retrieved from Carleton University Research Virtual Environment (CURVE) (Record b3819898).

\bibitem{futurpapier}
Dailly, A., Moncel, J. and Parreau, A. (2018). {\em Connected Subtraction Games on Subdivided Stars}. Private communication.

\bibitem{Spra36}  Sprague, R. (1935). \"Uber mathematische Kampfspiele. {\em Tohoku Mathematical Journal, First Series}, 41, 438--444.

\bibitem{DGM}   Duch\^ene, \'E., Gravier, S. and Mhalla, M. (2016). {\em Scoring octal games on trees}. Unpublished.

\bibitem{halFullProof} Beaudou, L., Coupechoux, P., Dailly, A., Gravier, S., Moncel, J., Parreau, A. and Sopena, \'E. (2018). {\em Octal Games on Graphs: The 0.33 game on subdivided stars and bistars. Full Proof of Theorem 22}. HAL deposit: \url{https://hal.archives-ouvertes.fr/hal-01807116}.

\end{thebibliography}
\end{document}